\documentclass[12pt,reqno]{amsart}
\usepackage{amsthm}
\usepackage{amssymb}
\usepackage{amsmath}
\usepackage{mathrsfs}
\usepackage{amsfonts}
\usepackage{amssymb,amsmath}
 \usepackage[colorlinks, linkcolor=black, citecolor=blue,   hypertexnames=false]{hyperref}
 \usepackage[numbers,sort&compress]{natbib}
\newtheorem{Theorem}{Theorem}[section]
\newtheorem{Proposition}{Proposition}[section]
\newtheorem{Lemma}{Lemma}[section]
\newtheorem{Corollary}{Corollary}[section]
\newtheorem{Remark}{Remark}[section]

\newtheorem{Definition}{Definition}[section]
\numberwithin{equation}{section}

\def \no#1#2#3 {{\bf #1} (#3), #2.}
\def \eds#1#2#3 {#1, #2, #3.}

\title[Global dynamics of Schr\"odinger flows from Hyperbolic Planes]
{\bf{
On global dynamics of Schr\"odinger map flows on  hyperbolic planes near  harmonic maps}}

\author[Z. Li ]
{Ze Li}

\address{Ze Li
\newline\indent
School of Mathematics and Statistics, Ningbo University
\newline\indent
Ningbo, 315000, Zhejiang, P.R. China
}
\email{rikudosennin@163.com}

\keywords{Schr\"odinger map flow, hyperbolic planes, harmonic maps, stability}

\begin{document}

\begin{abstract}
The results of this paper are twofold: In the first part,   we prove that for Schr\"odinger map flows from hyperbolic planes  to  Riemannian surfaces with non-positive  sectional curvatures,  the harmonic maps which are  holomorphic or anti-holomorphic  of arbitrary size  are asymptotically stable. In the second part, we prove  that for Schr\"odinger map flows from hyperbolic planes  into  K\"ahler manifolds, the admissible harmonic maps of small size  are asymptotically stable.
The asymptotic stability results stated here contain two types: one is the convergence in $L^{\infty}_x$ as the previous works, the other is convergence to harmonic maps plus radiation terms in the energy space, which is new in literature of Schr\"odinger map flows without symmetry assumptions.
\end{abstract}
\maketitle

\section{Introduction}

Let $(\mathcal{M},h)$ be a Riemannian manifold and $(\mathcal{N},J,g)$ be a K\"ahler manifold, the Schr\"odinger map flow is a map $u:(x,t)\in\mathcal{M}\times\Bbb R\longmapsto {\mathcal N}$ which satisfies
\begin{align}\label{1}
\begin{cases}
u_t =J(u)h^{jk}{\widetilde{\nabla}}_j\partial_k u \\
u\upharpoonright_{t=0} = u_0(x),
\end{cases}
\end{align}
where ${\widetilde{\nabla}}$  denotes the pullback covariant derivative on $u^*T  \mathcal{N}$.
When $\mathcal{N}$ is the 2 dimensional sphere, (\ref{1}) plays a fundamental role in solid-state physics  and is usually referred as the Landau-Lifshitz equation (LL) or the continuous isotropic Heisenberg spin model in physics literature.  LL describes the dynamics of the magnetization field inside ferromagnetic material (Landau-Lifshitz \cite{LL}), and  its various forms are also related to many other problems such as vortex motions \cite{R}, gauge theories, motions of membranes. The general Schr\"odinger map flow with K\"ahler targets is  a natural geometric generalization of the Landau-Lifshitz equation (see \cite{UT}), and was first studied by geometricians at late 1990s.  In  this paper, we consider the case where $\mathcal{M}$ is the hyperbolic plane and $\mathcal{N}$ is a Riemannian surface or more generally K\"ahler manifold.

The  Schr\"odinger map flow (SL) on Euclidean spaces has been intensively studied. The local   well-posedness and small data global well-posedness theory of SL on $\Bbb R^d$ were developed by  \cite{SSB,DW,Mc,CSU,B,IK,BIK,BIKT1,RS,Li4,Li5,NSZ,Smith2}. The dynamical behaviors of SL on $\Bbb R^d$ near harmonic maps were  studied  in the equivariant case.  For the equivariant SL from $\Bbb R^2$ into $\Bbb S^2$ with energy below the ground state and equivariant flows from $\Bbb R^2$ into $\Bbb H^2$ with initial data of finite energy, the global well-posedness and scattering in the gauge sense were proved by Bejenaru-Ionescu-Kenig-Tataru \cite{BIKT2, BIKT3}. For the m-equivariant SL from $\Bbb R^2$ into $\Bbb S^2$ with initial data near the harmonic map, Gustafson-Kang-Tsai \cite{GKT2} proved asymptotic stability for $m\ge4$. And later  Gustafson-Nakanishi-Tsai \cite{GNT} proved  that 3-equivariant SL is also in the stable regime and  the 2-equivariant (dissipative) LL has winding oscillatory solutions. For the 1-equivariant SL, Bejenaru-Tataru proved the harmonic map is unstable in the energy space and stable in a stronger topology. Merle-Raphael-Rodnianski \cite{MPR} and Perelman \cite{P} built type II blow up solutions near the harmonic maps for 1-equivariant SL.

For the curved background manifolds, the sequel pioneering  works of Lawrie, Oh, Shahshahani~\cite{LOS,LOS2,LOS3,LOS4,LOS5} studied global dynamics of wave maps on hyperbolic spaces in the equivariant case.
And the first  non-equivariant data result on hyperbolic spaces was obtained by Lawrie-Oh-Shahshahani  \cite{LOS} where  optimal small data global theory for high dimensional  wave maps on hyperbolic spaces was established.
More recently,  Lawrie-Luhrmann-Oh-Shahshahani \cite{LLOS1,LLOS2} studied global dynamics of SL under the non-equivariant perturbations of strongly linear stable equivariant harmonic maps from $\Bbb H^2$ to  rationally symmetric Riemannian surfaces. Their works exploit  the delicate smoothing effects of   Schr\"odinger equations.
On the other hand, for wave maps on small perturbations of Euclidean spaces, Lawrie \cite{La} studied small data global theory for high  dimensions, and  Gavrus-Jao-Tataru \cite{GJT} established optimal local well-posedness in the energy critical case. We also mention the works of \cite{Li1,Li2,Li3} where we proved asymptotic stability of harmonic maps under the wave map between hyperbolic planes. And for  wave maps on product spaces of spheres and Euclidean spaces, Shatah, Tahvildar-Zadeh \cite{ST} and   Shahshahani \cite{Sh} studied orbital stability of stationary solutions.

In this paper, we study stability of non-equivariant harmonic maps under the Schr\"odinger map flow from $\Bbb H^2$ to Riemannian surfaces or general K\"ahler manifolds. The main result is the stability of holomorphic or anti-holomorphic maps from $\Bbb H^2$ to  Riemannian surfaces.

In the following, all the maps studied are assumed or proved to  have  compact images in $\mathcal{N}$. So for convenience, given the unperturbed harmonic map $Q:\mathcal{M}\to{\mathcal{N}}$ with compact image, {\bf we fix $\widetilde{\mathcal{N}}$ to be an  open sub-manifold of $\mathcal{N}$ containing $Q (\mathcal{M}) $. Moreover, let $\mathcal{P}:\widetilde{\mathcal{N}}\to \Bbb R^N$  be one isometric embedding of $\widetilde{\mathcal{N}}$ into Euclidean space  $\Bbb R^N$.}

\begin{Theorem}\label{A1}
Let $\mathcal{M}=\Bbb H^2$, and $\mathcal{N}$ be a Riemannian surface with non-positive sectional curvatures. Assume that  $Q:\mathcal{M}\to \mathcal{N}$ is either a holomophic or an anti-holomorphic map such that the image $Q(\mathcal{M})$ is contained in a compact subset of $\mathcal{N}$. Given any $\delta>0$, there exists a sufficiently small  constant $\epsilon_*>0$ such that if $u_0\in \mathcal{H}^3_{Q}$ satisfies
\begin{align}
\|u_0-Q\|_{H^{2+2\delta}}\le \epsilon_*,
\end{align}
then the Schr\"odinger map flow with initial data $u_0$ evolves  into a global solution and converges to $Q$ as $t\to \infty$ in the following two senses:
\begin{itemize}
\item On one side, we have
\begin{align}\label{aP11}
\lim_{t\to\infty}\|u(t)-Q\|_{L^{\infty}_x}=0;
\end{align}
\item   On the other side, there exist functions  $f^1_+,f^2_+:\mathcal{M}\to \Bbb C^{N} $  belonging to  $H^1$  such that
\begin{align}
&\lim_{t\to\infty}\| u- Q- {\rm Re} (e^{it\Delta}f^1_+)- {\rm Im}(e^{it\Delta}f^2_+)\|_{H^1_x}=0, \label{aP12}
\end{align}
where we view $u$ and $Q$ as maps into $\Bbb R^{N}$.
\end{itemize}
\end{Theorem}

\begin{Remark}
Results as (\ref{aP11}) were first established by Tao \cite{Tao4} for the wave map equation on $\Bbb R^2$, and later obtained by \cite{LOS,LOS2,LLOS2,Li1,Li3} in the setting of wave maps/ Schr\"odinger map flows on hyperbolic spaces.  The type result  (\ref{aP12}) is new in the setting of non-equivariant Schr\"odinger map flows on  both Euclidean spaces and curved base manifolds.
\end{Remark}

\begin{Remark}
The same arguments of Theorem 1.1 also refine our previous results on stability of harmonic maps between $\Bbb H^2$ under the wave map evolutions. In fact, we can as well prove that given an admissible harmonic map $Q$ from $\Bbb H^2$ to non-positively curved Riemannian surfaces, the solution to the wave map equation  with initial data $(u_0,u_1)$ satisfying $\|u_0-Q\|_{H^2}+\|u_1\|_{H^1}\ll 1$, evolves to a global solution and scatters to $Q$ in the energy space, i.e., there exist  some functions $(v_0,v_1):\Bbb H^2\to \Bbb R^{N}\times \Bbb R^N$ belonging to $H^1\times L^2$ such that
\begin{align}
&\lim_{t\to\infty}\| (u,\partial_t u)-(Q,0)-S(t)(v_0,v_1)\|_{H^1_x\times L^2_x}=0, \label{asdP12}
\end{align}
where $S(t)(v_0,v_1)$ denotes the linear solution of wave equations on $\Bbb H^2$:
\begin{align*}
\left\{
  \begin{array}{ll}
    \partial^2_tv-\Delta v=0, & \hbox{ } \\
    (v,\partial_tv)|_{t=0}=(v_0,v_1). & \hbox{ }
  \end{array}
\right.
\end{align*}
\end{Remark}

\begin{Remark}
Holomorphic maps and   anti-holomorphic maps  are typical harmonic maps between K\"ahler manifolds. A  remarkable observation of \cite{LLOS2} reveals that linearized operators around  holomorphic maps between Riemannian surfaces are self-adjoint. We focus on holomorphic maps and  anti-holomorphic maps to take this convenience for large data. In several cases, one can show harmonic maps are either holomorphic or anti-holomorphic, see Siu-Yau \cite{SY}, Siu \cite{siu} and references therein. One may see Appendix B for more background materials for holomorphic maps and  anti-holomorphic maps.
\end{Remark}

\begin{Remark}
The perturbation regularity we assume is $H^{2+2\delta}$. In the work \cite{LLOS2}, for equivariant harmonic maps $Q$ (which are especially holomorphic)  and rotationally symmetric Riemannian surfaces $\mathcal{N}$,  the stability was proved for non-equivariant initial data $u_0$ satisfying $\|u_0-Q\|_{H^{1+\delta}}+\|\mathcal{D}_{\Omega} u_0\|_{H^{1+2\delta}}\ll 1$, where $\mathcal{D}_{\Omega}$ can be viewed as a derivative measuring how co-rotational $u_0$ is.  We see the topology of perturbations assumed here is almost the same   as that used in  \cite{LLOS2} in the angle direction and stronger in the radial direction. Moreover, [Remark 1.8,\cite{LLOS2}] pointed out that  non-equivariant harmonic maps $Q$ seem to require new ideas. Theorem \ref{A1} here covers all  holomorphic maps and  anti-holomorphic maps of compact images, which contain wide class of non-equivariant harmonic maps.
\end{Remark}

Similar arguments as  proof of Theorem  \ref{A1} also yield  stability of small size  harmonic maps into K\"ahler manifolds:
\begin{Theorem} \label{Th}
Let $\mathcal{M}=\Bbb H^2$, and $\mathcal{N}$ be a compact $2n$-dimensional K\"ahler manifold.
Let $Q:\Bbb H^2\to \mathcal{N}$ be an admissible  harmonic map with compact image.  Given any $\delta>0$, if  the initial data $u_0\in \mathcal{H}^{3}_{Q}$ to
(\ref{1}) satisfies
\begin{align}\label{as3}
\|u_0-Q\|_{H^{2+2\delta}}<\mu,
\end{align}
with $\mu>0$ being sufficiently small, then (\ref{1}) has a global solution $u(t)$ and as $t\to\infty$ we have
\begin{align*}
 &\mathop {\lim }\limits_{t\to\infty }  \| u(t,x)- Q(x)\|_{L^{\infty}_x} = 0\\
 &\mathop {\lim }\limits_{t\to\infty }  \| u(t)- Q -\sum^{n}_{j=1}{\rm Re}(e^{it\Delta} f^j_+)-\sum^{n}_{j=1}{\rm Im}(e^{it\Delta} g^j_+)\|_{H^{1}_x} = 0
\end{align*}
for some functions $f^1_+,...,f^n_+,g^1_+,...,g^n_+:\mathcal{M}\to \Bbb C^n$ belonging to $H^1_x$.
\end{Theorem}

\begin{Remark}
We call $Q:\Bbb \mathcal{M}\to {\mathcal{N}}$ an admissible harmonic map, if $\|\nabla^j Q\|_{L^{2}_x\cap L^{\infty}_x}\lesssim 1$, for $j=1,2,3$, and  $\|e^{r}dQ\|_{L^{\infty}_x}\lesssim 1$ with $r$ being the radial variable in the polar coordinates of $\Bbb H^2$. Any holomorphic map or anti-holomorphic map from $\Bbb H^2$ to K\"ahler targets with compact image is an admissible harmonic map.
\end{Remark}

\begin{Remark} It is interesting  that the dynamics for flows defined on Euclidean spaces and curved spaces are typically different and of independent interest. This in the setting of dispersive geometric flows  was first noticed by  \cite{LOS2}.   One example is the difference of solutions to the stationary problem. For instance, Lawrie-Oh-Shahshahani \cite{LOS2} showed there exists harmonic map $Q_{\lambda}$ from $\Bbb H^2$ to $\Bbb S^2$ and $\Bbb H^2$ with energy $\lambda$ for any given $\lambda>0$. However, there exists no non-trivial finite energy harmonic map from $\Bbb R^2$ to $\Bbb H^2$ and the minimal energy to support a non-trivial harmonic map from $\Bbb R^2$ to $\Bbb S^2$ is $2\pi$. The difference on stationary solutions directly leads to distinct long time dynamics.  For instance, the energy critical wave maps from $\Bbb R^2$ to $\Bbb H^2$ are proved to scattering to free waves in energy space, while wave maps from $\Bbb H^2$ to $\Bbb H^2$ are believed to scattering to one stationary solution in energy space.
\end{Remark}

\subsection{Outline of proof}

The whole proof is set up by the previous linearization strategy based on Tao's caloric gauge. This  linearization  scheme gives rises to a nonlinear Schr\"odinger equation governing the evolution of heat tension field $\phi_s$ and a nonlinear heat equation of the Schr\"odinger map tension field  $Z$. In the holomorphic or anti-holomorphic setting, the resulting linearized operator is self-adjoint. The difficulty is to compensate the derivative loss caused by the magnetic term $A\cdot \nabla\phi_s$.

In this work, to compensate the derivative loss, we essentially use three key estimates: inhomogeneous  Morawetz  estimates for time dependent magnetic Schr\"odinger operators; Strichartz estimates for linearized operator ${\bf H}$;
energy estimates. They are derived for different linear  evolution equations and play different roles, see Section  \ref{gvn} for more detailed expositions.

\subsection{Main ideas of Theorem 1.1 and main linear estimates }  \label{gvn}

The Morawetz estimate we adopt reads as follows: Suppose that $A=A_jdx^j$ is a real valued one form on $\mathcal{M}$, and denote $\Delta_{A}$, $D_{A}$ to be
\begin{align*}
\Delta_{A}f=(\nabla_j+iA_j)h^{kj}(\nabla_k+iA_k)f,\mbox{  }D_{A}f=\partial_j fdx^j+iA_jfdx^j.
\end{align*}
Let $u$ solve $i\partial_t u+\Delta_{A}u=F$, then there holds
\begin{align*}
\|e^{-\frac{r}{2}}\nabla u\|^2_{L^2_{t.x}}&\lesssim \|(-\Delta )^{\frac{1}{4}}u \|^2_{L^{\infty}_tL^2_x}+ \||\partial_t A||u|^2\|_{L^{1}_{t,x}}+\||A||u|^2\|_{L^1_{t,x}}+\|\nabla A||u||\nabla u|\|_{L^{1}_{t,x}}
\\
&+\|e^{-r}|A|^2|u|^2\|_{L^{1}_{t,x}}+\||D_{A}u ||F|\|_{L^1_{t.x}}+\||u ||F|\|_{L^1_{t.x}}.
\end{align*}
This behaves well in the sense that no weight nor  derivative  is imposed  to the inhomogeneous term $F$. In application, $u$ is just $\phi_s$ and $A$ is the connection one form under the caloric gauge.  Moreover, we remark that the derivative loss now indeed hides in $\partial_t A$, which is easier to handle than $\phi_s$, since the heat flow provides much better regularity for $A$ than $\phi_s$ itself.

The energy estimate is simple but useful:
Let $u$ solve $i\partial_t u+\Delta_{A}u=F$, then there holds
\begin{align*}
\|\nabla u\|^2_{L^{\infty}_tL^2_x}&\lesssim \|D_{A}u(0) \|^2_{L^2_x}+|\partial_tA||D_{A}u|^2\|_{L^{1}_{t,x}}+\||D_{A}F||D_{A}u|\|_{L^1_{t,x}}+\||Au||\nabla u|\|_{L^{1}_{t,x}}.
\end{align*}
As the
Morawetz estimates, the derivative loss is also hidden in $\partial_t A$.

The third is Strichartz estimates of linearized operator  ${\bf H}$, which enable us to control quadratic nonlinear terms. In fact, we have
\begin{align*}
\|\phi_s\|_{L^{2}_tL^q_x}&\lesssim \||\tilde{A}||\nabla \phi_s| \|_{L^{1}_{t}L^2_{x}}+{\rm lower \mbox{ }derivative \mbox{ }terms},
\end{align*}
where $\tilde{A}$ denotes the connection one form under the caloric gauge removing away the limit part, see Section 2.

We aim to use Morawetz estimates and energy estimates to control the above magnetic term $\tilde{A}\cdot\nabla \phi_s$ in Strichartz estimates. Meanwhile, we expect to use
Strichartz estimates to control the zero derivative terms in  Morawetz estimates and energy estimates.  In fact,  in the right hand side of both  Morawetz estimates and energy estimates, there are quadratic terms which are of the same order as the left hand side. This is something troublesome  for closing bootstrap arguments. However, one observes that all these terms contain at least one zero derivative term or two half derivative terms, which inspires us to set up the bootstrap assumption  as (\ref{om1})-(\ref{om3}). The two  key points for this set up is (i) the smallness size of  $\|\nabla \phi_s\|_{L^{\infty}_tL^2_x}$, $\| e^{-\frac{1}{2}r} \phi_s\|_{L^{2}_tL^2_x}$, $\|  \phi_s\|_{L^{2}_tL^q_x}$ shall be decreasing in order; (ii) the regularity loss measured by the power of $s$ shall be the same for these  three quantities. (see  (\ref{om1})-(\ref{om3}))
The point (i) helps to control  quadratic terms in  Morawetz estimates and energy estimates, and the point (ii) enables us to close bootstrap via using Morawetz estimates and energy estimates to dominate magnetic terms in Strichartz estimates.

The type result of (\ref{aP12}) for dispersive geometric flows with non-equivariant data  only appeared in the wave map setting with $Q$ being a fixed point.
The result itself is of independent interest, especially it coincides with the soliton resolution conjecture for critical dispersive PDEs, which claims solutions with bounded trajectories would decouple into either separated solitons and radiations or separated solitons and some weak limit solution.
The proof of (\ref{aP12}) consists of two parts, one is to derive Strichartz estimates for  higher order derivatives of $\phi_s$, the other is to  establish the linear scattering theory of the linearized magnetic Schr\"odigner
operators.

{\bf Notations}
We will use the notation $a\lesssim b$ whenever there exists some positive constant $C$ so that $a\le C b$. Similarly, we will use $a\sim b$ if $a\lesssim b \lesssim a$.
For a linear operator $T$ from Banach space $X$ to Banach space $Y$, we denote its operator norm by $\|T\|_{\mathcal{L}(X\to Y)}$.
All the constants are denoted by $C$ and they can change from line to line.

Let $\mathcal{Z}$ be a manifold with connection $D$. Given a frame $\{{e_{a}}\}$ for $T\mathcal{Z}$ and corresponding dual frame $\{\xi_{a}\}$ for  $T^*\mathcal{Z}$,  for
an arbitrary $(r,s)$-type tensor ${\Bbb T}$, we write
\begin{align*}
D_{c_1}... D_{c_k}{\Bbb T}^{a_1,...,a_r}_{b_1,...,b_s}=(DD ...D{\Bbb T})(e_{c_1},...,e_{c_k}; \xi_{a_1},...,\xi_{a_{r}},e_{b_1},...,e_{b_{s}}).
\end{align*}

Let $\Bbb R^{1+2}$ be the $(1+2)$-dimensional Minkowski space equipped with metric ${\bf m}:=-dy^0dy^0+dy^1dy^1+dy^2dy^2$.
The hyperbolic plane denoted by $\Bbb H^2$ is defined by
$$
\Bbb H^2:=\{y\in\Bbb R^{1+2}:(y^0)^2-(y^1)^2-(y^2)^2=1, y^0>0\},
$$
equipped with the pullback Riemannian  metric $h:=\iota^*{\bf m}$, where $\iota:\Bbb H^2\to \Bbb R^{1+2}$ is the inclusion map.

Let $(\mathcal{M},h)$ be a d-dimensional Riemannian manifold.
The Riemannian curvature tensor on $(\mathcal{M},h)$  denoted by $R$ is defined by
$$
R(X,Y)Z=\nabla_{X}\nabla_{Y}Z-\nabla_{Y}\nabla_{X}Z-\nabla_{[X,Y]}Z, \mbox{ }X,Y,Z\in T\mathcal{M}.
$$
and we also denote
$$
R(X,Y,Z,W) =h(R(Z,W)X,Y), \mbox{ }X,Y,Z\in T\mathcal{M}.
$$
For local coordinates $(x^1,...,x^d)$ for $\mathcal{M}$, the curvature tensor components are defined by
\begin{align*}
 R_{ijkl} :=h(R(\frac{\partial}{\partial x^{k}}, \frac{\partial}{\partial x^{l}})\frac{\partial}{\partial x^{i}},\frac{\partial}{\partial x^{j}}),
 \mbox{  }\mbox{  }R(\frac{\partial}{\partial x^{i}}, \frac{\partial}{\partial x^{j}})\frac{\partial}{\partial x^{k}} =R^{l}_{kij}\frac{\partial}{\partial x^{l}}.
\end{align*}
The Ricci tensor is defined by $R_{ij}dx^idx^j$ with
\begin{align*}
R_{ij}=h^{kl}R_{iklj}=R^{k}_{ikj}.
 \end{align*}
And the sectional curvature is defined via
\begin{align*}
K(X,Y):=-\frac{R(X,Y,X,Y)}{\|X\wedge Y\|^2}, \mbox{ }\|X\wedge Y\|^2:=h(X,X)h(Y,Y)-h(X,y)^2,\mbox{ }X,Y \in T\mathcal{M}.
 \end{align*}
We recall the skew-symmetry, symmetry, and first Bianchi identity of  Riemannian curvature tensor  as follows:
\begin{align*}
R_{ijkl}&=-R_{jikl}=-R_{ijlk}\\
R_{ijkl}&=R_{klij}\\
R_{ijkl}&+R_{ljik}+R_{kjli}=0.
 \end{align*}

The Riemannian curvature tensor on the target manifold $\mathcal{N}$ will be denoted by ${\bf R}$.
The induced covariant derivative on $u^*T\mathcal{N}$ is denoted by $\widetilde{\nabla}$.

\section{Preliminaries on Sobolev inequalities and Caloric Gauges}

\subsection{Sobolev inequalities}

In this part, we collect some preliminaries on function spaces and Sobolev inequalities on hyperbolic spaces.

Given a tensor $\Bbb T$ defined on $\Bbb H^2$, for $k\in \Bbb N$, $p\in[1,\infty]$, the   $W^{k,p}$ norm is defined by
\begin{align*}
\|{\Bbb T}\|_{W^{k,p}}:=(\sum^{k}_{j=0}\int_{\Bbb H^2}|\nabla^j{ \Bbb {T}}|^p{\rm dvol_h})^{\frac{1}{p}},
\end{align*}
with standard modifications when $p=\infty$.

For $1<p<\infty$, one has the equivalence relation:
\begin{align}
 \| f \|_{W^{k,p}} \sim \|(-\Delta)^{\frac{k}{2}}f\|_{L^p}.\label{equi}
\end{align}
If $\gamma\ge 0$ is not an integer, we define the norm $W^{\gamma,p}$ with $p\in (1,\infty)$ to be
\begin{align*}
 \| f \|_{W^{\gamma,p}} := \|(-\Delta)^{\frac {\gamma}{2}}f\|_{L^p}.
\end{align*}

\begin{Lemma}[\cite{AP,LOS,LLOS2}]
Let $f\in C^{\infty}_c(\mathbb{H}^2;\Bbb R)$. Then for $1\le p\le q\le \infty$, $0<\theta<1$, $\frac{1}{p} - \frac{\theta }{2} = \frac{1}{q}$, the Gagliardo-Nirenberg inequality is
\begin{align}
 {\left\| f \right\|_{{L^q}}} \lesssim \left\| {\nabla f} \right\|_{{L^2}}^\theta \left\| f \right\|_{{L^p}}^{1 - \theta }. \label{GN}
\end{align}
The Poincare  inequality is
\begin{align}
{\left\| f \right\|_{{L^p}}} \lesssim {\left\| {\nabla f} \right\|_{{L^p}}}, \mbox{ }  1<p<\infty.  \label{Poincare}
\end{align}
The $L^p$ interpolation  inequality holds as
\begin{align*}
 &{\left\| (-\Delta)^{\gamma}f \right\|_{{L^p}}} \lesssim \|f\|^{1-\frac{\gamma}{\alpha}}_{L^p}\|(- \Delta)^{\alpha} f \|^{\frac{\gamma}{\alpha}}_{L^p}, \mbox{  }p\in (1,\infty), 0\le \gamma\le \alpha.\\
  & {\left\|  f \right\|_{{L^q}}} \lesssim \|f\|^{1-\theta}_{L^p}\|(- \Delta)^{\alpha} f \|^{\theta}_{L^p}, \mbox{  }p\in (1,\infty), p\le q\le \infty,0<\theta= \frac{2}{\alpha}(\frac{1}{p}-\frac{1}{q})<1.
\end{align*}
The Riesz transform is bounded in $L^p$ for $1<p<\infty,$ i.e.,
\begin{align}
{\left\| {\nabla f} \right\|_{{L^p}}} \sim{\| {{{\left( { - \Delta } \right)}^{\frac{1}{2}}}f} \|_{{L^p}}} \label{Riesz}.
\end{align}
The Sobolev product rule is
\begin{align*}
 \| fg \|_{{H^{\gamma}}}\lesssim \|f\|_{L^{\infty}}\|g\|_{H^{\gamma}}+\|g\|_{L^{\infty}}\|f\|_{H^{\gamma}},\mbox{  }\gamma\ge 0.
\end{align*}

And we recall the general Sobolev inequality: Let $1<p,q<\infty$ and $\sigma_1,\sigma_2\in\Bbb R$  such that $\sigma_1-\sigma_2\ge 2/p-2/q\ge0$. Then for all $f\in C^{\infty}_c(\Bbb H^2;\Bbb R)$
\begin{align*}
\|(-\Delta)^{\sigma_2}f\|_{L^q}\lesssim \|(-\Delta)^{\sigma_1}f\|_{L^p}.
\end{align*}
The diamagnetic inequality known also as Kato's inequality is as follows (see \cite{LOS}): If ${\Bbb T}$ is a tension filed defined on $\Bbb H^2$, then in the distribution sense it holds that
\begin{align}\label{y6frtsr45}
|\nabla|{\Bbb T}||\le |\nabla {\Bbb T}|.
\end{align}
\end{Lemma}

Recall that we always assume that $\tilde{\mathcal{N}}$ is an open sub-manifold of $\mathcal{N}$ which is isometrically embedded into $\Bbb R^N$.
For $\gamma\ge 0$, define $\mathcal{H}^{\gamma}_Q$ to be
\begin{align*}
 \mathcal{H}^{\gamma}_Q:=\{u:\mathcal{M}\to \Bbb R^N: \|u-Q\|_{H^{\gamma}}<\infty, u(x)\in \tilde{\mathcal{N}} a.e.\}
\end{align*}

For initial data $u_0\in  \mathcal{H}^3_Q$, we have the local well-posedness.
\begin{Lemma}\label{local}
Let $k\ge 3$ be an integer. If $u_0\in  \mathcal{H}^k_Q$, then there exists a positive constant $T_0>0$ depending only on $\|u_0\|_{\mathcal{H}^k_Q}$ such that SL with initial data $u_0$ has a unique local solution $u\in C([0,T_0];\mathcal{H}^k_{Q})$.
\end{Lemma}
\begin{Remark}
McGahagan \cite{Mc}  introduced an approximate scheme for SL via considering a  generalized wave map equation.
We remark that in the flat case $\mathcal{M}={\Bbb  R}^d$, $ 1\le d\le 3$, Theorem 7.1 of Shatah, Struwe
\cite{ss}  gave a local theory for Cauchy problem of wave maps  in $H^2 \times H^1$ by an energy method. The same energy arguments also yield a local well-posedness theory for the generalized wave map equation on $\Bbb H^2$ introduced above.  Then our lemma follows by proving uniform Sobolev norms for approximate solutions.
\end{Remark}

\subsection{Moving frames}

Let $\Bbb I$ be $[0,T]$ for some $T>0$.
Let $\mathcal{N}$ be a 2n-dimensional K\"ahler manifold. In the following, we make the convention that Greek indexes run in $\{1,...,n\}$, and denote $\bar{\gamma}=\gamma+n$.

Since $\Bbb I\times \mathcal{M}$ with $\mathcal{M}=\Bbb H^2$ is contractible, there must exist global orthonormal frames for $u^*(T\mathcal{N})$. Using the complex structure one can assume the orthonormal frames are of the form
\begin{align*}
E:=\{e_1(t,x),Je_1(t,x),....,e_n(t,x),Je_n(t,x)\}.
\end{align*}
Let $\psi_j=(\psi^1_j,\psi^{\bar{1}}_j,...,\psi^n_j,\psi^{\bar{n}}_j)$ for $j=0,1,2$ be the components of $\partial_{t,x}u$ in the frame $E$:
\begin{align}\label{Hyu798}
\psi_j^{\gamma} = \left\langle {\partial _ju,{e_{\gamma}}} \right\rangle ,\psi^{\bar{\gamma}}_j = \left\langle {{\partial_j}u,J{e_{\gamma}}} \right\rangle, \mbox{ }\gamma=1,...,n.
\end{align}
We always use $0$ to represent $t$ in index. The isomorphism of $\Bbb R^{2n}$ to ${\Bbb C}^n$ induces a ${\Bbb C}^n$-valued function $\phi_j$ defined by $\phi^{\gamma}_j=\psi^{\gamma}_j+i\psi^{\bar{\gamma}}_j$ with $\gamma=1,...,n$. Conversely, given a function $\phi:\Bbb I\times \mathcal{M}\to {\Bbb C}^n$, we associate it with a section $\phi e$  of the bundle $u^*(T\mathcal{N})$ via
\begin{align}
\phi\longmapsto \phi e := {\rm Re} (\phi^{\gamma})e_{\gamma}+{\rm Im}(\phi^{\gamma})Je_{\gamma},
\end{align}
where $(\phi^1,...,\phi^n)$ denotes the components of $\phi$.

Then $u$ induces a covariant derivative on the trivial complex vector bundle over base manifold $\Bbb I\times\Bbb H^2$ with fiber $\Bbb C^n$ defined by
\begin{align*}
{\mathbf{D}}_j =\nabla_j  +A_{j},
\end{align*}
where $\nabla$ denotes the trivial covariant derivative on $\Bbb I\times \mathcal{M}$, and  $A_j$ takes values in $\Bbb C^n\times \Bbb C^n$  given by
\begin{align*}
(A_j)_{\beta,\alpha}&:= [A_j]^\beta_\alpha+i[A_j]^{\bar{\beta}}_{\alpha} \\
[A_j]^l_m&:= \left\langle \widetilde{\nabla}_je_l,{e_m} \right\rangle.
\end{align*}
For example, if $\varphi=(\varphi^1,...,\varphi^n)$ is a $\Bbb C^n$ valued function on $\Bbb I\times \mathcal{M}$, then
\begin{align*}
\mathbf{D}_j\varphi^\beta=\partial_j \varphi^\beta+\sum^{n}_{\alpha=1}\left([A_j]^\beta_\alpha+ i[A_j]^{\bar{\beta}}_{\alpha}\right)\varphi^\alpha;
\end{align*}
If $\varphi=\varphi_kdx^k$ is a $\Bbb C^n$ valued 1-form on $\Bbb I\times \mathcal{M}$, then $\mathbf{D}_j\varphi_{k}$ takes values in $\Bbb C^n$ whose $\beta$-component  is
\begin{align*}
\partial_j \varphi^{\beta}_k-\Gamma^{l}_{jk}\phi^{\beta}_{l}+\sum^{n}_{\alpha=1}\left([A_j]^\beta_\alpha+i[A_j]^{\bar{\beta}}_{\alpha}\right)\varphi^\alpha_{k}.
\end{align*}

We view $\phi$ defined in (\ref{Hyu798}) as a 1-form on $\Bbb I\times \mathcal{M}$:
\begin{align*}
\phi=\phi_tdt+\phi_jdx^j,
\end{align*}
and $A:=A_t dt+A_jdx^j$  as a $\Bbb C^{n}\times\Bbb C^n $  valued  1-form on $\Bbb I\times \mathcal{M}$.
The torsion free identity and the commutator identity hold as follows:
\begin{align*}
\mathbf{D}_k\phi_j&=\mathbf{D}_j\phi_k \\
[\mathbf{D}_k ,\mathbf{D}_j]\varphi&=\left(\nabla_k  A_j-\nabla_j A_k+[ A_k, A_j]\right)\varphi\Longleftrightarrow \mathbf{R}(\partial_ku, \partial_j u)(\varphi e).
\end{align*}
Schematically, we write
$$[\mathbf{D}_k,\mathbf{D}_j]=\mathcal{R}(\phi_k,\phi_j).
$$
With the notations given above, (\ref{1}) can be written as
\begin{align}\label{jnk}
\phi_t=i\sum^{2}_{j=1}\mathbf{D}^j\phi_j.
\end{align}

\subsection{Caloric gauge}

The caloric gauge was first introduced by Tao \cite{Tao3,Tao4} in the wave map setting. Later it was applied to various geometric PDEs, see \cite{BIKT1,Oh,LOS,OT1,OT2,OT3,Smith1} for instance.
Roughly speaking, Tao's caloric gauge in our setting is to parallel transpose the frame fixed on the limit harmonic map along the heat flow to the original map $u(t):\mathcal{M}\to \mathcal{N}$.
The heat flow  from  $\Bbb H^2$ to $\mathcal{N}$ with initial data near reasonable harmonic maps $Q$ evolves to  a global solution  and converges to the unperturbed  harmonic map $Q$ as time goes to infinity, see Lemma \ref{heat}.
Here, Tao's caloric gauge can be defined as follows.
\begin{Definition}\label{cg}
Assume $Q$ is an admissible harmonic map. Let  $u(t,x):[0,T]\times \mathcal{M}\to \mathcal{N}$ be a solution of (\ref{1}) in $C([0,T];\mathcal{H}^2_Q)$. For a given orthonormal frame $E^{\infty}:=\{e^{\infty}_{\alpha},Je^{\infty}_{\alpha}\}^n_{\alpha=1}$ for  $Q^*T\mathcal{N}$, a caloric gauge is a tuple consisting of a map  $v:\Bbb R^+\times [0,T]\times \mathcal{M}\to \mathcal{N}$ and an orthonormal frame $\mathbf{E}(v(s,t,x)):=\{\mathbf{e}_1,J\mathbf{e}_2,...,\mathbf{e}_n,J\mathbf{e}_n\}$ such that
\begin{align}\label{aq1}
\left\{ \begin{array}{l}
{\partial _s}v= \tau (v) \\
v(0,t,x)= u(t,x) \\
\end{array} \right.
\end{align}
and
\begin{align}\label{aq2}
\left\{ \begin{array}{l}
{\widetilde{\nabla} _s}{\mathbf{e} _k} = 0,\mbox{ } k=1,...,n.\\
\mathop {\lim }\limits_{s \to \infty } {\mathbf{e}_k} = {e^{\infty}_k} \\
\end{array} \right.
\end{align}
\end{Definition}

Let $u:[0,T]\times \Bbb H^2\to \mathcal{N}$ be a smooth map such that $u\in C([0,T];\mathcal{H}^2_{Q})$.
Then  given $t\in\Bbb [0,T]$, our previous works \cite{Li3,Li1} show the heat flow equation (\ref{aq1}) with initial data $u(t)$  has a global solution $v(s,t,x)$ and converges to $Q$ as $s\to \infty$.
Let $\psi_s=(\psi^1_s,\psi^{\bar{1}}_s,...,\psi^n_s,\psi^{\bar{n}}_s)$   be the components of $\partial_{s}v$ in the frame $\mathbf{E}$:
\begin{align*}
\psi_s^{\gamma} = \left\langle {\partial _sv(s,t),{\mathbf{e}_\gamma}} \right\rangle ,\psi^{\bar{\gamma}}_s = \left\langle {{\partial_s}v,J{\mathbf{e}_\gamma}} \right\rangle, \mbox{ }\gamma=1,...,n.
\end{align*}
Define the corresponding $\Bbb C^n$ valued function  $\phi_s$ by $\phi^{\alpha}_s=\psi^{\alpha}_s+{i}\psi^{\bar{\alpha}}_s$, $\alpha=1,...,n$.
$\phi_s$ is usually called the heat tension field.

The caloric gauge in Definition \ref{cg} does exist and we can decompose the connection coefficients according to the gauge.
\begin{Lemma}\label{3.3}
Given any solution $u$ of (\ref{1}) in $C([0,T];\mathcal{H}^2_Q)$. For any fixed frame $E^{\infty}:=\{e^{\infty}_{\alpha},Je^{\infty}_{\alpha}\}^n_{\alpha=1}$ for $Q^*T\mathcal{N}$,
there exists a unique corresponding caloric gauge defined in Definition \ref{cg}.
Moreover, denote ${\bf E}=\{{\bf e}_{\alpha},J{\bf e}_{\alpha}\}^{n}_{{\alpha}=1}$ the caloric gauge, then we have for $j=1,2$
\begin{align*}
&\mathop {\lim }\limits_{s \to \infty } [{A_j}]^l_k(s,t,x) =\left\langle {{\bar{\nabla} _j}{e^{\infty}_k }(x),{e^{\infty}_l}(x)} \right\rangle \\
&\mathop {\lim }\limits_{s \to \infty } {A_t}(s,t,x) = 0,
\end{align*}
where $\bar{\nabla}$ denotes the induced connection on $Q^*T\mathcal{N}$.
Particularly, denoting ${{A}}^{\infty}_i$ the limit coefficient matrix, i.e.,
\begin{align*}
(A^{\infty}_j)_{\beta,\alpha}&:= [A_j]^\beta_\alpha+i[A_j]^{\bar{\beta}}_{\alpha}, \mbox{ }\mbox{ }[{A}^{\infty}_j]^k_l:=\left\langle {{\bar{\nabla} _j}{e^{\infty}_k },{e^{\infty}_l}} \right\rangle|_{Q(x)},
\end{align*}
we have for $j=1,2$, $s>0$,
\begin{align*}
&{A_j}(s,t,x) = -\int_s^\infty  {\mathcal{R}(v(s'))}(\phi_j,{\phi}_s)  ds' + {{A}}^{\infty}_j\\
&{A_t}(s,t,x)=-\int^{\infty}_s{\mathcal{R}(v(s'))}(\phi_t,{\phi}_s)ds'.
\end{align*}
\end{Lemma}

\begin{Remark}
We adopt some notations for convenience:  Lemma  \ref{3.3} shows $A_j$ can be decomposed into the limit part and the effective  part:
\begin{align*}
A_j(s,t,x)&=A^{\infty}_j+{\tilde{A}}_j(s,t,x);\mbox{  }  \mbox{  }\mbox{  }\tilde{A}_j:=-\int^{\infty}_s {\mathcal{R}(v(s'))}(\phi_j,{\phi}_s)ds'.
\end{align*}
Similarly, we split $\phi_i$ into $\phi_i=\phi^{\infty}_i+\tilde{\phi}_i$, where $\tilde{\phi_i}=\int^{\infty}_s\partial_s\phi_id\kappa,$
and the $\beta$-component of $\phi^{\infty}_j$ is
\begin{align*}
\left\langle \partial _j Q(x), e^{\infty}_{\beta} \right\rangle+i\left\langle \partial _j Q(x),J e^{\infty}_{\beta} \right\rangle.
\end{align*}
\end{Remark}

We shall use the following evolution equations of heat tension field and  Schr\"odinger map tension filed.
\begin{Lemma}\label{asdf}
We have for any $s\ge 0$,
\begin{align}
\phi_s=\mathbf{D}^j\phi_j.\label{Yu79}
\end{align}
Along the Schr\"odinger direction, the heat tension field $\phi_s$ satisfies the nonlinear Schr\"odinger equation:
\begin{align}\label{heating1}
i\mathbf{D}_t\phi_s+\Delta_{A}\phi_s=i\partial_s Z+ \mathcal{R}(\phi^j,\phi_s)\phi_j,
\end{align}
Define  Schr\"odinger map tension  field to be  $Z:=\phi_t-i{\phi}_s$. One has
the Schr\"odinger map tension  field $Z$ in the heat direction satisfies
\begin{align}\label{Nsd2}
\partial_s Z=\Delta_{A}Z+\mathcal{R}(\phi^j,Z)\phi_j+\mathcal{R}(\phi^j, \phi_j)\phi_s.
\end{align}
\end{Lemma}

.
\section{Proof of Theorem 1.1, Part I. Estimates of heat tension field $\phi_s$}

Denote the induced connection on $Q^*T\mathcal{N}$ by $\bar{\nabla}$.
For Riemannian surface targets $\mathcal{N}$ with Gauss curvature $\kappa:\Bbb H^2\to \Bbb R$, the connection 1-form $A$ can be corresponded to two real valued functions  $A_1,A_2$:
\begin{align}\label{Effw}
A_j=\langle \bar{\nabla}_j e_{1},J e_1\rangle,
\end{align}
and the induced connection on the trivial complex vector bundle $M\times\Bbb I\times \Bbb C$ now reads as
\begin{align}\label{Effw2}
\mathbf{D}_j= \nabla_j+i A_{j}.
\end{align}
The commutator identity is
\begin{align}\label{Effw3}
[\mathbf{D}_j,\mathbf{D}_k]\psi= i\kappa(v) {\rm Im} (\phi_{j}\overline{\phi_k})\psi.
\end{align}
And the  identities in Lemma \ref{3.3} and  Lemma \ref{asdf} now can be formulated as
\begin{align}
{A_j}(s,t,x) &=- \int_s^\infty  \kappa    {\rm Im}({\phi}_s\overline{\phi_j})  ds' + {{A}}^{\infty}_j\\
\partial_s Z&=\Delta_{A}Z +i\kappa\phi^j\phi_j\overline{\phi_s}-i\kappa   {\rm Im}(\phi^{j}\overline{Z})\phi_j,\mbox{ } \mbox{ }Z(0,t,x)=0\label{2Kea}\\
i\partial_t \phi_s+\Delta_{A}\phi_s&=A_t\phi_s+i\partial_sZ+i\kappa   {\rm Im}(\phi^{j}\overline{\phi_s})\phi_j\label{Kea}\\
\partial_s \phi_s&= \Delta_{A}\phi_s-i\kappa   {\rm Im}(\phi^{j}\overline{\phi_s})\phi_j\\
\kappa   {\rm Im}(\phi^{j}\overline{\phi})\phi_j&=\kappa^{\infty}  {\rm Im}(\phi^{j}\overline{\phi})\phi_j+\tilde{\kappa} {\rm Im}(\phi^{j}\overline{\phi})\phi_j.
\end{align}
where we denote  $\kappa^{\infty}(x)=\kappa(Q(x))$, $\tilde{\kappa}=\kappa(v)-\kappa^{\infty}$.

The connection $A$ now can be decomposed  as
\begin{align}
{A_j}  &=\tilde{A}_j+ {{A}}^{\infty}_j \nonumber\\
\tilde{A}_{j}&=A^{lin}+A^{qua} \nonumber\\
A^{lin}_j&=-\kappa^{\infty}\int^{\infty}_s {\rm Im}(\phi^{\infty}_{j}\overline{\phi_s}) ds'\nonumber\\
A^{qua}_j&=-\tilde{\kappa}\int^{\infty}_s{\rm Im}(\phi^{\infty}_{j}\overline{\phi_s}) ds'- {\kappa}\int^{\infty}_s{\rm Im}(\tilde{\phi}_{j}\overline{\phi_s}) ds'.\label{qua}
 \end{align}

As mentioned in Remark 1.2, the linearized operator is self-adjoint.
\begin{Lemma}\label{linear}
Assume that $Q$ is a holomorphic map or an anti-holomorphic map from $\Bbb H^2$ to Riemannian surface $\mathcal{N}$, and $Q$ has a compact image.  Then the operator ${\bf H}$ defined by
\begin{align}
{\bf H}f:=\Delta_{A^{\infty}}f - \kappa^{\infty} {\rm Im}(\phi^{j}\overline{f})\phi_j
\end{align}
is self-adjoint in $L^2$ with domain $H^2$. If $\mathcal{N}$ is negatively curved, then ${\bf H}$ is non-negative.
\end{Lemma}
For more details on
holomorphic maps  and  anti-holomorphic maps, one may read Appendix B. And we remark that  $A^{\infty}$, $\phi^{\infty}$ and their covariant derivatives decay exponentially, see Lemma \ref{decay} in  Appendix B. In addition,  Proposition \ref{QNM} in Section 5 presents more properties of ${\bf H}$.

\begin{Remark}
The remarkable observation of \cite{LLOS2} shows the linear term $h^{jk}\phi^{\infty}_{k}\phi^{\infty}_{j}\phi_s$ in the RHS of (\ref{2Kea}) vanishes for   holomorphic maps $Q$. This also holds for anti-holomorphic maps, see Appendix B for some discussions.
\end{Remark}

\subsection{Bootstrap}

In this part, we explain the main bootstrap procedure  by assuming the three key estimates in Section \ref{gvn}.

Given $\gamma\in\Bbb R$,  define the  function $\omega_{\gamma}(s):(0,\infty)\to \Bbb R^+$:
\begin{align}
\omega_{\gamma} (s) = \left\{ \begin{gathered}
  s^{\gamma}, \mbox{ }s \in [0,1] \hfill \\
  {s^{L}},\mbox{ } s \in [1,\infty ) \hfill \\
\end{gathered}  \right.
\end{align}
and the function $\theta_{\gamma}(s):(0,\infty)\to \Bbb R^+$:
\begin{align}
\theta_{\gamma} (s) = \left\{ \begin{gathered}
  s^{-\gamma}, \mbox{ }s \in [0,1] \hfill \\
  {s^{-2L}},\mbox{ } s \in [1,\infty ) \hfill \\
\end{gathered}  \right.
\end{align}

Let $0<\delta<{\frac{1}{2}}$ be a fixed constant.
Also, define the function $\Theta_{\gamma}(s,s'):\{(s,s')\in\Bbb R^+: s\ge s'\} \to \Bbb R^+$:
\begin{align}
\Theta_{\gamma} (s,s') = \left\{ \begin{gathered}
  (s-s')^{-\gamma}, \mbox{ }s-s' \in [0,1] \hfill \\
  {e^{-\rho_{\delta}(s-s')}},\mbox{ } s-s' \in (1,\infty ) \hfill \\
\end{gathered}  \right.
\end{align}
where $\rho_{\delta}>0$ depends only on the given constant $\delta\in(0,\frac{1}{2})$.
Denote
$$M:=\|A^{\infty}\|_{L^{\infty}_{t}L^2_x\cap L^2_{x}}+\|\nabla A^{\infty}\|_{L^{\infty}_{t}L^2_x\cap L^{\infty}_x}+\|\phi^{\infty}\|_{L^{\infty}_{t}L^2_x\cap L^{\infty}_x}+\sup_{y\in \tilde{\mathcal{N}}}| \partial_y \kappa (y)|,
$$
where $\partial_y$ denotes the derivatives on $\Bbb R^{N}$.

Let $\epsilon_1,\epsilon_*>0$ be sufficiently small satisfying
\begin{align*}
& \epsilon_*\le \epsilon^{3}_1;\mbox{ }\mbox{ } \mbox{ }\epsilon^{\frac{1}{300}}_1(M^4+1)\le 1.
\end{align*}

Our main result is
\begin{Proposition}\label{MM22}
Assume that $u_0$ satisfies Theorem 1.1. Suppose that $u\in C([0,T];\mathcal{H}^3_{Q})$. Then we have
\begin{align}
\sup_{s>0}\omega_{\frac{1}{2}-\delta}(s) \|\phi_s\|_{L^2_tL^q_x\cap L^4_x([0,T]\times\Bbb H^2)}&\le \epsilon_1\label{mm1}\\
\sup_{s>0}\omega_{\frac{1}{2}-\delta}(s) \|e^{-\frac{1}{2}r}\nabla\phi_s\|_{L^2_{t}L^2_{x}([0,T]\times\Bbb H^2)}&\le  \epsilon_1\label{mm2}\\
\sup_{s>0}\omega_{\frac{1}{2}-\delta}(s) \|\nabla\phi_s\|_{L^{\infty}_{t}L^2_{x}([0,T]\times\Bbb H^2)}&\le \epsilon_1\label{mm3}.
\end{align}
\end{Proposition}

In order to prove Prop. \ref{MM22}, we rely on a bootstrap argument.
Fix
$$\frac{1}{2}+\frac{1}{50}<\alpha<\beta+\frac{1}{100}<\frac{1}{2}(1+\alpha)<1.
$$

Assume that $\mathcal{T}$ is the maximal positive time such that for all $0\le T< \mathcal{T}$, $q\in [4,\frac{8}{\delta}]$ there hold
\begin{align}
\sup_{s>0}\omega_{\frac{1}{2}-\delta}(s) \|\phi_s\|_{L^2_tL^q_x\cap L^{\infty}_tL^2_x([0,T]\times\Bbb H^2)}&\le \epsilon_1\label{om1}\\
\sup_{s>0}\omega_{\frac{1}{2}-\delta}(s) \|e^{-\frac{1}{2}r}\nabla\phi_s\|_{L^2_{t}L^2_{x}([0,T]\times\Bbb H^2)}&\le \epsilon^{\beta}_1\label{om2}\\
\sup_{s>0}\omega_{\frac{1}{2}-\delta}(s) \|\nabla\phi_s\|_{L^{\infty}_{t}L^2_{x}([0,T]\times\Bbb H^2)}&\le \epsilon^{\alpha}_1\label{om3}.
\end{align}

The following lemma and the local theory Lemma \ref{local} show  $\mathcal{T}>0$.
\begin{Lemma}\label{MM222}
Assume that $u\in C([0,T_0]; \mathcal{H}^3_{Q})$ be a solution to SL.  Then for any $0<T<T_0$, one has
\begin{align}\label{Fgkl}
\sup_{s>0}\omega_{0}(s) \|\phi_s\|_{L^{\infty}_tH^1_x([0,T]\times\Bbb H^2)}&\lesssim \epsilon_*.
\end{align}
\end{Lemma}
Lemma  \ref{MM222} is a corollary of Lemma \ref{heat} in  Appendix B.

By Sobolev embedding and local well-posedness of SL, we see Lemma \ref{MM222} indeed yields $\mathcal{T}>0$.

In order to prove $\mathcal{T}=\infty$, we apply  Strichartz estimates,
Morawetz estimates and energy estimates to the equation  (\ref{Kea}) to conclude that

\begin{Proposition}\label{MM2}
With the above assumptions (\ref{om1})-(\ref{om3}),  we have for any $T\in (0,\mathcal{T})$, $q\in [4,\frac{8}{\delta}]$ that
\begin{align}
\sup_{s>0}\omega_{\frac{1}{2}-\delta}(s) \|\phi_s\|_{L^2_tL^q_x\cap L^{\infty}_tL^2_x([0,T]\times\Bbb H^2)}&\lesssim_{M}  \epsilon^{1+\alpha}_1\label{bom1}\\
\sup_{s>0}\omega_{\frac{1}{2}-\delta}(s) \|e^{-\frac{1}{2}r}\nabla\phi_s\|_{L^2_{t}L^2_{x}([0,T]\times\Bbb H^2)}&\lesssim_{M} \epsilon^{\frac{1}{2}(1+\alpha)}_1\label{bom2}\\
\sup_{s>0}\omega_{\frac{1}{2}-\delta}(s) \|\nabla\phi_s\|_{L^{\infty}_{t}L^2_{x}([0,T]\times\Bbb H^2)}&\lesssim_{M} \epsilon^{\alpha+\frac{1}{2}}_1\label{bom3}.
\end{align}
\end{Proposition}

If Proposition \ref{MM2}  is done, with  Lemma \ref{MM222} and bootstrap we finish the proof of   Proposition \ref{MM22}. Then it is not hard to prove that  the solution $u$ of SL converges to harmonic maps  in $L^{\infty}$ norm by Proposition \ref{MM22}, see Section \ref{aa3} and Section \ref{aa2}. The resolution in energy space will be proved in Section \ref{aa4}.

\subsection{Proof of Proposition \ref{MM2}}

Slightly different  from Section 2, {\bf we make the convention that the notation $\phi$ now denotes the 1-form $\phi=\phi_jdx^j$, and $A$ now denotes the 1-form $A=A_jdx^j$. And the same holds for $A^{lin},A^{qua},\tilde{A},\tilde{\phi}$.}

The following is the heat estimates for $e^{s{\bf H}}$. (see Proposition \ref{QNM} in Section 5 for the proof)
\begin{Lemma}\label{MmmmL}
For any $\gamma\in (0,2)$, $1<q<p<\infty$,  there exists $\delta_{q}>0$ such that
\begin{align}
\|(-\Delta)^{ \gamma }e^{s{\bf H}}f\|_{L^{p}_x}\lesssim e^{-\delta_q s} s^{-{\gamma}-\frac{1}{q}+\frac{1}{p}}\|f\|_{L^q_x}\label{Ghhml}
\end{align}
\end{Lemma}

\begin{Lemma}(Parabolic Estimates of $\phi_s$)\label{XhhZ}
Let (\ref{om1})-(\ref{om3}) hold. Assume  that for any $q\in [4,\frac{8}{\delta}]$
\begin{align}
\sup_{s>0}\omega_0(s)\|\tilde{A}\|_{L^{\infty}_{t,x}}&\le  \epsilon^{\frac{1}{4}}_1 \label{0xdf}\\
\sup_{s>0}\omega_0(s)\|\nabla \tilde{A}\|_{L^{\infty}_{t}L^{\infty}_x} &\le  \epsilon^{\frac{1}{4}}_1 \label{1xdf}\\
\sup_{s>0}\omega_0(s)\|\tilde{\phi}\|_{L^{\infty}_{t}L^{\infty}_{x}\cap L^2_x}&\le  \epsilon^{\frac{1}{4}}_1 \label{2xdf}\\
\|\nabla \tilde{\phi}\|_{L^{\infty}_{t}L^2_{x}}&\le  \epsilon^{\frac{1}{4}}_1 \label{3xdf}\\
\sup_{s>0} \omega_{\frac{1}{2}-\delta}\|\nabla  \tilde{\phi}\|_{L^{2}_{t}L^q_{x}}&\le  \epsilon^{\frac{1}{4}}_1. \label{6xdf}
\end{align}
Let  $0<T<\mathcal{T}$, then for any $p\in [2,\infty]$, ${q}\in [4,\frac{8}{\delta}]$, we have
\begin{align}
\sup_{s>0} \omega_{1-\delta}(s)\| \nabla \phi_s\|_{L^{2}_tL^{{q}}_x}&\lesssim_{M}   \epsilon_1 \label{so1}\\
\sup_{s>0} \omega_{\frac{3}{2}-\delta}(s)\|\Delta \phi_s\|_{L^{2}_tL^{{q}}_x}&\lesssim_{M}  \epsilon_1 \label{so33}\\
\sup_{s>0} \omega_{1-\frac{7}{8}\delta}(s)\| \phi_s\|_{L^{2}_tL^{\infty}_x}&\lesssim_{M}   \epsilon_1 \label{axic}\\
\sup_{s>0} \omega_{1-\frac{7}{8}\delta}(s)\| \nabla \phi_s\|_{L^{2}_tL^{\infty}_x}&\lesssim_{M}   \epsilon_1 \label{3.31}\\
\sup_{s>0} \omega_{\frac{3}{2}-\frac{7}{8}\delta}(s)\|\Delta \phi_s\|_{L^{2}_tL^{\infty}_x}&\lesssim_{M}  \epsilon_1 \label{3.32}\\
\sup_{s>0} \omega_{1-\delta}(s)\|  \Delta \phi_s\|_{L^{\infty}_t L^{2}_x}&\lesssim_{M}  \epsilon^{\alpha}_1 \label{so2}\\
\sup_{s>0} \omega_{2-\frac{1}{2}\delta-\frac{1}{p}}(s)\|(-\Delta)^{1+\frac{1}{2}{\delta}} \phi_s\|_{L^{\infty}_{t}L^{p}_x}& \lesssim_{M} \epsilon^{\alpha}_1 \label{so3}\\
\sup_{s>0} \omega_{\frac{3}{2}-\delta-\frac{1}{p}}(s)\| (-\Delta)\phi_s\|_{L^{\infty}_{t}L^{p}_x}&\lesssim_{M}  \epsilon^{\alpha}_1 \label{so4}.
\end{align}
\end{Lemma}
\begin{proof}
{\bf Step 1. }
We obtain from (\ref{2xdf}), (\ref{3xdf}), (\ref{0xdf}) that
\begin{align*}
\|du\|_{L^{\infty}_{t}L^2_{x}}+\|\nabla du\|_{L^{\infty}_tL^2_x}\lesssim_{M} 1.
\end{align*}
Then
Lemma \ref{heat2} shows
\begin{align*}
 \|dv\|_{L^{\infty}_{t}L^2_{x}}+\|\nabla dv\|_{L^{\infty}_tL^2_x}
+\sup_{s>0}\min (1,s^{\frac{1}{2}})\|\nabla^2 dv\|_{L^{\infty}_tL^2_x}
+\sup_{s>0}\min( 1,s) \|\nabla^2 dv\|_{L^{\infty}_{t,x}}&\lesssim_M  1\\
 \sup_{s>0}\zeta_{\frac{1}{2}}(s)\| \partial_s v\|_{L^{\infty}_{t,x}}
+\sup_{s>0} \zeta_{1}(s) \|\nabla \partial_s v\|_{L^{\infty}_{t,x}} +
\sup_{s>0}\zeta_{{\frac{3}{2}}}(s) \|\nabla^2 \partial_s v \|_{L^{\infty}_{t,x}}& \lesssim_M  1.
\end{align*}
where $\zeta_{\gamma}(s):={\bf 1}_{s\in (0,1]}  s^{\gamma}+ {\bf 1}_{s\ge 1} e^{\rho_0 s}$, $\rho_0>0$.
Meanwhile, recall that in the intrinsic form one has
\begin{align}
|\nabla A|&\lesssim \int^{\infty}_s|\nabla dv||\partial_s v|ds'+\int^{\infty}_s| dv||\nabla\partial_s v|ds'\nonumber\\
|\nabla^2 A|&\lesssim \int^{\infty}_s|\nabla^2 dv||\partial_s v|ds'+\int^{\infty}_s| \nabla dv||\nabla\partial_s v|ds'+\int^{\infty}_s|dv||\nabla^2\partial_s v|ds'.\label{pad}
\end{align}
Hence, we conclude
\begin{align}\label{Hong}
\sup_{s>0}\omega_{\frac{1}{2}}(s)\|\nabla^2 A\|_{L^{\infty}_{t,x}}&\lesssim_{M} 1.
\end{align}

{\bf Step 2. }
Recall the evolution equation of $\phi_s$ in the heat direction:
\begin{align}\label{oKNjm}
\partial_s \phi_s=\Delta_{A}\phi_s-i\kappa {\rm Im} (\phi^j\overline{\phi_s})\phi_j.
\end{align}
As mentioned before, the RHS of  (\ref{oKNjm}) is in fact
 \begin{align*}
&{\bf H}\phi_s+2i\tilde{A}\nabla \phi_s+id^*\tilde{A}\phi_s-A^{\infty}\cdot\tilde{A}\phi_s-\tilde{A}\cdot\tilde{A}\phi_s-i\kappa{\rm Im} (\tilde{\phi}^j\overline{\phi_s})\phi^{\infty}_j\\
&-i\kappa {\rm Im} ({\phi}^{\infty}_j\overline{\phi_s})\tilde{\phi}^j-i\kappa{\rm Im} (\tilde{\phi}^j\overline{\phi_s})\tilde{\phi}_j-i\tilde{\kappa}h^{jk}{\rm Im} ({\phi}^{\infty}_k\overline{\phi_s}){\phi}^{\infty}_j,
\end{align*}
and we emphasize that ${\bf H}\phi_s$ is the linear part and the other above terms are at least quadratic. Then by Lemma \ref{MmmmL} and assumptions (\ref{0xdf})-(\ref{2xdf}), we get
\begin{align}
\|\nabla \phi_s\|_{L^2_tL^q_x}&\lesssim_{M}   \int^s_{\frac{s}{2}} \Theta_{\frac{1}{2}}(s,s')\|\tilde{A}\|_{L^{\infty}_{t,x}}\|\phi_s\|_{L^2_tL^q_x} ds'+\int^s_{\frac{s}{2}} \Theta_{\frac{1}{2}}(s,s')\|\nabla \tilde{A}\|_{L^{\infty}_{t,x}}\|\phi_s\|_{L^2_tL^{q}_x} ds'\nonumber\\
&+\int^s_{\frac{s}{2}} \Theta_{\frac{1}{2}}(s,s')\|\tilde{\phi}\|_{L^{\infty}_{t,x}}\|\phi_s\|_{L^2_tL^{q}_x} ds'+\int^s_{\frac{s}{2}} \Theta_{\frac{1}{2}}(s,s')\|\tilde{A}\|_{L^{\infty}_{t,x}}\|\nabla \phi_s\|_{L^2_tL^{q}_x} ds'\nonumber\\
&+\Theta_{\frac{1}{2}}(s,\frac{s}{2}) \|\phi_s(\frac{s}{2},t)\|_{L^2_tL^q_x}\nonumber\\
&\lesssim \epsilon_1 \Theta_{\frac{1}{2}}(s,\frac{s}{2})\theta_{\frac{1}{2}-\delta}(s)+\epsilon^{\frac{1}{4}}_1  \int^s_{\frac{s}{2}}  \Theta_{\frac{1}{2}}(s,s')\|\phi_s\|_{L^2_tL^q_x} ds'
+\epsilon^{\frac{1}{4}}_1 \int^{s}_{\frac{s}{2}} \Theta_{\frac{1}{2}}(s,s')\|\nabla \phi_s\|_{L^2_tL^q_x} ds'.\label{onKNjm}
\end{align}
Set
\begin{align*}
B(s):=\sup_{0<s'<s}\omega_{\frac{1}{2}-\delta}(s')\|\nabla \phi_s(s')\|_{L^2_tL^q_x}.
\end{align*}
Then by (\ref{om1}) and (\ref{onKNjm}), we get
\begin{align*}
B(s)\le \epsilon_1 + \epsilon^{\frac{1}{4}}_1 B(s).
\end{align*}
Thus $B(s)\lesssim \epsilon_1 $, which gives (\ref{so1}).

Now, let's prove (\ref{so2}). Again by Lemma \ref{MmmmL} and assumptions (\ref{0xdf})-(\ref{2xdf}), we get
\begin{align}
\|\Delta  \phi_s\|_{L^{\infty}_tL^2_x}&\lesssim
   \int^s_{\frac{s}{2}}\Theta_{\frac{1}{2}}(s,s') \|\nabla^2\tilde{A}\|_{L^{\infty}_tL^{\infty}_x} \|\nabla \phi_s\|_{L^{\infty}_{t}L^{2}_{x}} ds'\label{Hg1}\\
&+\int^s_{\frac{s}{2}}\Theta_{\frac{1}{2}}(s,s')\|\nabla \tilde{\phi}\|_{L^{\infty}_tL^{2}_x} \| \phi_s\|_{L^{\infty}_{t,x}} ds'\label{Hg2}\\
&+   \int^s_{\frac{s}{2}}\Theta_{\frac{1}{2}}(s,s')(\|\nabla \tilde{A}\|_{L^{\infty}_tL^4_x}+\|\tilde{\phi}\|_{L^{\infty}_{t}L^4_{x}})\|\nabla \phi_s\|_{L^{\infty}_tL^4_x} ds'\nonumber\\
&+  \int^{s}_{\frac{s}{2}} \Theta_{\frac{1}{2}}(s,s')\|\tilde{A}\|_{L^{\infty}_{t,x}}\| \Delta  \phi_s\|_{L^{\infty}_tL^2_x} ds'
 +\Theta_{\frac{1}{2}}(s,\frac{s}{2})\|\nabla \phi_s(\frac{s}{2},t)\|_{L^{\infty}_tL^2_x}\nonumber
\end{align}
Then by (\ref{0xdf})-(\ref{6xdf}), (\ref{Hong}), and Sobolev embeddings,
we obtain the function $B'(s)$ defined by
\begin{align*}
B'(s):=\sup_{0<s'<s}\omega_{\frac{1}{2}-\delta}(s')\|\nabla^2 \phi_s(s')\|_{L^{\infty}_tL^2_x}.
\end{align*}
satisfies
\begin{align*}
B'(s)\lesssim \epsilon^{\alpha}_1 + \epsilon^{\frac{1}{4}}_1 B'(s).
\end{align*}
Thus $B'(s)\lesssim \epsilon^{\alpha}_1 $, which gives (\ref{so2}).

Similarly, one can prove (\ref{so3}), (\ref{so4}) by additionally applying $\|e^{s{\bf H}}f\|_{L^p_x}\lesssim s^{-\frac{1}{2}+\frac{1}{p}}\|f\|_{L^2_x}$.

By Sobolev interpolation inequalities, we infer from (\ref{so3}),  (\ref{so2}), (\ref{om3}) that
\begin{align}
\sup_{s>0}\omega_{\frac{1}{2}-\frac{1}{4}\delta}\|\phi_s\|_{L^{\infty}_{t,x}}&\lesssim   \epsilon^{\alpha}_1\label{Xy1}\\
\sup_{s>0}\omega_{1-\frac{1}{4}\delta}\|\nabla \phi_s\|_{L^{\infty}_{t,x}}&\lesssim   \epsilon^{\alpha}_1.\label{Xy2}
\end{align}

Now, we prove (\ref{so33}).   Lemma \ref{MmmmL} yields
\begin{align*}
\|\Delta \phi_s\|_{L^2_t L^q_x}&\lesssim_{M}   \int^s_{\frac{s}{2}} \Theta_{\frac{1}{2}}(s,s')
\|\tilde{A}\|^2_{L^{\infty}_{t,x}}\|\nabla\phi_s\|_{L^2_t L^q_x} d s'+\int^s_{\frac{s}{2}}
\Theta_{\frac{1}{2}}(s,s')\|\nabla^2 \tilde{A}\|_{L^{\infty}_{t,x}}\|\nabla \phi_s\|_{L^{2}_{t}L^q_{x}} ds'\nonumber\\
&+ \int^s_{\frac{s}{2}} \Theta_{\frac{1}{2}}(s,s')\|\tilde{A}\|_{L^{\infty}_{t,x}}\|\nabla^2\phi_s\|_{L^2_t L^q_x} ds'
+\int^s_{\frac{s}{2}}\Theta_{\frac{1}{2}}(s,s')\|\nabla \tilde{\phi}\|_{L^{2}_{t}L^{q}_x}\|\phi_s\|_{L^{\infty}_{t,x}} ds'\nonumber\\
&+\int^s_{\frac{s}{2}}\Theta_{\frac{1}{2}}(s,s')\|\tilde{\phi}\|_{L^{\infty}_{t,x}}\|\nabla \phi_s\|_{L^2_t L^{q}_x} ds' +\Theta_{\frac{1}{2}}(s,\frac{s}{2})\|\nabla \phi_s(\frac{s}{2},t)\|_{L^{2}_tL^q_x}.
\end{align*}
Set
\begin{align*}
\hat{B}(s):=\sup_{0<s'<s}\omega_{\frac{3}{2}-\delta}(s')\|\Delta \phi_s(s')\|_{L^2_tL^q_x}.
\end{align*}
Then by (\ref{Hong}), (\ref{om1}) and (\ref{so1}), we get
\begin{align*}
\hat{B}(s)\lesssim  \epsilon_1 + \epsilon^{\frac{1}{4}}_1 \hat{B}(s).
\end{align*}
Thus $\hat{B}(s)\lesssim_{M} \epsilon_1 $, which gives (\ref{so33}).

(\ref{axic}), (\ref{3.31}) and (\ref{3.32}) follow  by the same arguments as (\ref{so1}), (\ref{so33}) via additionally applying
 $\|e^{s{\bf H}}f\|_{L^{\infty}_x}\lesssim s^{-\frac{\delta}{8}}\|f\|_{L^{\frac{8}{\delta}}_x}$.
\end{proof}

\begin{Lemma}(Close parabolic estimates of $\tilde{A},\tilde{\phi}$)\label{XffZ}
Let (\ref{om1})-(\ref{om3}) hold. Assume  that (\ref{0xdf}), (\ref{1xdf}), (\ref{2xdf}), (\ref{3xdf}), (\ref{6xdf}) hold.
Let  $0<T<\mathcal{T}$, then we have for any $\tilde{q}\in [4,\infty]$, ${q}\in [4,\frac{8}{\delta}]$, $p\in [2,\infty]$,
\begin{align}
\|\phi\|_{L^{\infty}_{t,x}}&\lesssim_{M} 1  \label{pqq1}\\
\sup_{s>0}\omega_{0}(s)\|\widetilde{\phi}\|_{L^{\infty}_{t}L^2_{x}\cap L^{\infty}_{t,x}}&\lesssim_{M} \epsilon^{\alpha}_1  \label{pqq2}\\
\sup_{s>0}\omega_{\max(0,\frac{1}{2}-\delta-\frac{1}{p})} (s)\|\nabla \widetilde{\phi}\|_{L^{\infty}_{t}L^{p}_{x}}&\lesssim_{M} \epsilon^{\alpha}_1  \label{pqq3}\\
\sup_{s>0}\omega_{0}(s)\|\widetilde{\phi}\|_{L^{2}_{t}L^{\tilde{q}}_{x}}&\lesssim_{M} \epsilon_1  \label{pqq4}\\
\sup_{s>0}\omega_{\frac{1}{2}-\delta} (s)\|\nabla \widetilde{\phi}\|_{L^{2}_{t}L^{q}_{x}}&\lesssim_{M} \epsilon_1\label{pqq5}\\
\sup_{s>0}\omega_{\frac{1}{2}-\frac{7}{8}\delta} (s)\|\nabla \widetilde{\phi}\|_{L^{2}_{t}L^{\infty}_{x}}&\lesssim_{M} \epsilon_1.\label{3.46}
 \end{align}
And the connection satisfies
\begin{align}
\| A\|_{L^{\infty}_{t,x}}&\lesssim_{M} 1 \label{aqq1}\\
\sup_{s>0}\omega_{0}(s)\|\tilde{A}\|_{L^{\infty}_{t,x}}+\sup_{s>0}\omega_{0}(s)\|\nabla \tilde{A}\|_{L^{\infty}_{t,x}}&\lesssim_{M}  \epsilon^{\alpha}_1 \label{aqq2}
 \end{align}
where the implicit constant is at most of $(1+M)^2$ order.
\end{Lemma}
\begin{proof}
(\ref{pqq1}) and (\ref{aqq1}) follow directly by (\ref{0xdf}), (\ref{2xdf}).

By the identity
\begin{align}
\widetilde{\phi}_j&=\int^{\infty}_s \partial_j \phi_s +iA_j\phi_s ds'\nonumber\\
&=\int^{\infty}_s \partial_j \phi_s +iA^{\infty}_j\phi_s ds'+i\tilde{A}_j\phi_s ds'\label{Df}
\end{align}
and (\ref{om3}), (\ref{0xdf}), Poincare inequality,  the $L^{\infty}_{t}L^2_{x}$ bounds stated in (\ref{pqq2}) follows.

(\ref{Xy1}) and (\ref{Xy2}) give
\begin{align}
\int^{\infty}_s\|{\phi}_s\|_{L^{\infty}_{t,x}}ds'&\lesssim \epsilon^{\alpha}_1\label{3.51}\\
\int^{\infty}_s\|\nabla {\phi}_s\|_{L^{\infty}_{t,x} }ds'&\lesssim \epsilon^{\alpha}_1\label{3.52}
\end{align}
Thus due to (\ref{Df}) and (\ref{0xdf}), the $L^{\infty}_{t,x}$ bounds stated in (\ref{pqq2}) follows.

(\ref{pqq1}) directly follows by (\ref{pqq2}).

By (\ref{so1}),  (\ref{axic}), (\ref{3.31}), (\ref{0xdf}) and (\ref{Df}), one easily deduces (\ref{pqq4}).

Using (\ref{Df}), we see
\begin{align}\label{Df2}
|\nabla \widetilde{\phi}| \lesssim \int^{\infty}_s |\nabla^2 \phi_s| +|A ||\nabla \phi_s|+|\nabla A||\phi_s| ds'.
\end{align}
By (\ref{so4}),  we get (\ref{pqq3}). And (\ref{3.46}) follows by (\ref{3.31}), (\ref{3.32}), (\ref{axic}).

Recall that $\tilde{A}$ is given by
\begin{align}\label{Df2}
 \tilde{A}=\int^{\infty}_s \kappa {\rm Im}(\phi^{i}\overline{\phi_s})\phi_{i} ds'.
\end{align}
Then the $\|\tilde{A}\|_{L^{\infty}_{t,x}}$ in  (\ref{aqq2})  follows by (\ref{pqq1}) and (\ref{3.52}).
Again due to (\ref{Df2}), we get
\begin{align}\label{Df3}
 |\nabla \tilde{A}|\lesssim \int^{\infty}_s( |\phi|^3|\phi_s|+|\phi_s||\nabla \phi|+|\nabla \phi_s||\phi|^2 )ds'.
\end{align}
Then the $\|\nabla \tilde{A}\|_{L^{\infty}_{t,x}}$ bound  stated in (\ref{aqq2}) follows by (\ref{Xy1}), (\ref{Xy2}) , (\ref{pqq1}) and (\ref{pqq3}).

\end{proof}

\begin{Proposition}(Parabolic estimates of $\phi_s$)\label{XssZ}
Assume that (\ref{om1})-(\ref{om3}) hold. Then for  $0<T<\mathcal{T}$, we have (\ref{so1})-(\ref{so4}), and there also hold (\ref{pqq1})-(\ref{3.46}), (\ref{aqq1})-(\ref{aqq2}). Furthermore, one has for $q\in [4,\frac{8}{\delta}]$
\begin{align}
\sup_{s>0}\omega_{0}(s)\|\tilde{\kappa}\|_{L^{\infty}_{t,x}\cap L^{2}_tL^{q}_x\cap L^{\infty}_tL^2_x}&\lesssim_{M}  \epsilon^{\alpha}_1 \label{Po1}\\
\sup_{s>0}\omega_{0}(s)\| \nabla \tilde{\kappa}\|_{L^{\infty}_{t,x}\cap L^{2}_t L^{q}_x\cap L^{\infty}_tL^2_x}&\lesssim_{M} \epsilon^{\alpha}_1 \label{Po2}\\
\sup_{s>0}\omega_{0}(s)\|{A}^{qua}\|_{L^{1}_{t}L^2_{x}\cap L^1_tL^{\infty}_x}&\lesssim_{M}  \epsilon^2_1 \label{aqq3}\\
\sup_{s>0}\omega_{0}(s)\|\nabla {A}^{qua}\|_{L^{2}_{t}L^q_{x}}&\lesssim_{M} \epsilon^2_1 \label{aqq4}
\end{align}
\end{Proposition}
\begin{proof}
By bootstrap argument and
Lemma \ref{XffZ}, we see  (\ref{0xdf})-(\ref{6xdf}) hold for all $s>0$, $T\in (0,\mathcal{T})$. Then Lemma \ref{XhhZ} shows  (\ref{pqq1})-(\ref{3.46}) hold
 for all $s>0$, $T\in (0,\mathcal{T})$. Meanwhile,  Lemma \ref{XffZ} also shows in fact one has  (\ref{pqq1})-(\ref{pqq4}), (\ref{aqq1})-(\ref{aqq2}) hold for any $s>0$, $T\in (0,\mathcal{T})$.

The rest is to prove (\ref{Po1})-(\ref{aqq4}).

Recall that
\begin{align*}
|\tilde{\kappa}|\lesssim \int^{\infty}_s |\phi_s| ds'ds'
\end{align*}
Then (\ref{Po1}) follows from  Lemma 3.4.

Also, we have
\begin{align*}
|\nabla \tilde{\kappa}|\lesssim \int^{\infty}_s |\phi||\phi_s| +|\nabla \phi_s|ds'
\end{align*}
Then  the $L^{\infty}_{t,x}$ bound in (\ref{Po2}) follows by (\ref{pqq1}), (\ref{pqq4}), (\ref{Xy1}) (\ref{Xy2}).
 Moreover,  the $L^{2}_{t}L^q_{x}\cap L^{\infty}_tL^2_x$ bound in (\ref{Po2}) follows by corresponding ones in Lemma  3.4 and Lemma 3.5 as well.

Recall that
\begin{align}\label{Df4}
A^{qua}=-\int^{\infty}_s\kappa{\rm Im}(\tilde{\phi} \overline{\phi_s})ds'-\int^{\infty}_s\tilde{\kappa}{\rm Im}({\phi}^{\infty} \overline{\phi_s})ds'
\end{align}
Then (\ref{aqq3}) follows by (\ref{pqq4}), (\ref{om1}), (\ref{Po1}).

By (\ref{Df4}), one has
\begin{align}
|\nabla {A}^{qua}|\lesssim \int^{\infty}_s( |\nabla \tilde{\phi}||\phi_s|+|\nabla \phi_s||\tilde{\phi}|+|\phi_s||\tilde{\phi}||\phi_x|)ds'.
\end{align}
Thus (\ref{aqq3}),
(\ref{aqq4}) follow by Lemma \ref{XffZ} and Lemma \ref{XhhZ}.

\end{proof}

\begin{Lemma}(Parabolic estimates of $Z$)\label{FcZ}
Assuming that (\ref{om1})-(\ref{om3}) hold, for  $0<T<\mathcal{T}$, and any $q\in [4,\frac{8}{\delta}]$, $p\in [2,\infty]$, we have
\begin{align}
\sup_{s>0}\omega_{-\frac{1}{2}-\delta}(s)\|Z\|_{L^{\infty}_{t}L^{2}_x}& \lesssim_{M}   \epsilon^{2\alpha}_1 \label{zq1}\\
\sup_{s>0}\omega_{\frac{1}{2}-\frac{1}{p}-\delta}(s)\|\nabla Z\|_{L^{\infty}_{t}L^{p}_x}& \lesssim_{M}   \epsilon^{2\alpha}_1 \label{bbzq1}\\
\sup_{s>0}\omega_{-\delta}(s)\|\nabla Z\|_{L^{1}_{t}L^{2}_x}& \lesssim_{M} \epsilon^{2\alpha}_1 \label{zq2}\\
\sup_{s>0}\omega_{\frac{1}{2}-\delta}(s)\|\nabla^2 Z\|_{L^{1}_{t}L^{2}_x}& \lesssim_{M} \epsilon^{2\alpha}_1 \label{zq3}\\
\sup_{s>0}\omega_{-\frac{1}{2}-\delta}(s)\| Z\|_{L^{2}_{t}L^{q}_x}& \lesssim_{M} \epsilon^{2\alpha}_1  \label{zq4}\\
\sup_{s>0}\omega_{-\delta}(s)\|\nabla  Z\|_{L^{2}_{t}L^{q}_x}& \lesssim_{M} \epsilon^{2\alpha}_1  \label{zq5}
\end{align}
\end{Lemma}
 \begin{proof}
Recall that $Z$ satisfies (\ref{2Kea}):
 \begin{align}\label{KNjm}
\partial_sZ=\Delta_{A}Z-i\kappa{\rm Im}(\phi^j\overline{Z})\phi_j +i\kappa {\phi}^k\phi_k\phi_s, \mbox{ }Z(0,t,x)=0.
\end{align}
As mentioned before, the RHS of  (\ref{KNjm}) is in fact
\begin{align*}
&{\bf H}Z+2i\tilde{A}\nabla Z+id^*\tilde{A}Z-A^{\infty}\cdot\tilde{A}Z-\tilde{A}\cdot\tilde{A}Z-i\kappa{\rm Im} (\tilde{\phi}^j\overline{Z})\phi^{\infty}_j\\
&-i\kappa {\rm Im} ({\phi}^{\infty}_j\overline{Z})\tilde{\phi}^j-i\kappa{\rm Im} (\tilde{\phi}^j\overline{Z})\tilde{\phi}_j-i\tilde{\kappa}h^{kj}{\rm Im} ({\phi}^{\infty}_j\overline{Z}){\phi}^{\infty}_k \\
&+i\kappa h^{kj}{\phi}^{\infty}_j\tilde{\phi}_k\phi_s+i\kappa  \tilde{\phi}^j\tilde{\phi}_j\phi_s ,
\end{align*}
we remark  that ${\bf H}Z$ is the only linear part and the other above terms are at least quadratic, and the only troublesome term is   $\tilde{A}\cdot\nabla Z$, which will be dominated by smoothing effects and the Leibnitz formula.

Thus  Lemma \ref{MmmmL}, Duhamel principle, and the $L^{\infty}_{t,x}$ bounds of $\tilde{A},\tilde{\phi},\nabla \tilde{A},\tilde{\kappa}$ in Proposition \ref{XssZ} (see (\ref{pqq2}), (\ref{aqq1})-(\ref{aqq2}), (\ref{Po1}))   give
\begin{align*}
\|Z\|_{L^1_tL^2_x}&\lesssim_{M} \int^s_0\Theta_{0}(s,s')(\|\tilde{A}\|_{L^{\infty}}+\|\nabla \tilde{A}\|_{L^{\infty}}+\|\tilde{\phi}\|_{L^{\infty}})\|Z\|_{L^1_{t}L^2_x}ds'\\
&+\int^{s}_0\Theta_{\frac{1}{2}}(s,s')\|\nabla \tilde{A}\|_{L^{\infty}}\|Z\|_{L^1_{t}L^2_x}ds'\\
&+\int^{s}_0\Theta_{0}(s,s')(\|\tilde{\kappa}\|_{L^2_tL^4_x}+\|\tilde{\phi}\|_{L^2_{t}L^4_x})\|\phi_s\|_{L^2_t L^4_x}ds',
\end{align*}
where the in the last line we also  applied $L^{2}_t L^4_x$ norms of $\tilde{\kappa},\tilde{\phi}$ in Proposition \ref{XssZ}. Thus we get
\begin{align*}
\|Z\|_{L^1_tL^2_x}&\lesssim_{M} \epsilon^{2\alpha}_1 \min(s^{\frac{1}{2}+\delta},s^{-L}),
\end{align*}
which gives (\ref{zq1}).
Again using Lemma \ref{MmmmL}, Duhamel principle, one obtains
\begin{align*}
\|\nabla Z\|_{L^1_tL^2_x}&\lesssim_{M} \int^s_0\Theta_{\frac{1}{2}}(s,s')(\|\tilde{A}\|_{L^{\infty}}+\|\nabla \tilde{A}\|_{L^{\infty}})\|Z\|_{L^1_{t}L^2_x}ds'\\
&+\int^{s}_0\Theta_{\frac{1}{2}}(s,s')\|\tilde{A}\|_{L^{\infty}_{t,x}}\|\nabla Z\|_{L^1_{t}L^2_x}ds'\\
&+\int^{s}_0\Theta_{\frac{1}{2}}(s,s')(\|\tilde{\kappa}\|_{L^2_tL^4_x}+\|\tilde{\phi}\|_{L^2_{t}L^4_x})\|\phi_s\|_{L^2_t L^4_x}ds'.
\end{align*}
Define $B_{1}(s)$ to be
\begin{align*}
B_{1}(s):=\sup_{0<\tau<s}\max(\tau^{-\delta},\tau^{L})\|\nabla Z(\tau)\|_{L^1_tL^2_x}.
\end{align*}
Then we arrive at
\begin{align*}
B_{1}(s)\lesssim  \epsilon^{2\alpha}_1 +\epsilon^{2\alpha}_1 B_{1}(s),
\end{align*}
from which (\ref{zq2}) follows. Similarly one has (\ref{bbzq1}).

For $\Delta Z$, we use
\begin{align*}
 \|\Delta  Z\|_{L^1_tL^2_x}&
\lesssim \int^s_0\Theta_{\frac{1}{2}}(s,s')\|\tilde{A}\|_{L^{\infty}}\|\nabla^2 Z\|_{L^1_{t}L^2_x}+\|\nabla \tilde{A}\|_{L^{\infty}}\|\nabla Z\|_{L^1_{t}L^2_x}ds'
+\|\nabla^2\tilde{A}\|_{L^{\infty}}\| Z\|_{L^1_{t}L^2_x}+\\
&+\int^{s}_0\Theta_{\frac{1}{2}}(s,s')(\|\nabla \tilde{\kappa}\|_{L^2_tL^4_x}+\|\nabla\tilde{\phi}\|_{L^2_{t}L^4_x})\|\phi_s\|_{L^2_t L^4_x}ds'\\
&+\int^{s}_0\Theta_{\frac{1}{2}}(s,s')(\|  \tilde{\kappa}\|_{L^2_tL^4_x}+\| \tilde{\phi}\|_{L^2_{t}L^4_x})\|\nabla \phi_s\|_{L^2_t L^4_x}ds'.
\end{align*}
where we applied (\ref{pqq5}), (\ref{Hong}). Then  (\ref{zq3}) follows by (\ref{equi}) and Proposition 3.3.

(\ref{zq4}) follows by the same way as  (\ref{zq1}) and $L^{\infty}_{t,x}$ bounds of $\tilde{A},\tilde{\phi},\nabla \tilde{A}$. And the same arguments as (\ref{zq2}) give (\ref{zq5}).

\end{proof}

\begin{Lemma}(Estimates of $\phi_t$ and $A_t$)
Assuming that (\ref{om1})-(\ref{om3}) hold. Then for  $0<T<\mathcal{T}$, any $q\in [4,\frac{8}{\delta}]$, $p\in [2,\infty]$, we have
\begin{align}
\sup_{s>0} \omega_{ \frac{1}{2}-\delta}(s)\|\phi_t\|_{L^{2}_{t}L^{q}_x}& \lesssim_{M} \epsilon^{2\alpha}_1\label{tzq2}\\
\sup_{s>0} \omega_{ \frac{3}{2}-\delta}(s)\|\nabla \phi_t\|_{L^{2}_{t}L^{q}_x}& \lesssim_{M} \epsilon^{2\alpha}_1\label{tzq1}\\
\sup_{s>0} \omega_{ 1-\delta-\frac{1}{p}}(s)\|\nabla \phi_t\|_{L^{\infty}_{t}L^{p}_x}& \lesssim_{M} \epsilon^{2\alpha}_1.\label{Aaa}
\end{align}
And $A_t,\nabla A_t$ satisfy
\begin{align}
\sup_{s>0} \omega_{0}(s)\|A_t\|_{ L^1_{t}L^{\infty}_x}& \lesssim_{M} \epsilon^{2}_1 \label{tzq3}\\
\sup_{s>0} \omega_{0}(s)\|A_t\|_{L^{2}_{t}L^{4}_x}& \lesssim_{M} \epsilon^{2}_1 \label{tzq5}\\
\sup_{s>0} \omega_{0}(s)\|A_t\|_{L^{\infty}_{t,x}} & \lesssim_{M} \epsilon^{2\alpha}_1, \label{tzq36}\\
\sup_{s>0} \omega_{0}(s)\|\nabla A_t\|_{L^{2}_{t}L^{\frac{2}{1-\delta}}_x}& \lesssim_{M} \epsilon^{2}_1 \label{tzq55}
\end{align}
\end{Lemma}
\begin{proof}
Recall that $Z=\phi_s-i\phi_t$. Then (\ref{tzq1}), (\ref{Aaa}) and (\ref{tzq2}) follow  by Lemma \ref{FcZ} and Proposition \ref{XssZ}.

Recall also that
\begin{align*}
|A_t|&\lesssim \int^{\infty}_s|\phi_t||\phi_s|ds'\\
|\nabla A_t|&\lesssim \int^{\infty}_s(|\phi_s||\phi||\phi_t|+|\nabla \phi_t||\phi_s|+|\phi_t||\nabla \phi_s|)ds'.
\end{align*}
Then, (\ref{tzq36}) follows directly from Lemma \ref{FcZ} and Proposition \ref{XssZ}. Moreover, (\ref{tzq3}), (\ref{tzq5}) are dominated by
\begin{align*}
\|A_t\|_{L^{1}_{t}L^{\infty}_x}& \lesssim \int^{\infty}_s\|\phi_t\|_{L^2_tL^{\infty}_x}\|\phi_s\|_{L^{2}_tL^{\infty}_x}ds' \\
\|A_t\|_{L^{2}_{t}L^{4}_x}& \lesssim \int^{\infty}_s\|\phi_t\|_{L^2_tL^4_x}\|\phi_s\|_{L^{\infty}_tL^{\infty}_x}ds'.
\end{align*}
The rest (\ref{tzq55}) is  bounded as
\begin{align*}
\|\nabla A_t\|_{L^{2}_{t}L^{\frac{2}{1-\delta}}_x}& \lesssim \int^{\infty}_s\|\nabla \phi_t\|_{L^{\infty}_tL^{\frac{8}{4-5\delta}}_x}\|\phi_s\|_{L^{2}_tL^{\frac{8}{\delta}}_x}ds'+ \int^{\infty}_s\|\nabla \phi_s\|_{L^{\infty}_tL^{ \frac{8}{4-5\delta}}_x}\|\phi_t\|_{L^{2}_tL^{\frac{8}{\delta}}_x}ds'\\
&+ \int^{\infty}_s\|\phi_t\|_{L^2_tL^{\frac{8}{\delta}}_x}\|\phi_s\|_{L^{\infty}_tL^{\frac{8}{4-5\delta}}_x}\|\phi\|_{L^{\infty}_tL^{\infty}_x}ds'.
\end{align*}
Thus  by the  formula  $Z=\phi_s-i\phi_t$, Lemma \ref{FcZ} and Proposition \ref{XssZ}, we obtain  (\ref{tzq3})-(\ref{tzq55}).

\end{proof}

\subsection{Evolution along the Schr\"odinger direction}

The proof of  Proposition \ref{MM2} will be divided into four lemmas.

First, we deal with $\phi_s(s,0,x)$. Lemma 10.2 in Appendix B gives
\begin{Lemma}\label{initial}
For initial data $u_0$ in Theorem 1.1, there holds
\begin{align}
\sup_{s>0}\omega_{0}(s) \|\phi_s(s,0,x)\|_{L^2_x}&\lesssim \epsilon_*\\
\sup_{s>0}\omega_{\frac{1}{4}-\delta}(s) \|\phi_s(s,0,x)\|_{H^{\frac{1}{2}}_x}&\lesssim \epsilon_*\\
\sup_{s>0}\omega_{\frac{1}{2}-\delta}(s) \|\phi_s(s,0,x)\|_{H^{1}_x}&\lesssim \epsilon_*.
\end{align}
\end{Lemma}

\begin{Lemma}\label{X33}
With the above assumption (\ref{om1})-(\ref{om3}), for any $q\in [4,\frac{8}{\delta}]$
one has
\begin{align}
\sup_{s>0}\omega_{\frac{1}{2}-\delta}(s) \|\phi_s\|_{L^2_tL^q_x\cap L^4_x([0,T]\times\Bbb H^2)}&\lesssim \epsilon^{\alpha+1}_1\label{ppom1}
\end{align}
\end{Lemma}

\begin{proof}
Let's first rewrite the equation of $\phi_s$ in (\ref{Kea}):
\begin{align*}
i\partial_t\phi_s+{\bf H}\phi_s&=A_t\phi_s+\Lambda_{\tilde{A}}\phi_s-A^{\infty}\cdot\tilde{A}\phi_s+i\kappa {\rm Im}(\tilde{\phi}\overline{\phi_s}){\tilde{\phi}}
+i\kappa {\rm Im}({\phi}^{\infty}\overline{\phi_s}) {\tilde{\phi}}+i\kappa {\rm Im}(\tilde{\phi}\overline{\phi_s}) {{\phi}^{\infty}}\\
&+i\tilde{\kappa} h^{jk}{\rm Im}({\phi}^{\infty}_k \overline{\phi_s}) {{\phi}^{\infty}_{j}}+i\partial_s Z.
\end{align*}
where for simplicity we write $\Lambda_{\tilde{A}}f=\Delta_{\tilde{A}}f-\Delta f$.

Let ${\bf N}$ denote the nonlinearity, i.e.
\begin{align}\label{aaFF}
i\partial_t\phi_s+{\bf H}\phi_s&={\bf N}.
\end{align}

Applying  endpoint Strichartz estimates of Lemma \ref{Strichartz} for ${\bf H}$ to  (\ref{aaFF}) gives
\begin{align*}
\|\phi_s\|_{L^2_t  L^q_x\cap L^{\infty}_tL^2_x}&\lesssim \|{\bf N}\|_{L^1_tL^2_x}+\|\phi_s(s,0,x)\|_{L^2_x},
\end{align*}
where we omit the integral domain $ ([0,T]\times\Bbb H^2)$ for all the above involved norms for simplicity.

The initial data term $\|\phi_s(s,0,x)\|_{L^2_x}$ is admissible by  Lemma \ref{initial}.
Now, let's consider the three classes of terms in ${\bf N}$.
The first class is quadratic terms in $\phi_s$:
\begin{align*}
{\bf N_1}&:=\Lambda_{A^{lin}}\phi_s-A^{\infty}\cdot A^{lin}\phi_s+\kappa^{\infty}{\rm Im}(\tilde{\phi}^j \overline{\phi_s}){\phi}^{\infty}_j
 +\kappa^{\infty}{\rm Im}(\tilde{\phi}^{\infty}_j\overline{\phi_s})\tilde{\phi}_j+\tilde{\kappa}h^{kj}{\rm Im}(\tilde{\phi}^{\infty}_k\overline{\phi_s}) {\phi}^{\infty}_j.
\end{align*}
The second is cubic and higher order terms in $\phi_s$:
\begin{align*}
{\bf N_2}&:= A_t\phi_s+\Lambda_{A^{qua}}\phi_s-A^{\infty}\cdot A^{lin}\phi_s-A^{\infty}\cdot A^{qua}\phi_s-A^{qua}\cdot A^{lin}\phi_s-A^{lin}\cdot A^{lin}\phi_s+i\kappa {\rm Im}(\tilde{\phi}\overline{\phi_s}){\tilde{\phi}}\\
&+i\tilde{\kappa} {\rm Im}({\phi}^{\infty}\overline{\phi_s}) {\tilde{\phi}}+i\tilde{\kappa} {\rm Im}(\tilde{\phi}\overline{\phi_s}) {{\phi}^{\infty}}
+i\tilde{\kappa} {\rm Im}(\tilde{\phi} \overline{\phi_s}) {\tilde{\phi}}.
\end{align*}
The third is the $Z$ term:
\begin{align*}
{\bf N_3}&:=i\partial_s Z.
\end{align*}
For the quadratic terms, we observe that each term in ${\bf N_1}$ has a decay weight $A^{\infty}$ or $\phi^{\infty}$. Thus the gradient term of $\phi_s$
can be controlled by
Morawetz estimates bootstrap assumption (\ref{om2}). We pick the most troublesome gradient term as the candidate:
\begin{align*}
\|{A^{lin}}\cdot\nabla \phi_s\|_{L^1_tL^2_x}\le \|e^{\frac{1}{2}r}A^{lin}\|_{L^2_t L^4_x}\|e^{-\frac{1}{2}r}\nabla \phi_s\|_{L^2_t L^4_x}\lesssim \theta_{\frac{1}{2}-\delta}(s)\epsilon^{\alpha+\beta}_1.
\end{align*}
Similarly, we have
\begin{align*}
\|A_t\phi_s\|_{L^1_tL^2_x}&\lesssim \|A_t\|_{L^2_tL^4_x}\|\phi_s\|_{L^2_tL^4_x}\lesssim  \theta_{\frac{1}{2}-\delta}(s) \epsilon^{3}_1.\\
\|A^{\infty}A^{lin}\phi_s\|_{L^1_tL^2_x}&\lesssim \|\phi_s\|_{L^2_tL^4_x}\|A^{lin} \|_{L^2_tL^4_x}\lesssim \theta_{\frac{1}{2}-\delta}(s)  \epsilon^{\alpha+1}_1.\\
\|\kappa^{\infty}{\rm Im}(\tilde{\phi}^j \overline{\phi_s}){\phi}^{\infty}_j\|_{L^1_tL^2_x}&\lesssim \|\tilde{\phi}\|_{L^2_tL^4_x}\|\phi_s \|_{L^2_tL^4_x}\lesssim    \theta_{\frac{1}{2}-\delta}(s)\epsilon^{\alpha+1}_1.\\
\|\tilde{\kappa}h^{kj}{\rm Im}(\tilde{\phi}^{\infty}_k\overline{\phi_s}) {\phi}^{\infty}_j\|_{L^1_tL^2_x}&\lesssim \|\tilde{\kappa}\|_{L^2_tL^4_x}\|\phi_s \|_{L^2_tL^4_x}\lesssim  \theta_{\frac{1}{2}-\delta}(s)\epsilon^{\alpha+1}_1,
\end{align*}
where we applied Proposition \ref{XssZ}.

For the cubic or higher order  terms involved in ${\bf N_2}$, if there exists gradient term of $\phi_s$, one uses $\|\nabla\phi_s\|_{L^{\infty}_t L^2_x}$ to control $\nabla \phi_s$ and $L^1_{t}L^{\infty}_x$ norms to control the other quadratic  or  higher order terms. If there exists no $\nabla\phi_s$, one directly uses $L^2_tL^4_x$ to bound two of them and use $L^{\infty}_t L^{\infty}_x$ to dominate the others.
In fact, for the gradient term we have
\begin{align*}
\| {A^{qua}}\nabla\phi_s\|_{L^1_tL^2_x}&\lesssim \|\nabla\phi_s\|_{L^{\infty}_tL^2_x}\|A^{qua} \|_{L^1_tL^{\infty}_x}\lesssim  \theta_{\frac{1}{2}-\delta}(s) \epsilon^{3\alpha}_1.
\end{align*}
And for the other terms we have
\begin{align*}
\| {A^{qua}}A^{qua}\phi_s\|_{L^1_tL^2_x}&\lesssim \|A^{qua}\|_{L^{1}_tL^4_x}\|\phi_s\|_{L^{\infty}_tL^4_x}\|A^{qua} \|_{L^{\infty}_{t,x}}\lesssim   \theta_{\frac{1}{2}-\delta}(s)\epsilon^{3\alpha}_1\\
\| {A^{qua}}A^{lin}\phi_s\|_{L^1_tL^2_x}&\lesssim \|A^{qua}\|_{L^{1}_tL^2_x}\|\phi_s\|\|A^{lin} \|_{L^{\infty}_{t,x}}\lesssim  \theta_{\frac{1}{2}-\delta}(s) \epsilon^{3\alpha}_1\\
\| {A^{qua}}A^{\infty}\phi_s\|_{L^1_tL^2_x}&\lesssim \|A^{qua}\|_{L^{\infty}_{t,x}}\|\phi_s\|_{L^{\infty}_tL^2_x}\|A^{qua} \|_{L^1_tL^{\infty}_x}\lesssim  \theta_{\frac{1}{2}-\delta}(s) \epsilon^{3\alpha}_1\\
\| {A^{lin}}A^{lin}\phi_s\|_{L^1_tL^2_x}&\lesssim \|A^{lin}\|_{L^2_tL^4_x}\|\phi_s\|_{L^{2}_tL^4_x}\|A^{lin} \|_{L^{\infty}_tL^{\infty}_x}\lesssim  \theta_{\frac{1}{2}-\delta}(s) \epsilon^{3\alpha}_1.
\end{align*}
and
\begin{align*}
\| \kappa {\rm Im}(\tilde{\phi}\overline{\phi_s}){\tilde{\phi}}\|_{L^1_tL^2_x}&\lesssim \|\phi_s\|_{L^{2}_tL^4_x}\|\tilde{\phi}\|_{L^2_tL^4_x}\|\tilde{\phi}\|_{L^{\infty}_{t,x}}\lesssim  \theta_{\frac{1}{2}-\delta}(s) \epsilon^{3\alpha}_1\\
\| \tilde{\kappa} {\rm Im}({\phi}^{\infty}\overline{\phi_s}) {\tilde{\phi}}\|_{L^1_tL^2_x}&\lesssim \|\phi_s\|_{L^{2}_tL^4_x}\|\tilde{\kappa}\|_{L^2_tL^4_x}\|\tilde{\phi}\|_{L^{\infty}_{t,x}}\lesssim  \theta_{\frac{1}{2}-\delta}(s) \epsilon^{3\alpha}_1\\
\| \tilde{\kappa} {\rm Im}(\tilde{\phi}\overline{\phi_s}) {{\phi}^{\infty}}\|_{L^1_tL^2_x}&\lesssim \|\phi_s\|_{L^{2}_tL^4_x}\|\tilde{\kappa}\|_{L^2_tL^4_x}\|\tilde{\phi}\|_{L^{\infty}_{t,x}}\lesssim  \theta_{\frac{1}{2}-\delta}(s) \epsilon^{3\alpha}_1\\
\|\tilde{\kappa} {\rm Im}(\tilde{\phi} \overline{\phi_s}) {\tilde{\phi}}\|_{L^1_tL^2_x}&\lesssim \|\phi_s\|_{L^{2}_tL^4_x}\|\tilde{\kappa}\|_{L^2_tL^4_x}\|\tilde{\phi}\|^2_{L^{\infty}_{t,x}}\lesssim  \theta_{\frac{1}{2}-\delta}(s) \epsilon^{3\alpha}_1.
\end{align*}

$ {\bf N}_3$ is direct  by applying (\ref{zq1}).

\end{proof}

\begin{Lemma}\label{X4}
With the above assumption (\ref{om1})-(\ref{om3}),
we have
\begin{align}
\omega_{\frac{1}{2}-\delta}(s) \|\nabla \phi_s\|_{L^{\infty}_tL^2_x([0,T]\times\Bbb H^2)}&\lesssim \epsilon^{\alpha+\frac{1}{2}}_1\label{ppom2}
\end{align}
\end{Lemma}

\begin{proof}
To apply the energy estimates, we use  the following  suitable equation for $\phi_s$
\begin{align}\label{GhKL}
(i\partial_t+\Delta_{A})\phi_s&=
iA_t\phi_s+i\kappa {\rm Im}(\tilde{\phi}\overline{\phi_s}){\tilde{\phi}}
+i\kappa {\rm Im}({\phi}^{\infty}\overline{\phi_s}) {\tilde{\phi}}+i\kappa {\rm Im}(\tilde{\phi}\overline{\phi_s}) {{\phi}^{\infty}}\\
&+i\tilde{\kappa} h^{jk}{\rm Im}({\phi}^{\infty}_k \overline{\phi_s}) {{\phi}^{\infty}_{j}}+i\partial_s Z.+i\partial_s Z.
\end{align}

And energy estimates of Proposition \ref{6.1} give
\begin{align}\label{7.A}
\|\nabla \phi_s\|^2_{L^{\infty}_t L^2_x}&\lesssim \|\nabla \phi_s(s,0,x)\|^2_{L^2_x}+ \||\partial_t A||\phi_s||D_{A}\phi_s|\|_{L^{1}_{t,x}}+ \||D_{A}{\bf F}||D_{A}\phi_s|\|_{L^{1}_{t,x}}
\end{align}
where ${\bf F}$ denotes the RHS of (\ref{GhKL}).
 The initial data term $\|\nabla \phi_s(s,0,x)\|_{L^2_x}$ is admissible by  Lemma \ref{initial}.

{\bf Step 0.}
It is obvious that
\begin{align*}
\||\partial_s Z||D_{A}\phi_s|\|_{L^{1}_t L^1_x}&\lesssim  \|\partial_s Z\|_{L^1_t L^2_x}(\|\nabla\phi_s \|_{L^{\infty}_{t}{L^2_x}}+ \|A\|_{L^{\infty}_{t,x}}\|\nabla\phi_s \|_{L^{\infty}_{t}{L^2_x}})
\end{align*}

{\bf Step 1.}
Let's calculate $\partial_t A$.
In fact, one has
\begin{align*}
\partial_t A_i=\int^{\infty}_s (\partial_t \kappa){\rm Im}(\phi_i\overline{ \phi_s })ds'+\int^{\infty}_s \kappa{\rm Im}(\partial_t\phi_i\overline{ \phi_s}) ds'+\int^{\infty}_s \kappa {\rm Im}(\phi_i{\partial_t\overline{\phi_s}}) ds'.
\end{align*}
Recall that
\begin{align*}
|\partial_t\phi_x|\le |A||\phi_t|+ |\nabla \phi_t|, \mbox{ }|\partial_t\phi_s|\le |A_t||\phi_s|+|\phi_x|^2|\phi_s|+|\Delta_{A}\phi_s|+| \partial_s Z|.
\end{align*}
And $\partial_t  {\kappa}$ is point-wisely bounded as
\begin{align*}
|\partial_t {\kappa}|\lesssim |\phi_t|.
\end{align*}
In a summary, we have
\begin{align*}
|\partial_t A|&\lesssim \int^{\infty}_s|\phi_t||\phi_x||\phi_s|ds' +\int^{\infty}_s  | \phi_s|(|A||\phi_t|+ |\nabla \phi_t|) ds'\\
& +\int^{\infty}_s  |\phi|(|A_t||\phi_s|+|\phi_x|^2|\phi_s|+|\Delta_{A}\phi_s|) ds'+\int^{\infty}_s|\partial_s Z|ds'.
\end{align*}
Let's rearrange these upper-bounds for $\partial_t A $ to be
\begin{align*}
|\partial_t A  |\le I_1+I_2+I_3+I_4,
\end{align*}
where $I_1,I_2,I_3,I_4$ are defined by
\begin{align*}
I_1&:=\int^{\infty}_s|{\phi}^{\infty}||\Delta_{A}\phi_s|ds'\\
I_2&:=\int^{\infty}_s|\phi_t||\phi_x||\phi_s|ds' +\int^{\infty}_s  | \phi_s||A||\phi_t| ds'
+\int^{\infty}_s  |\phi_x|(|A_t||\phi_s|+|\phi_x|^2|\phi_s|) ds'\\
I_3&:=\int^{\infty}_s | \phi_s| |\nabla \phi_t| ds'\\
I_4&:=\int^{\infty}_s | \partial_s Z| ds'.
\end{align*}
We see $I_2$ is at least quadratic.
By interpolation and decay of $\phi^{\infty}$, we get  from Proposition 3.3
\begin{align}
\|e^{\frac{1}{2}r}I_1\|_{L^{\infty}_t L^{\frac{2}{1-\delta}}_x}\lesssim \int^{\infty}_s ({\bf 1}_{\tilde{s}\in (0,1)}(\tilde{s}) \tilde{s}^{-1+\frac{1}{2}\delta}+{\bf 1}_{\tilde{s}\ge 1}(\tilde{s})\tilde{s}^{-L})d\tilde{s}\lesssim \epsilon^{\alpha}_1.\label{V1p}
\end{align}
Meanwhile, by Proposition \ref{XssZ} we have
\begin{align}
\|I_2\|_{L^{2}_{t}L^4_{x}}&\lesssim \int^{\infty}_s\|\phi_t\|_{L^{2}_tL^4_x}\|\phi_s\|_{L^{\infty}_{t,x}}(\|\phi\|_{L^{\infty}_{t,x}}+\|A\|_{L^{\infty}_{t,x}})ds'\nonumber\\
&
+\int^{\infty}_s ( \|\phi\|_{L^{\infty}_{t,x}}\|A_t\|_{L^{\infty}_{t,x}}+\|\phi \|^2_{L^{\infty}_{t,x}})\|\phi_s\|_{L^2_tL^4_x} ds'\nonumber\\
&\lesssim \epsilon^{1+\alpha}_1.\label{V2p}
\end{align}
The $I_3$ term is bounded by
\begin{align}
\|I_3\|_{L^{2}_{t}L^{\frac{4}{2-\delta}}_{x}}&\lesssim \int^{\infty}_s\|\phi_t\|_{L^{2}_tL^{\frac{4}{\delta}}_x}\|\nabla \phi_t\|_{L^{\infty}_{t}L^{\frac{2}{1-\delta}}_x} ds'\lesssim \epsilon^{1+\alpha}_1.\label{V3p}
\end{align}
The $I_3$ term is dominated  by
\begin{align}
\|I_4\|_{L^{2}_{t}L^{4}_{x}}&\lesssim \epsilon^{2\alpha}_1.\label{V4p}
\end{align}

{\bf Step 2.} In this step we bound $\||\partial_t A||\phi_s||D_{A}\phi_s|\|_{L^{1}_{t,x}}$.
By Step 1, we see
\begin{align*}
\||\partial_t A||\phi_s||\nabla \phi_s|\|_{L^{1}_{t,x}}&\lesssim \|e^{\frac{1}{2}r}I_1\|_{L^{\infty}_t L^{\frac{2}{1-\delta}}_x}\|\phi_s\|_{L^2_tL^{\frac{2}{\delta}}_x}\|e^{-\frac{1}{2}r}\nabla \phi_s\|_{L^2_t L^2_x}+ \|I_4\|_{L^{2}_{t}L^4_{x}}\|\phi_s\|_{L^{2}_t L^4_x}\|\nabla \phi_s\|_{L^{\infty}_tL^{2}_x} \\
&+ \|I_2\|_{L^{2}_{t}L^4_{x}}\|\phi_s\|_{L^{2}_t L^4_x}\|\nabla \phi_s\|_{L^{\infty}_tL^{2}_x}
 +\|I_3\|_{L^2_{t}L^{\frac{4}{2-\delta}}_x}\|\phi_s \|_{L^{2}_{t}L^{\frac{4}{\delta}}_x}\|\nabla\phi_s \|_{L^{\infty}_{t}L^2_x}
\end{align*}
Then by (\ref{V1p}) -(\ref{V3p}), we obtain
\begin{align*}
\||\partial_t A||\phi_s||\nabla \phi_s|\|_{L^{1}_{t,x}}\lesssim \theta_{1-2\delta}(s)\epsilon^{1+2\alpha}_1.
\end{align*}
The left $\||\partial_t A||\phi_s||A|| \phi_s|\|_{L^{1}_{t,x}}$ is much easier, and we conclude for this step that
\begin{align*}
\||\partial_t A||\phi_s||D_{A}\phi_s|\|_{L^{1}_{t,x}}\lesssim \theta_{1-2\delta}(s)\epsilon^{1+2\alpha}_1.
\end{align*}

{\bf Step 3.} Now, let's bound   $\||D_{A}{\bf L}||D_{A}\phi_s|\|_{L^{1}_{t,x}}$,
where ${\bf L}$ denotes the RHS of (\ref{GhKL}). Let's consider the top derivative term, i.e.  $\||\nabla{\bf L}||\nabla\phi_s|\|_{L^{1}_{t,x}}$.
As before, we divide ${\bf L}$  into three parts:
\begin{align*}
L_1&:=i\widetilde{\kappa}{\rm Im}h^{kj}({\phi}^{\infty}_j\overline{\phi_s})\phi^{\infty}_k+i\kappa^{\infty}{\rm Im}(\tilde{\phi}^j\overline{\phi_s})\phi^{\infty}_j+i\kappa^{\infty}  h^{kj}{\rm Im}({\phi}^{\infty}_k\overline{\phi_s})\tilde{\phi}_j\\
L_2&:=i\widetilde{\kappa}{\rm Im} (\tilde{\phi}^j\overline{\phi_s})\phi^{\infty}_j+i\widetilde{\kappa}{\rm Im}({\phi}^{\infty}_j\overline{\phi_s})\tilde{\phi}^j
+i\widetilde{\kappa}{\rm Im}(\tilde{\phi}^j\overline{\phi_s})\tilde{\phi}_j+ i{\kappa}^{\infty}{\rm Im}(\tilde{\phi}^j\overline{\phi_s})\tilde{\phi}_j\\
L_3&:=A_t\phi_s.
\end{align*}
The $\||\nabla L_1||\nabla\phi_s|\|_{L^{1}_{t,x}}$ norm is bounded by
\begin{align*}
&\||\nabla L_1 ||\nabla\phi_s|\|_{L^{1}_{t,x}}\lesssim\|e^{\frac{1}{2}r} |\nabla L_1| \|_{L^{2}_{t,x}} \|e^{-\frac{1}{2}r}|\nabla\phi_s|\|_{L^{2}_{t,x}}\\
&\lesssim  \| \nabla  \tilde{\kappa} \|_{L^{\infty}_{t}L^4_{x}}\|\phi_s\|_{L^2_{t}L^{4}_x}\|e^{-\frac{1}{2}r}|\nabla\phi_s|\|_{L^{2}_{t,x}} + \| \nabla  \tilde{\phi} \|_{L^{\infty}_{t}L^{ \frac{4}{2-\delta}}_{x}}\|\phi_s\|_{L^2_{t}L^{\frac{4}{\delta}}_x}\|e^{-\frac{1}{2}r}|\nabla\phi_s|\|_{L^{2}_{t,x}}\\
&+(\| \widetilde{\kappa}\|_{L^{\infty}_{t,x}}+\| \tilde{\phi}\|_{L^{\infty}_{t,x}})\|e^{-\frac{1}{2}r}|\phi^{\infty}||\nabla  \phi_s|\|_{L^2_{t,x}}\|e^{-\frac{1}{2}r}|\nabla\phi_s|\|_{L^{2}_{t,x}}.
\end{align*}
The $\||\nabla L_2||\nabla\phi_s|\|_{L^{1}_{t,x}}$ norm is bounded by
\begin{align*}
&\||\nabla L_2 ||\nabla\phi_s|\|_{L^{1}_{t,x}}\lesssim \|\nabla L_2|\|_{L^1_{t}L^2_x}\|\nabla\phi_s|\|_{L^{\infty}_{t}L^2_x}\\
&\lesssim \|\nabla\tilde{\kappa}\|_{L^{\infty}_{t}L^{\frac{2}{1-\delta}}_x}\|\tilde{\phi}\|_{L^{2}_tL^{\frac{4}{\delta}}_x}
\|\phi_s\|_{L^2_{t}L^{\frac{4}{\delta}}_x}\|\nabla\phi_s\|_{L^{\infty}_{t}L^2_x}
+ \|\tilde{\kappa}\|_{L^2_tL^{\infty}_x}\|\nabla \tilde{\phi}\|_{L^{\infty}_tL^{\frac{4}{2-\delta}}_x}\|\phi_s\|_{L^2_{t}L^{\frac{4}{\delta}}_x}\|\nabla\phi_s \|_{L^{\infty}_{t}L^2_x}\\
&+\|\tilde{\phi}\|_{L^2_{t}L^{\frac{8}{\delta}}_x}\|\nabla \tilde{\phi}\|_{L^{\infty}_tL^{\frac{4}{2-\delta}}_x}\|\phi_s\|_{L^2_{t}L^{\frac{8}{\delta}}_x}\|\nabla\phi_s \|_{L^{\infty}_{t}L^2_x}
+ \|\tilde{\kappa}\|_{L^2_tL^{4}_x}\|  \tilde{\phi}\|_{L^{2}_tL^{4}_x}\|\nabla\phi_s \|^2_{L^{\infty}_{t}L^2_x}\\
&+\|  \tilde{\phi}\|^2_{L^{2}_tL^{4}_x}\|\nabla\phi_s \|^2_{L^{\infty}_{t}L^2_x}.
\end{align*}
The $\||\nabla{L_3}||\nabla\phi_s|\|_{L^{1}_{t,x}}$ norm is dominated  by
\begin{align*}
&\||\nabla L_3 ||\nabla\phi_s|\|_{L^{1}_{t,x}}\lesssim \| A_t\|_{L^1_{t}L^{\infty}_x}\|\nabla\phi_s\|^2_{L^{\infty}_{t}L^{2}_x}+\|\nabla A_t \|_{L^2_{t}L^{\frac{2}{1-\delta}}_x}\|\phi_s\|_{L^2_{t}L^{\frac{2}{\delta}}_x}\|\nabla\phi_s\|_{L^{\infty}_{t}L^{2}_x}
\end{align*}

Hence, by  Lemma 3.7 and Proposition 3.3,  we get
\begin{align*}
\||\nabla{\bf L}||\nabla\phi_s|\|_{L^{1}_{t,x}}\lesssim  \theta_{1-2\delta}(s) \epsilon^{1+2\alpha}_1.
\end{align*}
The other lower derivative order terms in $ \||D_{A}{\bf L}||D_{A}\phi_s|\|_{L^{1}_{t,x}}$ are easier to dominate and we conclude
that
\begin{align*}
 \||D_{A}{\bf L}||D_{A}\phi_s|\|_{L^{1}_{t,x}}\lesssim \theta_{1-2\delta}(s) \epsilon^{1+2\alpha}_1.
\end{align*}
Combining these estimates together, the desired result follows by (\ref{7.A}).
\end{proof}

\begin{Lemma}\label{X5}
With the above assumption (\ref{om1})-(\ref{om3}),
we have
\begin{align}
\sup_{s>0}\omega_{\frac{1}{2}-\delta}(s) \|e^{-\frac{1}{2}r}|\nabla \phi_s|\|_{L^{2}_t L^2_x([0,T]\times\Bbb H^2)}&\le   \epsilon^{\frac{1}{2}(1+\alpha)}_1.\label{ppom2}
\end{align}
\end{Lemma}
\begin{proof}
To  apply
Morawetz  estimates, we use the following  equation of $\phi_s$
\begin{align*}
(i\partial_t+\Delta_{A})\phi_s =A_t\phi_s+i\kappa {\rm Im} (\phi\overline{\phi_s}) \phi+i\partial_s Z.
\end{align*}

Since $\|A \|_{L^{\infty}_{t,x}}\lesssim 1$, Morawetz estimates in Corollary \ref{smo}  give
\begin{align*}
&\|e^{-\frac{1}{2}r}|\nabla \phi_s|\|^2_{L^2_t L^2_x}\\
&\lesssim  \|(-\Delta)^{\frac{1}{4}} \phi_s \|^2_{L^{\infty}_tL^2_x}+\||\nabla A||\phi_s||\nabla\phi_s|\|_{L^1_{t,x}}+\||\phi_s|^2|A|^2 e^{-r}\|_{L^1_{t,x}}+\||\partial_s Z||D_{A}\phi_s\|_{L^1_{t,x}}\\
&+\||A_t\phi_s||D_{A}\phi_s|\|_{L^1_{t,x}}+  \||\kappa {\rm Im} (\phi\overline{\phi_s}) \phi||D_{A}\phi_s|\|_{L^1_{t,x}} +\||\partial_t A| |u|^2\|_{L^{1}_{t,x}}+\||\phi_s|^2A\|_{L^{1}_{t,x}}.
\end{align*}
By interpolation, we have
\begin{align*}
\|(-\Delta)^{\frac{1}{4}} \phi_s \|^2_{L^{\infty}_tL^2_x}\lesssim \|\nabla \phi_s \|_{L^{\infty}_tL^2_x} \| \phi_s \|_{L^{\infty}_tL^2_x}\lesssim \epsilon^{1+\alpha }_1\theta_{1-2\delta}.
\end{align*}
For the other terms,  we have
\begin{align*}
\||\nabla A||\phi_s||\nabla\phi_s|\|_{L^1_{t,x}}&\lesssim \||\nabla A^{\infty}|e^{\frac{1}{2}r}\|_{L^{\infty}_{t}L^4_{x}}\|\phi_s\|_{L^2_tL^4_x}\|e^{-\frac{1}{2}r}|\nabla\phi_s|\|_{L^2_{t,x}}\\
&+\||\nabla\tilde{ A} \|_{L^{2}_{t}L^4_{x}}\|\phi_s\|_{L^2_tL^4_x}\| \nabla\phi_s|\|_{L^{\infty}_{t}L^2_{x}}\\
 \||\phi_s|^2|A|^2 e^{-r}\|_{L^1_{t,x}}&\lesssim\| \phi_s\|^2_{L^2_tL^4_x}\|A\|^2_{L^{\infty}_{t}L^4_{x}}\\
\||A_t\phi_s||D_{A}\phi_s|\|_{L^1_{t,x}}&\lesssim\|A_t\|_{L^2_tL^4_x}\| \phi_s\|_{L^2_tL^4_x}\| \nabla\phi_s|\|_{L^{\infty}_{t}L^2_{x}}+
  \|A\|_{L^{\infty}_{t}L^2_{x}}\|A_t\|_{L^{\infty}_{t,x}}\| \phi_s\|^2_{L^2_tL^4_x}.
\end{align*}
And similar arguments as  Lemma \ref{X4}  yields
\begin{align*}
&\||\kappa {\rm Im} (\phi\overline{\phi_s}) \phi||D_{A}\phi_s|\|_{L^1_{t,x}}\\
&\lesssim
\|\tilde{\phi}\phi\|_{L^{\infty}_tL^{\frac{4}{2-\delta}}_x}\| \phi_s\|_{L^2_tL^{\frac{4}{\delta}}_x}\|e^{-\frac{1}{2}r} |\nabla\phi_s|\|_{L^{2}_{t}L^2_{x}}
+\|\tilde{\phi}\|_{L^2_tL^4_x}\|\phi\|_{L^{\infty}_{t,x}}\| \phi_s\|_{L^2_tL^{4}_x}\| \nabla\phi_s\|_{L^{\infty}_{t}L^2_{x}}.
\end{align*}
Also, we have
\begin{align*}
\||\phi_s|^2A\|_{L^{1}_{t,x}}&\lesssim \|A\|_{L^{\infty}_{t}L^{2}_x}\|\phi_s\|^2_{L^2_{t}L^4_{x}}\\
\|\partial_sZ |D_{A}\phi_s\|_{L^{1}_{t,x}}&\lesssim \||\partial_sZ\|_{L^1_tL^2_x}\|\nabla \phi_s\|_{L^{\infty}_tL^2_x}+\| A\|_{L^{\infty}}\|\phi_s\|_{L^{2}_{t}L^4_{x}}\|\partial_s Z\|_{L^2_tL^4_x}.
\end{align*}
The rest is  the to control
$ \||\partial_t A| |\phi_s|^2\|_{L^{1}_{t,x}}$.
In fact, it is much easier than corresponding $\partial_t A$ term in  Lemma \ref{X4}, since no derivatives hit $\phi_s$.
As a conclusion, we  get (\ref{ppom2}):
\begin{align*}
\|e^{-\frac{1}{2}r}|\nabla \phi_s|\|_{L^{2}_tL^2_x([0,T]\times\Bbb H^2)}\lesssim \theta_{\frac{1}{2}-\delta}(s) \epsilon^{\frac{1}{2} (1+\alpha) }_1.
\end{align*}
\end{proof}

{\bf End of proof of Proposition \ref{MM2}.}

Hence, Proposition \ref{MM2} follows by bootstrap and Lemma \ref{X33}, Lemma \ref{X4}, Lemma \ref{X5}.

\section{Proof of Theorem 1.1, Part II. Asymptotic behaviors  }

\subsection{Global Existence}\label{aa3}

The following Lemma \ref{Rgh} states a blow-up criterion for SL. Its analogy form is well-known in 2D  heat flows. The SL analogy is also known, and without claiming any originality we gave  a detailed proof in the case $\mathcal{M}=\Bbb R^2$  in Appendix B of our previous work \cite{Li4}. Since the proof in \cite{Li4} is purely an energy argument, it is  easy to give a parallel proof for $\mathcal{M}=\Bbb H^2$.

\begin{Lemma}\label{Rgh}
Let $u\in C((-T,T);\mathcal{H}^3_{Q})$ solve SL, and assume that  there exists some constant $C>0$ such that
\begin{align*}
\|\nabla u\|_{L^{\infty}_{t,x}((-T,T)\times\Bbb H^2)}\le C.
\end{align*}
Then there exists some $\rho>0$ such that $u\in C([-T-\rho,T+\rho];\mathcal{H}^3_{Q})$.
\end{Lemma}

Now, let's use Lemma \ref{Rgh} to show $u$ is global. Assume that $u\in C([0,T_*);\mathcal{H}^3_{Q})$, and $T_*>0$ is the maximum lifespan. Proposition \ref{MM2} has shown all the bounds obtained in Section 3 hold uniformly on $t\in [0,T_*)$. Particularly,  (\ref{3.52}) holds, and thus
\begin{align}\label{Uijm}
 \int^{\infty}_0\|\nabla\phi_s\|_{L^{\infty}_{t,x}([0,T_*)\times\Bbb H^2)}ds'\le C,
\end{align}
for some $C>0$.

Recall that
\begin{align}
\phi_i(0,t,x)=\phi^{\infty}_i-\int^{\infty}_0(\partial_i\phi_s+iA_i\phi_s) ds'.
\end{align}
Then (\ref{Uijm}) shows
\begin{align*}
\|\nabla u\|_{L^{\infty}_{t,x}([0,T_*)\times\Bbb H^2)}\le C,
\end{align*}
which by Lemma \ref{Rgh} implies that $u$ can be continuously extended beyond $T_*$, thus contradicting with the maximum of $T_*$. Therefore, $u$ is global and belongs to $C([-T,T];\mathcal{H}^k_Q)$ for any $T>0,k\in \Bbb Z_+$, as long as $u_0$ verifies Theorem 1.1 and $u_0\in \mathcal{H}^{k}_Q$.

\subsection{Convergence to harmonic map in $L^{\infty}$}\label{aa2}

This is now standard by applying
\begin{align}
\|u-Q\|_{L^{\infty}_x}\le \int^{\infty}_0\| \phi_s\|_{L^{\infty}_{x}}ds'.
\end{align}
It suffices to show
\begin{align}\label{harr}
\lim_{t\to\infty}  \int^{\infty}_0\| \phi_s\|_{L^{\infty}_{x}}ds'=0
\end{align}
One may see \cite{LOS,LLOS2,Li2} for details.

\subsection{Resolution in energy spaces}\label{aa4}

In this part, we prove the resolution of SL solutions  in energy spaces claimed in Theorem 1.1. The proof is divided into three parts.

\begin{Lemma}\label{4.1a}
We have
\begin{align}
\sup_{s>0}\omega_{1-\delta}(s)\|\nabla^3 Z\|_{L^1_tL^2_x}\lesssim 1\label{Ik2}\\
\sup_{s>0}\omega_{1-\delta}(s)\|e^{-\frac{1}{2}r}\nabla^2 \phi_s\|_{L^2_tL^2_x}\lesssim 1.\label{Ik3}
\end{align}
\end{Lemma}
\begin{proof}
(\ref{Ik2}) follows by the same arguments as Lemma 3.6, if one has proved
\begin{align}\label{Bhnj}
 \|\nabla^3 A\|_{L^{\infty}_{t,x}}\lesssim \max(s^{-1},1).
\end{align}
To prove (\ref{Bhnj}), it suffices to apply Lemma 10.3 in Appendix B and formulas like (\ref{pad}).

Next, let's prove (\ref{Ik3}). We introduce two cutoff functions. Let $\vartheta_{\mu_1}:\Bbb H^2\to [0,1]$ be a smooth function which equals to 1 for $r>3\mu_1$ and vanishes for $0\le r\le \mu_1$. Let   $\nu_{\mu_2}:\Bbb H^2\to [0,1]$ be a smooth function which vanishes  for $r>3\mu_2$ and  equals to   1 for $0\le r\le \mu_2$.

Then $e^{-\frac{1}{2}r}\vartheta_{\mu_1}\phi_s$ satisfies
\begin{align*}
(\partial_s -\Delta)[e^{-\frac{1}{2}r}\vartheta_{\mu_1}\phi_s]=-2ie^{-\frac{1}{2}r}\vartheta_{\mu_1}A^{\infty}\cdot \nabla \phi_s-2\nabla (e^{-\frac{1}{2}r}\vartheta_{\mu_1})\cdot \nabla \phi_s+L,
\end{align*}
where $L$  denotes terms involved with zero order derivative of $\phi_s$.
We remark that  since $\vartheta_{\mu_1}$ is supported in $r>2\mu_1$, $e^{-\frac{1}{2}r}$ is now smooth.
Then by  smoothing estimates  for heat equations and the equivalence relation (\ref{equi}), we get
\begin{align*}
&\|\nabla^2(e^{-\frac{1}{2}r}\vartheta_{\mu_1}\phi_s)\|_{L^2_{t,x}}\lesssim \Theta_{\frac{1}{2}}(s,\frac{s}{2})\| \nabla (e^{-\frac{1}{2}r}\vartheta_{\mu_1}\phi_s(\frac{s}{2},t,x) )\|_{L^2_{t,x}}+C_{\mu_1}\int^{s}_{\frac{s}{2}}\Theta_{\frac{1}{2}}(s,s')\|\nabla^2\phi_s\|_{L^2_{t}{L^4_x}}ds'\\
&+C_{\mu_1}\int^{s}_{\frac{s}{2}}\Theta_{\frac{1}{2}}(s,s')(\|\nabla \phi_s \|_{L^2_{t}{L^4_x}}+\| \phi_s \|_{L^2_{t}{L^4_x}})ds',
\end{align*}
which by Proposition 3.3, further gives
\begin{align*}
\sup_{s>0}\omega_{1-\delta}(s)\|\nabla^2(e^{-\frac{1}{2}r}\vartheta_{\mu_1}\phi_s)\|_{L^2_{t,x}}\lesssim C_{\mu_1}.
\end{align*}
This,  together with chain rule and previous bounds
\begin{align*}
\sup_{s>0}\omega_{\frac{1}{2}-\delta}(s)(\|e^{-\frac{1}{2}r}\nabla \phi_s\|_{L^2_{t,x}}+\| \phi_s\|_{L^2_{t}L^4_{x}})\lesssim 1,
\end{align*}
indeed provides out ball estimates for (\ref{Ik3}). It remains to consider bounds near the origin. Since $e^{-r}\sim 1$ for $r\in[0,1]$, it suffices to bound
$\|\nu_{\mu_2} \nabla^2\phi_s\|_{L^2_{t,x}}$. Then by the same arguments as above, one has
\begin{align*}
\sup_{s>0}\omega_{1-\delta}(s)(\|\nu_{\mu_2} \nabla^2\phi_s\|_{L^2_{t,x}}+\| \phi_s\|_{L^2_{t}L^4_{x}})\lesssim C_{\mu_2}.
\end{align*}
Thus, (\ref{Ik3}) has been verified.

\end{proof}

\begin{Lemma}\label{Aoosd}
There exist functions  $\vec{v}_l,\vec{v}_2:\mathcal{M}\to \Bbb R^{N} $ and $f :{\mathcal M}\to \Bbb C^n$,  such that
\begin{align}
&\lim_{t\to\infty}\| \nabla u- \nabla Q-\vec{v}_1{\rm Re} (e^{it\Delta}\nabla f )-{\vec{v}_2}{\rm Im}(e^{it\Delta}\nabla f )\|_{L^2_x}=0 \label{aPp12}\\
&\vec{v}_1\bot \vec{v}_2, \mbox{ }|\vec{v}_1|=|\vec{v}_2|, \mbox{ }f \in H^1.\nonumber
\end{align}
where we view $u$ and $Q$ as maps into $\Bbb R^{N}$.
\end{Lemma}
\begin{proof}

Recall that $\phi_s$ satisfies
\begin{align*}
i\partial_t\phi_s+{\bf H}\phi_s={\bf N}
\end{align*}
Applying $(-{\bf H})^{\frac{1}{2}}$ to above equation and using Duhamel principle, we see
\begin{align*}
e^{-it{\bf H}}(-{\bf H})^{\frac{1}{2}}\phi_s (t)=(-{\bf H})^{\frac{1}{2}}\phi_s(0)-i\int^{t}_0e^{-i\tau{\bf H}}(-{\bf H})^{\frac{1}{2}}{\bf N}d\tau.
\end{align*}
Using Lemma \ref{4.1a} and Proposition 3.3, one can verify
\begin{align*}
\sup_{s>0}\omega_{1-\delta}(s)\|(-{\bf H})^{\frac{1}{2}}{\bf N}\|_{L^1_tL^2_x}\lesssim\sup_{s>0}\omega_{1-\delta}(s)\|\nabla {\bf N}\|_{L^1_tL^2_x}\lesssim 1.
\end{align*}
Let
\begin{align*}
h_s:=(-{\bf H})^{\frac{1}{2}}\phi_s(s,0,x)-i\int^{\infty}_0e^{-\tau {\bf H}}(-{\bf H})^{\frac{1}{2}}{\bf N}d\tau.
\end{align*}
Then there holds
\begin{align}\label{Xab}
\sup_{s>0}\omega_{1-\delta}(s)\|h_s\|_{L^{2}_x}\lesssim 1.
\end{align}
Thus we arrive at
\begin{align}\label{Hnmpok}
\lim_{t\to\infty}\sup_{s>0}\omega_{1-\delta}\|(-{\bf H})^{\frac{1}{2}}\phi_s (t)-e^{it{\bf H}}h_s\|_{L^2_x}=0.
\end{align}
Set
\begin{align*}
g_s:=(-{\bf H})^{-\frac{1}{2}}h_s.
\end{align*}
Then, since $(-{\bf H})^{-\frac{1}{2}}$ is bounded from $L^2_x$ to $H^1_x$, we  get from (\ref{Xab}) that
\begin{align}\label{Xab2}
\sup_{s>0}\omega_{1-\delta}(s)\|g_s\|_{H^{1}_x}\lesssim 1.
\end{align}
Thus applying the equivalence relation $\|\nabla f\|_{L^2_x}\sim\|(-{\bf H})^{\frac{1}{2}}\|_{L^2_x}$ gives
\begin{align*}
\|\nabla(\phi_s-e^{it{\bf H}} g_s)\|_{L^2_x}\lesssim  \|(-{\bf H})^{\frac{1}{2}}(\phi_s-e^{it{\bf H}} g_s)\|_{L^2_x}.
\end{align*}
And one infers from (\ref{Hnmpok}) that
\begin{align}\label{Hmmnm}
\lim_{t\to\infty}\sup_{s>0} \omega_{1-\delta}(s)\|\nabla(\phi_s-e^{it{\bf H}} g_s)\|_{L^2_x}=0.
\end{align}

Recall that
\begin{align*}
\phi_i=\phi^{\infty}_i-\int^{\infty}_s(\partial_{i}\phi_s +iA_i\phi_s)d s.
\end{align*}
By  $|A|\in L^2_x$, we deduce from (\ref{harr}) that
\begin{align*}
\lim_{t\to\infty}\int^{\infty}_0 \|A||\phi_s|\|_{L^2_x}d s'=0.
\end{align*}
Thus (\ref{Hmmnm}) implies
\begin{align*}
\lim_{t\to\infty}\|\phi-\phi^{\infty} - \nabla \int^{\infty}_0e^{it{\bf H}}g_s d s\|_{L^2_x}=0.
\end{align*}
Moving $e^{it{\bf H}}$ outside of the integral (it is reasonable by (\ref{Xab2})), we see
\begin{align*}
\lim_{t\to\infty}\|\phi-\phi^{\infty} - \nabla  e^{it{\bf H}} \int^{\infty}_0 g_s d s\|_{L^2_x}=0.
\end{align*}
Set
\begin{align*}
 {f}_1= \int^{\infty}_0 g_s d s'.
\end{align*}
By Lemma \ref{lxxs}, there exists a function ${f}\in H^1$ such that
\begin{align*}
\lim_{t\to\infty}\|\nabla  [e^{it{\bf H}} {f}_1-e^{it\Delta}f]\|_{L^2_x}=0.
\end{align*}
Hence, we have arrived at
\begin{align}\label{Sdfhj}
\lim_{t\to\infty}\|\phi-\phi^{\infty} -  e^{it{\Delta}} \nabla f\|_{L^2_x}=0.
\end{align}

Recall that
\begin{align}
\partial_j u&= {\rm Re}(\phi^1_j) e_{1}+ {\rm Im}(\phi^2_j) e_{2}\label{Jfgv},
\end{align}
and $\mathcal{P}$ denotes the isometric embedding  of $\tilde{N}$ into $\Bbb R^{N}$. By caloric gauge condition,
\begin{align}
|d\mathcal{P}e_{j}-d\mathcal{P}e^{\infty}_j|&\lesssim \int^{\infty}_0 |\phi_s|ds'.\label{Xfgv}
\end{align}
Then we see
\begin{align}
&\|d\mathcal{P}\nabla u-d\mathcal{P}[ {\rm Re}(\phi^1) e^{\infty}_{1}+ {\rm Im}(\phi^2) e^{\infty}_{2}]- d\mathcal{P}[{\rm Re}( e^{it{\Delta}} \nabla f)  e^{\infty}_1+{\rm Im}( e^{it{\Delta}} \nabla f)e^{\infty}_2]\|_{L^2_x}\nonumber\\
&\lesssim \|\phi-\phi^{\infty} -  e^{it{\Delta}} \nabla f\|_{L^2_x}+\| |\phi^{\infty} +  e^{it{\Delta}} \nabla f|| d\mathcal{P}e -d\mathcal{P}e^{\infty} |\|_{L^2_x}\label{3Sdfhj}\\
&+\| |\phi-\phi^{\infty} -  e^{it{\Delta}} \nabla f|| d\mathcal{P}e -d\mathcal{P}e^{\infty} |\|_{L^2_x}\label{2Sdfhj}
\end{align}
By (\ref{Sdfhj}), it is direct to see the first term in the RHS of (\ref{3Sdfhj}) and the term in the RHS of (\ref{2Sdfhj}) tend to zero as $t\to\infty$.

Let
\begin{align}\label{Xcvbddnm}
\vec{v}_1=d\mathcal{P}(e^{\infty}_1), \mbox{ } \vec{v}_2=d\mathcal{P}(e^{\infty}_2).
\end{align}
To finish our proof, it suffices to prove
\begin{align}\label{jxxnmkl}
\lim_{t\to \infty}\| |\phi^{\infty} +  e^{it{\Delta}} \nabla f|| d\mathcal{P}e -d\mathcal{P}e^{\infty} |\|_{L^2_x}=0.
\end{align}
In fact,  (\ref{Xfgv}) and H\"older inequality give
\begin{align*}
&\| |\phi^{\infty} +  e^{it{\Delta}} \nabla f|| d\mathcal{P}e -d\mathcal{P}e^{\infty} |\|_{L^2_x}\lesssim (\| |\phi^{\infty} \|_{L^2_x}+  \|\nabla f\|_{L^2_x})\| d\mathcal{P}e -d\mathcal{P}e^{\infty} \|_{L^{\infty}_x}\\
&\lesssim \int^{\infty}_0\|\phi_s\|_{L^{\infty}_x}ds'.
\end{align*}
Then (\ref{jxxnmkl}) follows by (\ref{harr}).

Therefore, (\ref{aPp12}) follows and we have accomplished the proof.

\end{proof}

\begin{Lemma}
There exist  functions   $f^1_{+},f^2_+:{\mathcal M}\to \Bbb C^n$ belonging to $H^1$,  such that
\begin{align}
&\lim_{t\to\infty}\|   u-  Q- {\rm Re} (e^{it\Delta}f^1_+)- {\rm Im}(e^{it\Delta}f^2_+)\|_{H^1_x}=0. \label{aPvv23}
\end{align}
where we view $u$ and $Q$ as maps into $\Bbb R^{N}$.
\end{Lemma}
\begin{proof}
By Lemma \ref{Aoosd},
\begin{align}
\lim_{t\to\infty}\| \nabla u- \nabla Q-\vec{v}_1{\rm Re} (e^{it\Delta}\nabla f )-{\vec{v}_2}{\rm Im}(e^{it\Delta}\nabla f )\|_{L^2_x}=0 \label{acP22}.
\end{align}
where $\vec{v}_1,\vec{v}_2$ are given by (\ref{Xcvbddnm}).  By density argument and dispersive estimates, we infer from
\begin{align*}
\|  \nabla \vec{v}_1 \|_{L^4_x}+ \|  \nabla \vec{v}_2 \|_{L^4_x}\lesssim \|\phi^{\infty}_x\|_{L^4_x}+\|A^{\infty}\|_{L^4_x}\lesssim 1,
\end{align*}
that there holds
\begin{align}
\lim_{t\to\infty}\sum^{2}_{j=1}\|  |\nabla \vec{v}_j|   |e^{it\Delta}  f| \|_{L^2_x}=0 \label{aP32}.
\end{align}
Hence (\ref{acP22}) and (\ref{aP32}) give
\begin{align}
\lim_{t\to\infty}\|   u-   Q-\vec{v}_1{\rm Re} (e^{it\Delta} f )-{\vec{v}_2}{\rm Im}(e^{it\Delta} f )\|_{H^1_x}=0 \label{accP23}.
\end{align}

We claim that there exist functions $f^1_+,f^2_+:\mathcal{M}\to\Bbb C^n$ belonging to $H^1$ such that
\begin{align}
\lim_{t\to\infty}\|   \vec{v}_1 e^{it\Delta}  f-e^{it\Delta}f^1_+  \|_{H^1_x}=0 \label{afv1}\\
\lim_{t\to\infty}\|   \vec{v}_2 e^{it\Delta}   f-e^{it\Delta}f^2_+  \|_{H^1_x}=0. \label{afv2}
\end{align}
If (\ref{afv1})-(\ref{afv2}) are done, the desired result (\ref{aPvv23}) follows by  (\ref{accP23}).
We will prove (\ref{afv1}) in detail  and (\ref{afv2}) follows by the same way.

Now, let's prove (\ref{afv1}).
$ \vec{v}_1 e^{it\Delta}  f$ solves the equation
\begin{align*}
[i\partial_t +\Delta]( \vec{v}_1 e^{it\Delta}  f)=(\Delta \vec{v}_1)e^{it\Delta}  f+2\nabla  \vec{v}_1\cdot \nabla e^{it\Delta}  f.
\end{align*}
Thus Duhamel principle yields
\begin{align*}
 e^{-it\Delta}\vec{v}_1 e^{it\Delta}  f= \vec{v}_1f-i\int^{t}_0e^{ -i\tau\Delta}[(\Delta \vec{v}_1)e^{i\tau\Delta}  f+2\nabla  \vec{v}_1\cdot   \nabla  e^{i\tau\Delta} f]d\tau.
\end{align*}
By endpoint Strichartz estimates of $e^{it\Delta}$ and the equivalence relation  $\|(-\Delta)^{\frac{1}{2}}g\|_{L^2}\sim\|\nabla g\|_{L^2}$, we deduce that
\begin{align*}
&\|(-\Delta)^{\frac{1}{2}}[\int^{t_2}_{t_1}e^{ -i\tau\Delta}[(\Delta \vec{v}_1)e^{i\tau\Delta}  f+2\nabla  \vec{v}_1\cdot   \nabla  e^{i\tau\Delta} f]d\tau\|_{L^2_x}\\
&\lesssim \| \nabla [(\Delta \vec{v}_1)e^{i\tau\Delta}  f+2\nabla  \vec{v}_1\cdot   \nabla  e^{i\tau\Delta} f]\|_{L^2_{\tau}L^{\frac{4}{3}}_x}\\
&\lesssim \sum^{3}_{j=1}\|  \nabla^{j}\vec{v}_1\|_{L^2_x}\sum^{1}_{k=0} \|\nabla^{k} e^{i\tau\Delta}  f\|_{L^2_{\tau}L^4_x} \\
&\lesssim \|f\|_{H^1_x}
\end{align*}
where the integral domains are $(\tau,x)\in [t_1,t_2]\times\Bbb H^2$.
Hence, there exists a function $g^1_+:\mathcal{M}\to\mathcal{N}$ belonging to $L^2$ such that
\begin{align*}
\lim_{t\to\infty}\|(-\Delta)^{\frac{1}{2}} (\vec{v}_1 e^{it\Delta}  f) -e^{-it\Delta}g^1_+\|_{L^2_x}=0.
\end{align*}
Letting $f^+_1=(-\Delta)^{-\frac{1}{2}}g^1_+$, using Sobolev embedding  we obtain (\ref{afv1}).
\end{proof}

{\bf End of Proof to Theorem 1.1}

Until now, we have proved SL with initial data $u_0$ evolves to a global solution and the two  convergence results (\ref{aP11}) and (\ref{aP12}) hold.

\subsection{Linear Scattering Theory}

We now prove linear scattering results for ${\bf H}$, which   will be divided into three lemmas.

\begin{Lemma}\label{ls}(Scattering in $L^2$)
Let $f\in L^2(\Bbb H^2)$, then there exists  a function  $g\in L^2(\Bbb H^2)$ such that
\begin{align}
&\lim_{t\to\infty}\|e^{it{\bf H}}f-e^{it\Delta}g \|_{L^2_x}=0.
\end{align}
\end{Lemma}
\begin{proof}
By a density argument, it suffices to consider $f\in H^1$.
 The function
$e^{-it\Delta} e^{it{\bf H}}f$ can be written as
\begin{align*}
e^{-it\Delta} e^{it{\bf H}}f=f-i\int^{t}_0e^{-i\tau\Delta}(2iA^{\infty}\nabla -A^{\infty}A^{\infty}+id^*A^{\infty}+\kappa^{\infty}|\phi^{\infty}|^2)(e^{i\tau{\bf H}}f)d\tau.
\end{align*}
By endpoint Strichartz estimates for $e^{it\Delta}$, one has
\begin{align}
&\|\int^{t_2}_{t_1}e^{-i\tau\Delta}(2iA^{\infty}\nabla -A^{\infty}A^{\infty}+id^*A^{\infty}+\kappa^{\infty}|\phi^{\infty}|^2)(e^{i\tau{\bf H}}f)d\tau\|_{L^2_x}\label{Bhgmm}\\
&\lesssim
\|2iA^{\infty}\nabla -A^{\infty}A^{\infty}+id^*A^{\infty}+\kappa^{\infty}|\phi^{\infty}|^2)(e^{i\tau{\bf H}}f)\|_{L^2_{\tau}L^{\frac{4}{3}}_x}\nonumber\\
&\lesssim \|e^{-\alpha'r}|\nabla  (e^{i\tau{\bf H}}f)|\|_{L^2_{\tau}L^{2}_x}\|e^{\alpha'r}A^{\infty}\|_{L^{\infty}_{\tau}L^4_x}\nonumber\\
&+\|e^{-\alpha'r}|e^{i\tau{\bf H}}f|\|_{L^2_{\tau}L^{2}_x}\left(\|e^{\alpha'r}|A^{\infty}|^2\|_{L^{\infty}_{\tau}L^4_x}
+\|e^{\alpha'r}|d^*A^{\infty}| \|_{L^{\infty}_{\tau}L^4_x}+\|e^{\alpha'r}|\phi^{\infty}|^2 \|_{L^{\infty}_{\tau}L^4_x}\right)\nonumber
\end{align}
where we take $0<\alpha'\ll 1$, and all the integral domains in the involved  norms  are $[t_1,t_2]\times\Bbb H^2$.
Then since $f\in H^1$, we get by Lemma 7.2  that  (\ref{Bhgmm})  converges to 0 as $t_2\ge t_1\to\infty$.
Hence, $e^{-it\Delta} e^{it{\bf H}}f$  converges  in $L^2$ as $t\to\infty$, and thus setting
$$g:=\lim_{t\to\infty}e^{-it\Delta} e^{it{\bf H}}f
$$
suffices for the desired result.
\end{proof}

\begin{Lemma}\label{2s}(Asymptotic vanishing)
Let $f\in H^1(\Bbb H^2)$. We have
\begin{align}
&\lim_{t\to\infty}\|[(-{\bf  H})^{\frac{1}{2}} -(-\Delta)^{\frac{1}{2}}]e^{it\Delta}f\|_{L^2_x}=0.\label{Cv1}\\
&\lim_{t\to\infty}\|[(-{\bf  H})^{-\frac{1}{2}} -(-\Delta)^{-\frac{1}{2}}]e^{it\Delta}f\|_{H^1_x}=0.\label{Cv2}
\end{align}
\end{Lemma}
\begin{proof}
Recall that
\begin{align}
{\bf H}=\Delta+V+X
\end{align}
where $V$ denotes the electric potential and $X$ denotes the magnetic field.
Recall the resolvent identity
\begin{align} \label{ASD4}
(-{\bf H}+\lambda)^{-1}-(-\Delta+\lambda)^{-1}=(-{\bf H}+\lambda)^{-1}(V+X )(-\Delta+\lambda)^{-1},
\end{align}
and the Balakrishnan formula for self-adjoint and non-negative operators, i.e.
\begin{align*}
T^{\frac{1}{2}}h:=c_1\int^{\infty}_0\lambda^{-\frac{1}{2}}(T+\lambda)^{-1}T h d\lambda,
\end{align*}
Then direct calculations give
\begin{align}\label{Hgnnn}
(-{\bf  H})^{\frac{1}{2}}h -(-\Delta)^{\frac{1}{2}}h=-c_1\int^{\infty}_0\lambda^{\frac{1}{2}}({\bf H}+\lambda)^{-1}(V+X)(-\Delta +\lambda) h d\lambda.
\end{align}
Thus applying resolvent estimates in Lemma \ref{XCZ} and   H\"older  inequality yields
\begin{align*}
&\|[ (-{\bf  H})^{\frac{1}{2}}-(-\Delta)^{\frac{1}{2}}]h\|_{L^2_x}\nonumber\\
&\lesssim \int^{\infty}_0\|\lambda^{\frac{1}{2}}({\bf H}+\lambda)^{-1}\|_{L^2\to L^2}\|V\|_{L^4_x} \|  (-\Delta_A+\lambda)  \|_{L^4_x\to L^4_x} \|h\|_{L^4_x}
d\lambda\\
&+\int^{\infty}_0\|\lambda^{\frac{1}{2}}({\bf H}+\lambda)^{-1}\|_{L^2\to L^2}\|X\|_{L^4_x}\|\nabla (-\Delta_A+\lambda)  \|_{L^4_x\to L^4_x}\|h\|_{L^4_x}d\lambda\\
&\lesssim \|h\|_{L^4_x}.
\end{align*}
Hence, we get
\begin{align}
\|[(-{\bf  H})^{\frac{1}{2}} -(-\Delta)^{\frac{1}{2}}]e^{it\Delta}f\|_{L^2_x} \lesssim \|e^{it\Delta}f\|_{L^{4}_x}.\label{ASD5}
\end{align}
Also, using $\|X\|_{L^{\infty}}+\|V\|_{L^{\infty}}\lesssim 1$, we have
\begin{align}
\|[(-{\bf  H})^{\frac{1}{2}} -(-\Delta)^{\frac{1}{2}}]e^{it\Delta}f\|_{L^2_x} \lesssim \|e^{it\Delta}f\|_{L^{2}_x}.\label{AsspSD5}
\end{align}

Now, let's prove (\ref{Cv1}) by density argument. Assume that $f_n\in C^{\infty}_c(\Bbb H^2)$ are a sequence of functions such that $f_n\to f$ in $H^1$. Then
(\ref{ASD5})and (\ref{AsspSD5})  show
\begin{align}
 \|[ (-{\bf  H})^{\frac{1}{2}}-(-\Delta)^{\frac{1}{2}}]e^{it\Delta}f\|_{L^2_x}&\lesssim\|[ (-{\bf  H})^{\frac{1}{2}}-(-\Delta)^{\frac{1}{2}}]e^{it\Delta}(f_n-f)\|_{L^2_x}\\
&+
\|[ (-{\bf  H})^{\frac{1}{2}}-(-\Delta)^{\frac{1}{2}}]e^{it\Delta}(f_n)\|_{L^2_x}\\
& \lesssim \|f_n-f\|_{L^2_x}+\|e^{it\Delta}(f_n)\|_{L^{4}_x}.\label{yyASD5}
\end{align}
Then for any small constant $\eta>0$, fixing a sufficiently large  $n$ such that the first term in the RHS is smaller than half of $\eta$ and then using dispersive estimates for the second term on the RHS implies that  (\ref{yyASD5})  tends to zero as $t\to\infty$.

(\ref{Cv2}) follows by the same way with (\ref{Hgnnn}) replaced by
\begin{align*}
(-{\bf  H})^{-\frac{1}{2}}h -(-\Delta)^{-\frac{1}{2}}h=c_2\int^{\infty}_0\lambda^{-\frac{1}{2}}({\bf H}+\lambda)^{-1}(V+X )(-\Delta+\lambda) h d\lambda.
\end{align*}

\end{proof}

Now, we are ready to prove the scattering in  $H^1$.

\begin{Lemma}\label{lxxs}
Let $f\in H^1(\Bbb H^2)$, then there exists a  function  $g \in H^1(\Bbb H^2)$ such that
\begin{align}
&\lim_{t\to\infty}\|e^{it{\bf H}}f-e^{it\Delta}g \|_{H^1_x}=0.\label{xip2}
\end{align}
\end{Lemma}
\begin{proof}
Given $f\in H^1$, applying Lemma \ref{ls} to $(-{\bf H})^{\frac{1}{2}}f$ shows there exists a function  ${g}_1\in L^2_x$ such that
\begin{align}\label{ryu3}
\lim_{t\to\infty}\|e^{it{\Delta }} {g}_1-e^{it{\bf H}}(-{\bf H})^{\frac{1}{2}}f\|_{L^2_x}=0.
\end{align}
Since $(-{\bf H})^{-\frac{1}{2}}$ is bounded from $L^2$ to $H^1$ (see (\ref{WW})), we have
\begin{align*}
\lim_{t\to\infty}\|(-{\bf H})^{-\frac{1}{2}}e^{it{\Delta }} {g}_1-e^{it{\bf H}}f\|_{H^1_x}=0.
\end{align*}
By (\ref{Cv2}) of Lemma \ref{2s}, one also has
\begin{align}\label{ryu4}
\lim_{t\to\infty}\|[(-{\bf H})^{-\frac{1}{2}}-(-\Delta)^{-\frac{1}{2}}]e^{it\Delta} {g}_1\|_{H^1_x}=0.
\end{align}
Thus (\ref{ryu3}) and (\ref{ryu4}) yield
\begin{align*}
\lim_{t\to\infty}\|e^{it{\Delta }}(-\Delta)^{-\frac{1}{2}} {g}_1-e^{it\Delta}f\|_{H^1_x}=0.
\end{align*}
Set
\begin{align*}
g :=(-{ \Delta})^{-\frac{1}{2}}{g}_1,
\end{align*}
then (\ref{xip2}) holds.

\end{proof}

\section{Resolvent estimates and Equivalent norms }

Recall that  the resolvent of ${\bf{H}}$ is denoted by
\begin{align*}
R(\lambda,{\bf{H}}):=(-\bf{H}+\lambda)^{-1}.
\end{align*}

By applying the results in our previous paper \cite{Li3}, we obtain the following estimates for ${\bf H}$.

\begin{Proposition}\label{QNM}
\begin{itemize}

\item (Uniform Resolvent Estimates for ${\bf H}$)
Fixing $0<\alpha'\ll 1$, for any $\lambda\in \Bbb R$, there holds
\begin{align}\label{FxxgH}
\mathop {\sup }\limits_{0 < \varepsilon  < 1} \| \rho^{\alpha'}R(\lambda+ i\varepsilon, {\bf H})  \rho^{\alpha'}\|_{L^2\to L^2}\lesssim 1.
\end{align}
\item The operator {\bf H} is self-adjoint in $L^2$£¬ and has only absolutely continuous  spectrum, i.e. $\sigma({\bf H})=\sigma_{ac}({\bf H})$.
\item (Resolvent Estimates for $\lambda\ge 0$)
 The operator ${\bf H}$ satisfies the resolvent estimates stated in Lemma \ref{XCZ} below.

\item (Smoothing Estimates for $e^{s{\bf H}}$) For any $1<q<p\le \infty$, $\gamma\in[0,1)$, and any $s>0$, we have for some $\delta_q >0$
\begin{align}\label{kn16}
\|(-\Delta)^{ \gamma}e^{s{\bf H}}f\|_{L^p_x}&\le Cs^{-\gamma}s^{\frac{1}{p}-\frac{1}{q} }e^{-\delta_q s}\|f\|_{L^q_x}.
\end{align}
\item (Exact equivalence in $L^p$ and weighted $L^2$) Let $1<p<\infty$, $0<\gamma<2$, $0<\alpha'\ll 1$, then we have
 \begin{align}
 \|(-{\bf H})^{-\frac{\gamma}{2}}f\|_{L^p_x}&\lesssim \|f\|_{L^p_x}.\label{w4ib}\\
\|(-{\bf H})^{-\frac{\gamma}{2}}f\|_{\rho^{-\alpha'}L^2_x}&\lesssim \|f\|_{\rho^{-\alpha'}L^2_x}\label{w5ib}\\
\|(-\Delta)^{\frac{\gamma}{2}}f\|_{L^p_x}&\lesssim \|(-{\bf H})^{\frac{\gamma}{2}}f\|_{L^p_x}\label{w2ib}\\
\|(-{\bf H})^{\frac{\gamma}{2}}f\|_{L^p_x}&\lesssim \|(-\Delta)^{\frac{\gamma}{2}}f\|_{L^p_x}.\label{w3ib}
 \end{align}
 \item (Sobolev embedding of ${\bf H}$) For  $0<\gamma<2$, we have
 \begin{align}
\|(-{\bf H})^{-\frac{\gamma}{2}}f\|_{H^{\gamma}_x}&\lesssim \|f\|_{L^2_x}.\label{WW}
 \end{align}
\end{itemize}
\end{Proposition}
\begin{proof}
(\ref{FxxgH}) follows by  our previous paper \cite{Li3}. We remark that in  \cite{Li3} we use Coulomb gauge for $Q^*T\mathcal{N}$ to rule out possible bottom resonance of ${\bf H}$, but \cite{LLOS1}
proved that the bottom resonance does not emerge no matter  what  gauge  one uses. Hence, (\ref{FxxgH})  also holds without taking  Coulomb gauge for $Q^*T\mathcal{N}$.

By [Theorem XIII.19, \cite{RS}] and the continuity of spectral projection operators, to prove the spectrum is absolutely continuous, it suffices to prove for any bounded interval $(a,b)$ and any $g\in C^{\infty}_c$
\begin{align}\label{ssss}
\mathop {\sup }\limits_{0 < \varepsilon  < 1}\int_a^b\left| \Im\left\langle {g,R(\tau + i\varepsilon,{\bf H} )g} \right\rangle \right|^2d\tau < \infty.
\end{align}
Using (\ref{FxxgH}) we have  for any $\varepsilon>0$
$$\left| {\left\langle {g ,R (\tau + i\varepsilon, {\bf H} )g} \right\rangle } \right| = \left| {\left\langle {g {\rho ^{ - \alpha' }},{\rho ^{\alpha'} }R (\tau + i\varepsilon ,{\bf H}){\rho ^{\alpha'} }{\rho ^{ - \alpha' }}g} \right\rangle } \right| \lesssim \left\| {f {\rho ^{ - \alpha' }}} \right\|_2^2 < \infty,
$$
which leads to (\ref{ssss}).
Meanwhile, by Weyl's criterion, $\sigma_{ess}(H)=[\frac{1}{4},\infty)$. Thus, we see $\sigma(H)=\sigma_{ac}(H)=[\frac{1}{4},\infty)$.

(\ref{kn16}) indeed follows by  \cite{Li3} and \cite{LLOS2}. In \cite{Li3}, we prove the same bounds as (\ref{kn16}) for $s\ge 1$. For $s\in[0,1]$, we prove
\begin{align*}
\| e^{s{\bf H}}f\|_{L^p_x}&\lesssim s^{-(\frac{1}{q}-\frac{1}{p})}\|f\|_{L^q_x}.
\end{align*}
And for $s\in[0,1]$, \cite{LLOS2} proved
\begin{align*}
\|(-\Delta)^{\gamma} e^{s{\bf H}}f\|_{L^p_x}&\lesssim s^{-\gamma}\|f\|_{L^p_x}.
\end{align*}
Then (\ref{kn16}) follows by the inequality
\begin{align*}
\|(-\Delta)^{\gamma}e^{s{\bf H}}f\|_{L^p_x}\lesssim \|(-\Delta)^{\gamma}e^{\frac{1}{2}s{\bf H}} \|_{L^p_x\to L^p_x}\|e^{\frac{1}{2}s{\bf H}}\|_{L^q_x\to L^p_x}\|f\|_{L^q_x}.
\end{align*}

Now let's prove (\ref{w4ib}), (\ref{w5ib}).
By the identity
\begin{align*}
T^{-\alpha}=\frac{\sin (\pi \alpha)}{\pi}\int^{\infty}_0\lambda^{-\alpha}(\lambda+T)^{-1}d\lambda,
\end{align*}
which holds provided that $\alpha\in (0,1)$, $T$ is a non-negative self-adjoint operator,  we find it suffices to show
\begin{align}
\int^{\infty}_0 \lambda^{-\alpha}\|(\lambda-{\bf H})^{-1}\|_{L^p\to L^p}d\lambda&\lesssim 1\label{1CvL}\\
\int^{\infty}_0 \lambda^{-\alpha}\|(\lambda-{\bf H})^{-1}\|_{\rho^{-\beta}L^p\to \rho^{-\beta}L^p}d\lambda&\lesssim 1\label{2CvL}
\end{align}
Applying the change of variables $\lambda=\sigma^2-\frac{1}{4}$,  (\ref{1CvL})-(\ref{2CvL}) follow directly from Lemma \ref{XCZ} below.

In our previous work \cite{Li3}, we proved  for $s\in(0,2)$, $p\in(1,\infty)$,
\begin{align}
\|(-\Delta)^{\frac{s}{2}}f\|_{L^p_x}&\lesssim \|(-{\bf H})^{\frac{s}{2}}f\|_{L^p_x}+\|f\|_{L^p_x}\label{p2ib}\\
\|(-{\bf H})^{\frac{s}{2}}f\|_{L^p_x}&\lesssim \|(-\Delta)^{\frac{s}{2}}f\|_{L^p_x}+\|f\|_{L^p_x}.\label{p1ib}
\end{align}
Then (\ref{w2ib})  follows by (\ref{w4ib}) and (\ref{p2ib}), and (\ref{w3ib}) follows by Sobolev embedding and (\ref{p1ib}).

(\ref{WW}) follows by (\ref{w2ib}) and the fact
\begin{align*}
\|f\|_{H^{\gamma}_x}\thicksim \|(-\Delta)^{\frac{\gamma}{2}}f\|_{L^2_x}.
\end{align*}

\end{proof}

\begin{Lemma}\label{XCZ}
For all $\sigma\ge \frac{1}{2}$, $p\in [2,\infty)$, we have
\begin{align}
\|(-{\Delta}+\sigma^2-\frac{1}{4})^{-1}\|_{L^p\to L^p}&\lesssim \min(1,\sigma^{-2}) \label{YZTR1}\\
\|\nabla (-\Delta+\sigma^2-\frac{1}{4})^{-1}\|_{L^p\to L^p}&\lesssim \min(1,\sigma^{-1}) \label{YZTR2} \\
\|  (-{\bf H}+\sigma^2-\frac{1}{4})^{-1}\|_{L^p\to L^p}&\lesssim \min(1,\sigma^{-2}). \label{YZTR3}\\
\| \nabla (-{\bf H}+\sigma^2-\frac{1}{4})^{-1}\|_{L^p\to L^p}&\lesssim \min(1,\sigma^{-1}). \label{YZTR4}
\end{align}
And for $0<\alpha'\ll1$, we also have
\begin{align}
\| e^{-\alpha' r} (-{\bf H}+\sigma^2-\frac{1}{4})^{-1}e^{-\alpha' r}\|_{L^2\to L^2}&\lesssim \min(1,\sigma^{-2}). \label{Gao1}\\
\| e^{-\alpha' r}\nabla (-{\bf H}+\sigma^2-\frac{1}{4})^{-1}e^{-\alpha' r}\|_{L^2\to L^2}&\lesssim \min(1,\sigma^{-1}). \label{Gao2}
\end{align}
\end{Lemma}
\begin{proof}
(\ref{YZTR1}), (\ref{YZTR2}) were  proved in [Lemma 3.7, \cite{Li3}]. (\ref{YZTR3}) was obtained in [(6.18), \cite{Li3}]. Now we prove (\ref{YZTR4}).

Denote ${\bf H}=\Delta+W$.
Before going on, we first point put that $W(-\Delta+\lambda)^{-1}$ is a compact operator in $L^p_x$ when $\lambda$ lies in the right half complex plane. In fact, for any $\Re \lambda\ge 0$, $f\in L^p_x$, by (\ref{YZTR2}) and (\ref{YZTR1}), we see $W(-\Delta+\lambda)^{-1}f\in L^p_x$ with the bound:
\begin{align}
\| W(-\Delta+\lambda)^{-1}f\|_{L^p_x}\lesssim  \|f\|_{L^p_x}.
\end{align}
And furthermore, by applying (\ref{YZTR1})-(\ref{YZTR2}) and the identity
$$-\Delta(-\Delta+\lambda)^{-1}=I-\lambda(-\Delta+\lambda)^{-1}.
$$
we obtain
\begin{align}
\| \nabla \left(W(-\Delta+\lambda)^{-1}f\right)\|_{L^p_x}\lesssim (1+|\lambda|)\|f\|_{L^p_x}.
\end{align}
Therefore, $W(-\Delta+\lambda)^{-1}\in \mathcal{L}(L^p,W^{1,p})$. Meanwhile, using the decay of $A^{\infty},\phi^{\infty}$ at infinity, we see for any $\epsilon>0$ there exists $R_{\epsilon}$ such that
\begin{align}
\| W(-\Delta+\lambda)^{-1}f\|_{L^p_x(r\ge R_{\epsilon})}\le \epsilon.
\end{align}
Then since Sobolev embedding $W^{1,p}\hookrightarrow L^p$ is compact in bounded domains, we conclude $W(-\Delta+\lambda)^{-1}$ is compact.

Second, we claim that\\
{\bf Claim 1.1}:
For any $\Re\lambda\ge 0$, the operator $I+W(-\Delta+\lambda)^{-1}$ is invertible in $\mathcal{L}(L^p_x,L^p_x)$ and is analytic w.r.t $\lambda\in\{z\in \Bbb C: \Re z\ge 0\}$.

We prove Claim 1.1 by contradiction. Assume that there exists some $\lambda_0\in\{z\in \Bbb C: \Re z\ge 0\}$ such that $I+W(-\Delta+\lambda)^{-1}$ is not invertible in $L^p_x$. Then by Fredholm's alternative, there exists $f\in L^p_x$ such that
\begin{align}\label{aZAQE}
f+W(-\Delta+\lambda_0)^{-1}f=0
\end{align}
Since $W(-\Delta+\lambda)^{-1}\in \mathcal{L}(L^p,W^{1,p})$, we see $g:=(-\Delta+\lambda_0)^{-1}f \in W^{1,p}$. Then by H\"older and $p\ge 2$, we infer that $Wg\in L^2_x$. Then by (\ref{aZAQE}), $f\in L^2_x$ and thus $g\in W^{2,2}$ because $-\lambda_0\in \rho(-\Delta)$.  Since (\ref{aZAQE}) shows $-{\bf H}g+\lambda_0 g=0$, and now $g\in W^{2,2}$, we see $g$ is an eigenfunction of $-\Delta$ with eigenvalue $-\lambda_0$. This implies $g=0$ since $-\lambda_0\in \rho(-\Delta)$. Hence Claim 1.1 follows.

Now, we are ready to prove (\ref{YZTR4}). By the formal identity
\begin{align}\label{ZAQE}
(-{\bf H}+\lambda)^{-1}=(-\Delta+\lambda)^{-1}(I+W(-\Delta+\lambda)^{-1})^{-1},
\end{align}
and (\ref{YZTR2}), it suffices to prove $(I+W(-\Delta+\lambda)^{-1})^{-1}$ is bounded in $L^p_x$ with a uniform bound:
\begin{align}\label{2ZAQE}
\sup_{ \lambda\ge 0}\|(I+W(-\Delta+\lambda)^{-1})^{-1}\|_{\mathcal{L}(L^p_x,L^P_x)}\lesssim 1.
\end{align}
Claim 1.1 shows for any bounded interval $[0,R_0]$ there uniformly holds
\begin{align}\label{ZAQE1}
\sup_{ \lambda\in [0,R_0]}\|(I+W(-\Delta+\lambda)^{-1})^{-1}\|_{\mathcal{L}(L^p_x,L^P_x)}\lesssim C(R_0).
\end{align}
By (\ref{YZTR1}), (\ref{YZTR2}) for $R_0\gg 1$ and $\lambda\ge R_0$,
\begin{align}
\| W(-\Delta+\lambda)^{-1}\|_{\mathcal{L}(L^p_x,L^P_x)}\le \frac{1}{2}.
\end{align}
Then by Neumann  series argument we see $\|(I+W(-\Delta+\lambda)^{-1})^{-1}\|_{\mathcal{L}(L^p_x,L^P_x)}\lesssim 2$ for $\lambda\ge R_0$, which combined with (\ref{ZAQE1}) yields (\ref{2ZAQE}). Therefore (\ref{YZTR4}) is obtained.

The rest (\ref{Gao1}), (\ref{Gao2}) follow by the same way.
\end{proof}

\section{Morawetz estimates}

\begin{Theorem}\label{kk}(Morawetz estimates for magnetic Schr\"odinger )
Let $\mathcal{M}$ be a $d$-dimensional Riemannian manifold without boundary. Let $A$ be a real valued time-dependent one form on $\mathcal{M}$, i.e., if $\{x^j\}^{d}_{j=1}$ are the local coordinates for $\mathcal{M}$, we can write $A=A_jdx^j$ where $A_j$ may dependent on $t\in \Bbb R$.
Assume that $x\mapsto a(x):\mathcal{M}\to \Bbb R$ is a $C^2$ real valued function.
Let $u$ solve
\begin{align}\label{mas2}
(\sqrt{-1}\partial_t-\Delta_{A})u& =F.
\end{align}
Let
\begin{align}\label{mas2}
M(t)=\sqrt{-1}\int_{\Bbb H^2}\left( u,(2\nabla a\cdot D_{A}+\Delta a) u\right)_{\Bbb C} {\rm{dvol_h}},
\end{align}
then the following identity holds:
\begin{align}
\frac{d}{dt}M(t)&=4\int \left(\nabla^2 a \cdot (D_{A}  u\otimes \overline{D_{ A} u})\right){\rm{dvol_h}}-\int (\Delta^2a)|u|^2{\rm{dvol_h}}\nonumber\\
&-4{\rm Im}\int u (\nabla^ja )(\nabla^k\bar{u})\nabla_k A_j{\rm{dvol_{h}}}-2i\int  |u|^2{Ric}(\nabla a, A) {\rm{dvol_{h}}}+2\int ( F,T u)_{\Bbb C} {\rm{dvol_{h}}}\label{KJY}
\end{align}
Here, $Ric(\nabla a, A):=h^{jl}h^{mk}R_{lk}A_m\nabla_j a$, and $R_{lk}dx^ldx^k$ denotes the Ricci curvature tensor, the operator $T$ is defined by  (\ref{pijuh}).
\end{Theorem}
\begin{proof}
For simplicity, we introduce some notations: For complex valued 1-forms on $\Bbb H^2$ we use $\alpha\cdot\beta$ to denote the ``real" inner product, i.e.
\begin{align}
\alpha\cdot\beta:= h^{kj}\alpha_k\beta_j,
\end{align}
where $\alpha=\alpha_kdx^k, \beta=\alpha_jdx^j.$ And let $(\alpha,\beta)_{\Bbb C}$ denote the complex inner product
\begin{align}
(\alpha,\beta)_{\Bbb C}:=h^{kj}\alpha_k\overline{\beta_j}.
\end{align}
In addition for 1-form $\alpha$ and k-form $\omega$ we define k-form valued product
\begin{align}\label{iuygh}
\alpha\cdot (\nabla \omega);=h^{lj}\alpha_j\nabla_{l}\omega.
\end{align}
It is easy to check (\ref{iuygh}) is free of coordinates.

And we use $\langle f,g\rangle $ to denote the inner product in $L^2$, i.e.
\begin{align*}
\langle f,g\rangle:=\int_{\mathcal{M}}f\bar{g}{\rm{dvol_{h}}}.
\end{align*}

It is easy to see for any smooth function $f$
\begin{align}\label{pijuh}
Tf:=[a,\Delta_{A}]f=2\nabla a\cdot D_{A} f+(\Delta a) f.
\end{align}
Then we have
\begin{align}
\frac{d}{dt}M(t)&=2\langle u,da\cdot (\partial_t A)u\rangle \langle \Delta u+F,T u\rangle -\langle u,T (\Delta_{A}u+F)\rangle\nonumber\\
&=\langle u,[\Delta_{A},T] u\rangle + \langle F,T u\rangle-\langle u,T F\rangle\label{vac}
\end{align}
Now we calculate $[\Delta_{A},T]$:
\begin{align}\label{mas3}
[\Delta_{A},T]f=[\Delta_{A},\Delta a]f+2[\Delta_{A},\nabla a\cdot D_{ A}]f,
\end{align}
By (\ref{pijuh}), the first right side term of (\ref{mas3}) is
\begin{align}\label{mas4}
[\Delta_{A},\Delta a]f=-2(\nabla \Delta a)\cdot D_{A} f-(\Delta^2 a)f.
\end{align}
For the second right side term of (\ref{mas3}), expanding $\Delta_{{A}}=\Delta+2i{A}\cdot \nabla+id^*{ A} -A\cdot A$ we obtain
\begin{align}
&2[\Delta_{A},\nabla a\cdot D_{ A}]f\nonumber\\
&=-2\Delta_{{A}}(\nabla a\cdot D_{ A}f)+2\nabla a\cdot D_{A}(\Delta_{{A}}f)\nonumber\\
&=-2(\Delta+2\mathbf{i}{A}\cdot \nabla )(\nabla a\cdot D_{A}f)+2\nabla a\cdot D_{A}(\Delta f+2iA\cdot \nabla f)\nonumber\\
&+2f\nabla a\cdot\nabla(id^*{ A}-{ A}\cdot { A})\nonumber\\
&:=-2[\Delta,\nabla a\cdot D_{ A}]f-4i[{A}\cdot \nabla,\nabla a\cdot D_{A}]f+I\label{ms7}
\end{align}
where in the last line we used $D_{A} (fg)=gD_{A} f+ f\nabla g$.

For the left terms in (\ref{ms7}), we calculate $[{A}\cdot \nabla,\nabla a\cdot D_{ A}]$ first.
We consider two 1-forms $\alpha=\nabla a$, $\beta=\nabla f$. Then by the comparability of covariant derivatives and Riemannian metric we have
\begin{align*}
&({A}\cdot \nabla)[\nabla a\cdot D_{A}f]=\sum^{2}_{i=1}h^{ij}A_j\nabla_{i}(\alpha\cdot \beta)+i({A}\cdot \nabla)\nabla a\cdot (f{A})\\
&=h^{kj}A_j(\nabla_{k}\alpha\cdot \beta)+h^{kj}A_j(\alpha\cdot \nabla_{k}\beta)+
i({A}\cdot \nabla)(f\nabla a\cdot A)\\
&=(A \cdot \nabla \alpha)\cdot \beta+\alpha\cdot ({ A}\cdot \nabla \beta)
+if({A}\cdot \nabla)(\nabla a\cdot { A})+i(\nabla a\cdot {  A})({A}\cdot \nabla f)
\end{align*}
and similarly
\begin{align*}
&\nabla a\cdot D_{A}({A}\cdot \nabla f)=\alpha\cdot(A\cdot \nabla \beta)+\alpha \cdot(\beta\cdot \nabla A)+i(\nabla a\cdot  A)({A}\cdot \nabla f).
\end{align*}
Recall $\beta=\nabla f$. Let $\{x_j\}^{d}_{j=1}$ be the normal coordinates, it is easy to see
\begin{align*}
\alpha\cdot(A\cdot \nabla \beta)=\sum^{d}_{l,j=1}\alpha_lA_jf_{lj}\\
(A \cdot \nabla \alpha)\cdot \beta=\sum^{d}_{l,j=1} A_l\alpha_lf_{jl}
\end{align*}
Thus we summarize that
\begin{align*}
[{A}\cdot \nabla,\nabla a\cdot D_{ A}]f=({{A}} \cdot \nabla \alpha)\cdot (\nabla f)
+if({A}\cdot \nabla)(\nabla a\cdot {A}).
\end{align*}
Now we calculate the first term in (\ref{ms7}):
\begin{align}\label{m8}
[\Delta,\nabla a\cdot D_{A}]f=[\Delta,\nabla a\cdot \nabla ]f+ [\Delta,\nabla a\cdot iA]f.
\end{align}
The second right hand side term in (\ref{m8}) is easy:
\begin{align}\label{nb}
 [\Delta,\nabla a\cdot iA]f= i\Delta (\nabla a\cdot { {A}})f+2\left(\nabla(\nabla a\cdot {{A}})\right) \cdot \nabla f.
\end{align}
For the first RHS term in (\ref{m8}), letting $\{x_j\}^{d}_{j=1}$ be the normal coordinates we obtain by the comparability of covariant derivatives and Riemannian metric that
\begin{align*}
\Delta(\nabla a\cdot \nabla f)=\sum^{j}_{k=1}(\nabla_k\nabla_k \nabla a)\cdot  (\nabla f)+(\nabla a)\cdot \nabla_k\nabla_k (\nabla f)+2
\sum^{d}_{k=1}(\nabla_k \nabla a)\cdot  (\nabla_k \nabla f),
\end{align*}
and
\begin{align*}
&(\nabla a)\cdot \nabla_l\nabla_l (\nabla f)=(\nabla a)\cdot \nabla_l \left((\partial_{lj}f-\Gamma^{k}_{lj}\partial_k f)dx^j\right)\\
&=(\nabla a)\cdot \left(\partial_{ljl}f-\Gamma^{k}_{lj}\partial_{lk}f-\partial_i\Gamma^{k}_{lj}\partial_k f)dx^j\right)-(\nabla a)\cdot \Gamma^{j}_{lp} \left((\partial_{lj}f-\Gamma^{k}_{lj}\partial_k f)dx^p\right)\\
&=\sum^{d}_{j=1}a_j\partial_{ljl}f-\partial_{j}a\partial_l\Gamma^{k}_{lj}\partial_k f.
\end{align*}
Meanwhile,in the normal coordinates it similarly   holds
\begin{align*}
\nabla a\cdot \nabla (\Delta f)=\sum^{d}_{j=1}a_j\partial_{j}(\Delta f)=\sum^{d}_{l,j=1}a_j\partial_{jll}f-a_j(\partial_{j}\Gamma^k_{l,l})\partial_kf.
\end{align*}
Since $R^{\nu}_{\mbox{ }\mu,\kappa,\lambda}=\partial_{\kappa}\Gamma^{\nu}_{\mu,\lambda}-\partial_{\lambda}\Gamma^{\nu}_{\mu,\kappa}$ in normal coordinates, we conclude the first RHS term in (\ref{m8}) is
\begin{align*}
&[\Delta,\nabla a\cdot \nabla ]f\\
&=Ric(\nabla a,\nabla f)+\sum^{d}_{k=1}(\nabla_k\nabla_k\nabla a)\cdot  (\nabla f)+2
\sum^{d}_{l=1}(\nabla_l \nabla a)\cdot  (\nabla_i \nabla f)\\
&
=Ric(\nabla a,\nabla f)+(\nabla f)\cdot {\rm trace} \nabla^2(\nabla a) +2(\nabla  (\nabla a)\cdot  \nabla (\nabla f))
\end{align*}
where we adopt the notations
\begin{align}
Ric(\alpha,\beta)&=h^{ml}\alpha^jR_{\mbox{ }m,k,j,l}\beta^k\label{ijhr54}\\
{\rm trace} \nabla^2(\alpha)&=h^{lj}\nabla_j\nabla_l(\alpha)-h^{lj}\Gamma^{k}_{lj}\nabla_k(\alpha)\label{ijhr55}\\
(\nabla  (\nabla a)\cdot  \nabla (\nabla f))&=\sum^{2}_{i,j=1}h^{ij}(\nabla_i \alpha)\cdot  (\nabla_j \beta)\label{ijhr56}
\end{align}
for any 1-form $\alpha,\beta$. It is direct to check (\ref{ijhr54})-(\ref{ijhr56}) are independent of coordinates thus defining an intrinsic quantity.
Therefore, by (\ref{nb}), (\ref{ms7}), (\ref{m8}),
\begin{align}
& [\Delta_{A},T]f\nonumber\\
&=-2(\nabla \Delta a)\cdot D_{A} f-(\Delta^2a)f-2Ric(\nabla a,\nabla f)-2(\nabla f)\cdot {\rm trace} \nabla^2(\nabla a)\nonumber \\
&-4(\nabla  (\nabla a)\cdot  \nabla (\nabla f))-2if\Delta (\nabla a\cdot {{A}})
-4i\left(\nabla(\nabla a\cdot {{A}})\right) \cdot \nabla f-4i(A\cdot \nabla \alpha)\cdot\nabla f\nonumber\\
&+4f({A}\cdot \nabla)(\nabla a\cdot {A})+2f\nabla a\cdot\nabla(id^*{ A}-{A}\cdot {A})\label{xcds}
\end{align}

Hence, we deduce from (\ref{mas3}), (\ref{mas4}), (\ref{xcds}) that the term $ \langle u,[\Delta_{A},T] u\rangle$
in (\ref{vac}) now expands as
\begin{align}
&\langle u,[\Delta_{A},T] u\rangle=II\nonumber\\
&-2\langle u, (\nabla \Delta a)\cdot \nabla u\rangle-2\langle u,(\nabla u)\cdot {\rm trace} \nabla^2(\nabla a)\rangle
-4\langle u,(\nabla  (\nabla a)\cdot  \nabla (\nabla u))\rangle\label{oijhu}
\end{align}
where (\ref{oijhu}) are the three leading terms, and $II$ denotes the corresponding parts to other terms in (\ref{xcds}).

We now aim to use integration by parts to obtain simpler formula for $ \langle u,[H_0,T] u\rangle$.
Under the normal coordinates one can check
\begin{align*}
u(\nabla \Delta a)\cdot \nabla \bar{u}&=- u Ric (\nabla \bar{u},\nabla a)-d^{*} (u(\nabla {\bar u}\cdot\nabla^2 a)) -
\left(\nabla^2 a \cdot (\nabla u\otimes\nabla \bar{u})\right)\\
&-u\left(\nabla (\nabla a )\cdot (\nabla (\nabla {\bar u}))\right).
\end{align*}
And it is easy to see $\widetilde{\omega}:=u(\nabla \bar{u})\cdot \nabla_i(\nabla a)dx^i$ is free of coordinates and thus defines a global 1-form, then
\begin{align*}
u(\nabla\bar{u})\cdot {\rm trace} \nabla^2(\nabla a)=d^{*}\widetilde{\omega}-
\nabla {\bar u}\cdot \left(\nabla u\cdot \nabla (\nabla a)\right)-u\left(\nabla (\nabla a )\cdot (\nabla (\nabla {\bar u}))\right).
\end{align*}
Thus by integration by parts (\ref{oijhu}) equals
\begin{align}\label{1oijhu}
4\int \left(\nabla^2 a \cdot (\nabla u\otimes\nabla \bar{u})\right){\rm{dvol_h}}+2\int u Ric (\nabla \bar{u},\nabla a){\rm{dvol_h}}.
\end{align}
In order to obtain a neat formula we change $\nabla u$ in (\ref{1oijhu}) to be $D_{A} u$ and calculate the error terms involved with ${A}$. Thus by (\ref{xcds}),
 $\langle u,[\Delta_{A},T] u\rangle$ can be written as
\begin{align}
&\langle u,[\Delta_{A},T] u\rangle\nonumber\\
&=4\int \left(\nabla^2 a \cdot (\nabla u\otimes \overline{\nabla u})\right){\rm{dvol_h}}-\int (\Delta^2a)|u|^2{\rm{dvol_h}}\nonumber\\
&+2i\int u(\nabla \Delta a)\cdot ({\bar{ u}} {A}){\rm{dvol_h}}
-2\int u\overline{iu}\Delta (\nabla a\cdot {  {A}}){\rm{dvol_h}}\label{pmmno2}\\
&-2\int  Ric(\nabla a,u\nabla \bar{u}){\rm{dvol_h}}+2\int   Ric (u\nabla \bar{u},\nabla a) {\rm{dvol_h}}\nonumber\\
&+4i\int (u  A\cdot \nabla \alpha)\cdot\nabla \bar{u}{\rm{dvol_h}}+4i\int \left(\nabla(\nabla a\cdot { {A}})\right) \cdot \nabla \bar{u}{\rm{dvol_h}}\label{mmno2} \\
&+4\int |u|^2({ A}\cdot \nabla)(\nabla a\cdot  {A}){\rm{dvol_h}}+2\int |u|^2\nabla a\cdot\nabla(id^*{  A}-{ A}\cdot { A}){\rm{dvol_h}}\label{xccds}
\end{align}
Using the normal coordinates, it is easy to check the point-wise identities for following two terms in (\ref{mmno2}), (\ref{xccds})
\begin{align*}
I_1:&=4 iu\left(\nabla(\nabla a\cdot {  {A}})\right) \cdot \nabla \bar{u}=4 iu\nabla^2 a\cdot (A\otimes\nabla \bar{u})+4 i u\nabla^{j}\bar{u} (\nabla^ka\nabla_j A_k) \\
I_2:&=4 |u|^2({ A}\cdot \nabla)(\nabla a\cdot {A})=4 |u|^2\nabla^2a\cdot ( A\otimes  A)+4 |u|^2(\nabla^j a) A_{k}\nabla^k{A}_j ,
\end{align*}
and the two terms in (\ref{pmmno2})
\begin{align*}
I_3:&=2iu(\nabla \Delta a)\cdot ({\bar{ u}}  {A})=-2iu Ric (\bar{u} {A},da)-2id^{*} (u( {\bar u  {A}}\cdot\nabla^2 a)) \\
&-2i\left(\nabla^2 a \cdot (\nabla u\otimes  \bar{u} {A})\right)\\
I_4:&=-2 u\overline{iu}\Delta (\nabla a\cdot {  {A}})=2id^{*}\left(|u|^2 d(\nabla a\cdot {A})\right)-2i(\nabla^2 a)\cdot (\nabla |u|^2\otimes {A})\\
&-2i \nabla^j a ( \nabla^k |u|^2)\nabla_k {A}_j .
\end{align*}
Thus integration by parts shows
\begin{align*}
&4\int \left(\nabla^2 a \cdot (\nabla u\otimes \overline{\nabla u})\right){\rm{dvol_h}}+\int I_1+I_2+I_3+I_4 {\rm{dvol_h}}\\
&=4\int \left(\nabla^2 a \cdot (D_{A}   u\otimes \overline{D_{A} u})\right){\rm{dvol_h}}+
4i\int u(\nabla^k a ,\nabla_j A_k)\nabla^j\bar{u} ){\rm{dvol_h}}-2i\int \nabla^ja\cdot (\nabla_k |u|^2\nabla^k{A}_j){\rm{dvol_h}}\\
&-2i\int |u|^2 Ric ({A},da){\rm{dvol_h}}.
\end{align*}
Therefore, we conclude that
\begin{align*}
&\langle u,  [\Delta_{A},T]u\rangle \nonumber\\
&=4\int \left(\nabla^2 a \cdot (D_{A}  u\otimes \overline{D_{A} u})\right){\rm{dvol_h}}-\int (\Delta^2a)|u|^2{\rm{dvol_h}}\\
&+4i\int \nabla^k a (\nabla_j A_k) (u\nabla^j\bar{u}){\rm{dvol_h}}-2i\int \nabla^ja  (\nabla^k |u|^2)\nabla_k {A}_j {\rm{dvol_h}}\\
&-2 \int u{Ric} ( \nabla\bar{u},\nabla a){\rm{dvol_h}}+2\int u{Ric} (\nabla a,\nabla{\bar{u}} ){\rm{dvol_h}}-2i\int  |u|^2 Ric(\nabla a,A){\rm{dvol_h}}.
\end{align*}
Since the Ricci tensor is symmetric, we further get (\ref{KJY}) by (\ref{vac}).

\end{proof}

\begin{Corollary}\label{smo}
Assume that $\|A \|_{L^{\infty}_{t,x}}\lesssim 1$. Let $\mathcal{M} =\Bbb H^2$ in Theorem \ref{kk},
then the Morawetz estimates hold:
\begin{align}
\|e^{-\frac{r}{2}}\nabla u\|^2_{L^2_{t.x}}&\lesssim \|(-\Delta)^{\frac{1}{4}}u \|^2_{L^{\infty}_tL^2_x}+\||\partial_t A||u|^2\|_{L^{1}_{t,x}}+\|e^{-r}|A|^2|u|^2\|_{L^1_{t,x}}+\|\nabla A||u||\nabla u|\|_{L^{1}_{t,x}}
\nonumber\\
&+\||u|^2|A|\|_{L^1_tL^1_x}+\||D_{A}u ||F|\|_{L^1_{t.x}}+\||u||F|\|_{L^1_{t.x}}\label{1kz}
\end{align}
\end{Corollary}
\begin{proof}
As \cite{IS}, we take $a$ to be a radial function such that $\Delta a=1$, then $a$ is given by
\begin{align}
a(r)=\int^r_0\left(\frac{1}{\sinh \tau}\int^{\tau}_0{\sinh} sds\right)d\tau,
\end{align}
and
\begin{align}
\partial_ra(r)&=\frac{1}{\sinh r}\int^r_0{\sinh} sds=\tanh(\frac{r}{2})\\
\partial^2_ra(r)&=\frac{1}{2\cosh^2\frac{r}{2}}.
\end{align}
By simple bilinear argument (see Lemma \ref{complex}) one obtains $M(t)$ defined in
(\ref{mas2}) satisfies
\begin{align}
|M(t)|\lesssim \|(-\Delta)^{\frac{1}{4}}u\|^2_{L^2_x}.
\end{align}
Then integrating (\ref{KJY}) in $t\in[t_1,t_2]$ gives
\begin{align}
\sup_{t\in[t_1,t_2]}\|(-\Delta)^{\frac{1}{4}}u(t)\|^2_{L^2_x}&\ge 4\int^{t_2}_{t_1}\int_{\Bbb H^2} \nabla^2 a \cdot (D_{A}  u\otimes \overline{D_{A} u}) {\rm{dvol_h}}dt\label{YTFR}\\
&-4\|u|\nabla u||\nabla A|\|_{L^1_tL^1_x([t_1,t_2]\times\Bbb H^2)}
-2\||u|^2|A|\|_{L^1_tL^1_x([t_1,t_2]\times\Bbb H^2)}\nonumber\\
&-2\|(F,T u)\|_{L^1_tL^1_x([t_1,t_2]\times\Bbb H^2)}.\nonumber
\end{align}
The first RHS term of (\ref{YTFR}) contains the leading term and it can be expanded as
\begin{align*}
&\int^{t_2}_{t_1}\int_{\Bbb H^2} \nabla^2 a \cdot (\nabla   u\otimes  \nabla \bar{u}){\rm{dvol_h}}dt\\
&=\int^{t_2}_{t_1}\int_{\Bbb H^2} \frac{1}{2\cosh^2\frac{r}{2}} |\partial_ru|^2+\frac{\cosh r}{\sinh^3 r}\tanh \frac{r}{2}|\partial_{\theta}u|^2{\rm{dvol_h}}dt\\
&\gtrsim\int^{t_2}_{t_1}\int_{\Bbb H^2} e^{-r} |\nabla u|^2{\rm{dvol_h}}dt.
\end{align*}
Then the desired result follows by combing the above together.
\end{proof}

\subsection{Energy estimates}

\begin{Proposition}\label{6.1}
Denote $B=B_j(s,t,x)dx^j$ be a 1-form on  $\Bbb H^2\times{\Bbb I}$.
Assume that $u$ solves linear Schr\"odinger equation corresponding to $B$:
\begin{align}
\left\{
  \begin{array}{ll}
&\sqrt{-1}\partial_tu+\Delta_{B}u=F \hbox{ } \\
    &u(0,x)=u_0(x), \hbox{ }
  \end{array}
\right.
\end{align}
Then there holds
\begin{align}
&\|\nabla u\|^2_{L^{\infty}_tL^{2}_x([0,T]\times \Bbb H^2)}\label{huwahua}\\
&\lesssim \|\nabla  u_0\|^2_{L^2_x}+\|D_{B} FD_{B}u\|_{L^{1}_tL^{1}_x([0,T]\times \Bbb H^2)}
+\||\partial_t{B}| |u||D_{B}u|\|_{L^1_tL^1_x([0,T]\times \Bbb H^2)}.
\end{align}
with implicit constant independent of $T>0$.
\end{Proposition}
\begin{proof}
This follows directly by calculating
\begin{align*}
 \frac{d}{dt}\frac{1}{2}\int_{\Bbb H^2} D_{B}u \overline{D_{B}u}{{\rm dvol_h}}
\end{align*}
and integration by parts.
\end{proof}

\section{Strichartz Estimates}

Let us recall the Strichartz Estimates for linear free Schr\"odinger equations on $\Bbb H^2$.

For free Schr\"odigner equations on $\Bbb H^n$,  Anker-Pierfelice \cite{AP} proved  the following  dispersive estimates:
\begin{align*}
\|e^{it\Delta }f\|_{L^q(\Bbb H^n)}&\lesssim t^{-\max(\frac{1}{2}-\frac{1}{q},\frac{1}{2}-\frac{1}{\tilde{q}})n}\|f\|_{L^{\tilde{q}'}}, \mbox{ }0<|t|<1\\
\|e^{it\Delta }f\|_{L^q(\Bbb H^n)}&\lesssim t^{-\frac{3}{2}}\|f\|_{L^{\tilde{q}'}}, \mbox{ }|t|\ge 1,
\end{align*}
provided that $q,\tilde{q}\in (2,\infty]$, and $\tilde{q}'$ denotes  the conjugate of $\tilde{q}$.

Define the admissible pair $(p,q)$ to be
\begin{align}\label{end}
\left\{  (\frac{1}{p},\frac{1}{q})\in (0,\frac{1}{2}]\times (0,\frac{1}{2}):\frac{2}{p}+\frac{2}{q}\ge 1\right\}\cup \left\{(\frac{1}{p},\frac{1}{q})=(0,\frac{1}{2})\right\},
\end{align}
Anker-Pierfelice \cite{AP} proved  the following Strichartz estimates on $\Bbb H^2$:
Let $u$ solve the equation
\begin{align*}
\left\{
  \begin{array}{ll}
    i\partial_t u+\Delta u=F, & \hbox{ } \\
    u(0,x)=u_0\in L^2. & \hbox{ }
  \end{array}
\right.
\end{align*}
Then
for all admissible pairs $(p,q)$, $(\tilde{p},\tilde{q})$  in (\ref{end}), one has
\begin{align*}
\|u\|_{L^p_tL^q_x}&\lesssim \|u_0\|_{L^2_x} +\|F\|_{L^{\tilde{p}'}_tL^{\tilde{q}'}_x},
\end{align*}

The start point to endpoint Strichartz estimates for magnetic Schr\"odinger operators is the following estimates.  The $\Bbb H^2$ version of  (\ref{CDF}) was obtained by  our previous work \cite{Li2} for magnetic  wave equations.

\begin{Lemma}\label{NZQ1}
Let $u$ solve the equation
\begin{align}
\left\{
  \begin{array}{ll}
    i\partial_tu+\Delta u=F, & \hbox{ } \\
    u(0,x)=0, & \hbox{ }
  \end{array}
\right.
\end{align}
Let $\alpha'>0$, then for any admissible pair $(p,q)$ one has
\begin{align}\label{CDF}
\|(-\Delta)^{\frac{1}{4}}u\|_{L^p_tL^q_x}&\lesssim \|e^{\alpha'r}F\|_{L^2_{t,x}}.
\end{align}
\end{Lemma}

\begin{Remark}
 The $\Bbb R^n$ version of (\ref{CDF}) was obtained by  Ionescu-Kenig \cite{IK2}. The proof of  Lemma \ref{NZQ1} relies on Keel-Tao's bilinear arguments and kernel estimates for frequency localized  Schr\"odinger propagators analogous to  \cite{APV}. The proof is presented in Appendix C.  One may see our previous work \cite{Li2} for the wave analogy.
\end{Remark}

We also recall the following Kato smoothing estimates of   ${\bf H}$  as a corollary of \cite{LLOS1}.

\begin{Lemma}\label{Q1}
Let $u$ solve the equation
\begin{align}\label{Uhy78j}
\left\{
  \begin{array}{ll}
    i\partial_tu+{\bf H} u=0, & \hbox{ } \\
    u(0,x)=u_0, & \hbox{ }
  \end{array}
\right.
\end{align}
Then for admissible pair $(p,q)$ one has
\begin{align}\label{4F}
\|e^{-\alpha' r} \nabla u\|_{L^2_tL^2_x}+\|e^{-\alpha' r} u\|_{L^2_tL^2_x}&\lesssim \|(-\Delta)^{\frac{1}{4}}u_0\|_{L^2_{x}}.
\end{align}
provided that $\alpha'>0$.
\end{Lemma}
\begin{proof}
As a corollary of the local smoothing estimates by  \cite{LLOS1},  one has for $\alpha'>0$
\begin{align*}
\| e^{-\alpha'r} (-\Delta)^{-\frac{1}{4}} u\|_{L^2_tL^2_x}&\lesssim \|u_0\|_{L^2_{x}}.
\end{align*}
Then using the equivalence of $(-\Delta)^{\frac{1}{4}}$ and $(-{\bf H})^{\frac{1}{4}}$ in weighted $L^2$ space (see Proposition 5.1)  yields
\begin{align*}
\| e^{-\alpha'r}(-{\bf H})^{\frac{1}{4}} u\|_{L^2_tL^2_x}&\lesssim \|u_0\|_{L^2_{x}}.
\end{align*}
Therefore,
 applying $(-{\bf H})^{\frac{1}{2}}$ to (\ref{Uhy78j}) shows
\begin{align*}
\|e^{-\alpha'r}(-{\bf H})^{\frac{1}{2}} u\|_{L^2_tL^2_x}&\lesssim \|(-{\bf H})^{\frac{1}{4}}u_0\|_{L^2_{x}}.
\end{align*}
Then the first RHS term of (\ref{4F}) follows by applying the equivalence of $(-\Delta)^{\frac{1}{4}}$ and $(-{\bf H})^{\frac{1}{4}}$ in weighted space and $L^2$.
The   second  RHS term of (\ref{4F}) follows by applying  Poincare inequality to $e^{-\alpha' r}u$ with $0<\alpha'\ll 1$.
\end{proof}

\begin{Lemma}\label{MQ}
Let $u$ solve the equation
\begin{align}\label{GCF}
\left\{
  \begin{array}{ll}
    i\partial_t u+{\bf H}u=0, & \hbox{ } \\
    u(0,x)=u_0, & \hbox{ }
  \end{array}
\right.
\end{align}
Then for admissible pairs $(p,q)$ we have
\begin{align}\label{sxxf}
\|u\|_{L^p_tL^q_x(\Bbb R\times\Bbb H^2)}&\lesssim \| u_0 \|_{L^2_{ x}}.
\end{align}
\end{Lemma}
\begin{proof}

{\bf Step 1.} We first prove the following claim:
{{\it Claim A}.} Let $u$ solve the equation
\begin{align}
\left\{
  \begin{array}{ll}
    i\partial_tu+{\bf H}u=0, & \hbox{ } \\
    u(0,x)=u_0, & \hbox{ }
  \end{array}
\right.
\end{align}
Then for admissible pairs $(p,q)$ there holds that
\begin{align}
\|(-\Delta)^{\frac{1}{4}} u\|_{L^p_tL^q_x}&\lesssim \|(-\Delta)^{\frac{1}{4}}u_0 \|_{L^2_{ x}}.
\end{align}
Applying Lemma \ref{NZQ1} and Duhamel principle, we obtain
\begin{align*}
\|(-Delta)^{\frac{1}{4}} u\|_{L^p_tL^q_x([0,T]\times\Bbb H^2)}&\lesssim \|D^{\frac{1}{2}}u_0\|_{L^2_x}+\left\| e^{\alpha'r} |A^{\infty}|\right\|_{L^{\infty}_{ x}}\| e^{-\alpha'r}\nabla u\|_{L^2_{t,x}}\\
&+\left\|(|\phi^{\infty}|^2+|{A}^{\infty}|^2+|\nabla {A}^{\infty}|^2) e^{\alpha'r}\right\|_{L^{\infty}_x}\|e^{-\alpha'r} u\|_{L^2_{t,x}},
\end{align*}
where we take  $0<\alpha'\ll 1$.  Then by  Lemma \ref{Q1},
\begin{align*}
\|(-\Delta)^{\frac{1}{4}} u\|_{L^p_tL^q_x([0,T]\times\Bbb H^2)}\lesssim \|D^{\frac{1}{2}}u_0\|_{L^2_x}.
\end{align*}

{\bf Step 2.}
If $u$ solves (\ref{GCF}), then $v:=(-{\bf H})^{-\frac{1}{4}} u(t)$ solves
\begin{align}
i\partial_tv+{\bf H}v=0.
\end{align}
Thus Claim A of Step 1 gives
\begin{align*}
\|(-\Delta)^{\frac{1}{4}} v\|_{L^p_tL^q_x}\lesssim \|(-\Delta)^{\frac{1}{4}}v_0\|_{L^2_x}.
\end{align*}
 Applying the equivalence of $(-\Delta)^{\frac{1}{4}}$ and $(-{\bf H})^{\frac{1}{4}}$ in $L^q_x$ (see Prop. 5.1) gives  (\ref{sxxf}).
\end{proof}

By Christ-Kiselev Lemma, Lemma \ref{MQ} immediately yields
\begin{Lemma}\label{Strichartz}
Let $u$ solve the equation
\begin{align}
\left\{
  \begin{array}{ll}
    i\partial_tu+{\bf H} u=F, & \hbox{ } \\
    u(0,x)=u_0(x), & \hbox{ }
  \end{array}
\right.
\end{align}
Then for any admissible pair $(p,q)$ one has
\begin{align}\label{XZWE}
\| u\|_{L^p_tL^q_x([0,T]\times\Bbb H^2)}&\lesssim \|u_0\|_{L^2_x}+\|F\|_{L^1_{t}L^2_x([0,T]\times\Bbb H^2)}.
\end{align}
where the implicit constant is independent of $T$.
\end{Lemma}

\section{Proof  of  K\"ahler targets}

The proof for K\"ahler targets is the almost same as Riemannian surface targets. The main estimates we rely on are energy estimates in  Proposition \ref{6.1},
Morawetz estimates in  Corollary \ref{smo} (where $A_j$ now shall be matrix-valued functions rather than scalar-valued functions), and endpoint Strichartz estimates for $e^{it\Delta}$ of \cite{AP}.

We point out the differences: (i) Since we treat small harmonic maps in K\"ahler targets case, Strichartz estimates for $e^{it\Delta}$ is enough (Particularly, the $\phi^j\phi_j\phi_s$ term in the RHS of (\ref{Nsd2}) causes no trouble); (ii)  The curvature part in K\"ahler targets case shall be treated slightly different from Riemannian surface.

The curvature term  in (\ref{heating1})  can be schematically written as
\begin{align*}
{\rm Re}[\mathcal{R}(v)(\phi_i, {\phi})\phi_j]^{\alpha}=\psi^{i_1}_i\psi^{i_2}\psi^{i_3}_{j}{\bf R}(e_{i_1},e_{i_2},e_{i_3},e_{\alpha})\\
{\rm Im}[\mathcal{R}(v)(\phi_i, {\phi})\phi_j]^{\alpha}=\psi^{i_1}_i\psi^{i_2}\psi^{i_3}_{j}{\bf R}(e_{i_1},e_{i_2},e_{i_3},e_{\alpha+n})
\end{align*}
By caloric condition $\widetilde{\nabla}_se_{i}=0$, we further have
\begin{align*}
{\bf R}(e_{i_1},e_{i_2},e_{i_3},e_{i_4})=limit-\int^{\infty}_s\phi^{k}_s{D {\bf R}}(e_{k};e_{i_1},...,e_{i_4})ds'
\end{align*}
where ``limit" refers to the limit of the RHS as $s\to\infty$.
For simplicity we denote all terms like
$$\int^{\infty}_s\phi^{k}_s{D {\bf R}}(e_{k};e_{i_1},...,e_{i_4})ds'
$$
for  various indexes $i_1,...,i_4,k$ by the same notation $\widetilde{\mathcal{G}}$. The ``limit" for different $e_{i_1},e_{i_2},e_{i_3},e_{i_4}$ will denoted by $\Gamma$ for simplicity.
Moreover, the cubic product $\psi^{i_1}_i\psi^{i_2}\psi^{i_3}_{j}$ for  various indexes $i_1,...,i_3$ will be denoted as the same notation
$$\phi_i\diamond\phi\diamond \phi_j.
$$
Using these notations,  (\ref{heating1}) now reads as
\begin{align}\label{Eff}
{i}\mathbf{D}_t{\phi}_s+\Delta_{A}{\phi}_s=i\partial_s Z+ \Gamma  \phi^j \diamond {\phi}_s\diamond  \phi_j  +  \widetilde{\mathcal{G}}\phi^j \diamond {\phi}_s\diamond \phi_j .
\end{align}

In the following lemma, we collect the point-wise estimates related with  curvatures.
\begin{Lemma}
For compact K\"ahler targets, we have under the caloric gauge that
\begin{align*}
 |\widetilde{\mathcal{G}}| &\lesssim  \int^{\infty}_s |\phi_s|ds'\\
 |\nabla  \widetilde{\mathcal{G}}| &\lesssim   \int^{\infty}_s (|\nabla  \phi_s|+|A||\phi_s|)ds'\\
 |\partial_t \widetilde{ \mathcal{G}} |& \lesssim   \int^{\infty}_s (|\partial_t  \phi_s|+|A_t||\phi_s|)ds'.
\end{align*}
\end{Lemma}

Using these estimates and repeating the proof of Theorem 1.1 yields Theorem  1.2.

\section{Appendix A}

The estimate of the heat semigroup in $\mathbb{H}^2$ is as follows.
\begin{Lemma}[\cite{DM,LLOS2,CGS,LZ}]
For $1\le r\le p\le\infty$, $\alpha\in[0,1]$, $1<q<\infty$, the heat semigroup on $\Bbb H^2$ denoted by $e^{s\Delta}$ satisfies
\begin{align}
\|e^{s\Delta}(-\Delta)^{\alpha} f\|_{L^{q}_x}&\lesssim s^{-\alpha}e^{-\rho_{q} s}\|f\|_{L^{q}_x},
\end{align}
for any $0<\rho_{q}\le \frac{1}{2}\min(\frac{1}{q},1-\frac{1}{q})$.
And for $f\in L^2$ it holds that
\begin{align*}
\int^{\infty}_0\|e^{s\Delta}f\|^2_{L^{\infty}_x}ds\lesssim \|f\|^2_{L^2}.
\end{align*}
\end{Lemma}

\begin{Lemma}\label{complex}
Assume that $\|A \|_{L^{\infty}_{t,x}}\lesssim 1$, then the bilinear form of $M$ defined by (\ref{mas2}) satisfies
\begin{align*}
|M(f,g)|&\lesssim \|f\|_{H^{\frac{1}{2}}}\|g\|_{H^{\frac{1}{2}}}.
\end{align*}
\end{Lemma}
\begin{proof}
Recall that
\begin{align*}
M(f,g)=\sqrt{-1}\int_{\Bbb H^2}(f,(2\nabla a\cdot D^{\tilde A}+\Delta a) g)_{\Bbb C}{\rm{dvol_h}}.
\end{align*}
Since $T:= 2\nabla a\cdot D^{\tilde A}+\Delta a$ is skew-symmetric in $L^2$ we find
\begin{align*}
M(f,g)=M(g,f).
\end{align*}
Then by Poincare's inequality  and $\Delta a=1$, $|\nabla a|\le 1$ we have
\begin{align*}
|M(f,g)|&\le \|f\|_{L^2}\|g\|_{L^2}+2\|f\|_{L^2}\|\nabla g\|_{L^2}+2\|A\|_{L^{\infty}}\|f\|_{L^2}\|g\|_{L^2}\\
&\le \left(\frac{5}{2}+\|A\|_{L^{\infty}}\right)\|f\|_{L^2}\|\nabla g\|_{L^2}.
\end{align*}
Thus
\begin{align*}
|M(f,g)|&\le\left(\frac{5}{2}+\|A\|_{L^{\infty}}\right)\|f\|_{L^2}\|\nabla g\|_{L^2}\\
|M(g,f)|&\le\left(\frac{5}{2}+\|A\|_{L^{\infty}}\right)\|g\|_{L^2}\|\nabla f\|_{L^2},
\end{align*}
which gives the desired result by bilinear interpolation.
\end{proof}

\section{Appendix B}

We say a few words on holomorphic and anti-holomorphic maps. A map $f$ from K\"ahler manifold   $(\mathcal{M}_1,J_1, g_1)$ to  $(\mathcal{M}_2,J_2,g_2)$ is said to be
 holomorphic if
\begin{align}\label{1Dfg}
f_*(J_1X)=J_2f_*(X), \mbox{ }\forall X\in T\mathcal{M}_1,
\end{align}
or anti-holomorphic if
\begin{align}\label{2Dfg}
f_*(J_1X)=-J_2f_*(X), \mbox{ }\forall X\in T\mathcal{M}_1.
\end{align}
Both  holomorphic and anti-holomorphic maps are harmonic maps.
Assume that  $\mathcal{M}_1,\mathcal{M}_2$ are Riemannian surfaces. Fixing  local complex  coordinates for $\mathcal{M}_1,\mathcal{M}_2$,  one may treat $f$ as a complex valued function  $f(z,\bar{z})$, then
holomorphic maps refer to
\begin{align*}
\frac{\partial}{\partial {z}}f=0,
\end{align*}
and anti-holomorphic maps refer to
\begin{align*}
\frac{\partial}{\partial {\bar{z}}}f=0.
\end{align*}

From the above discussions, when $\mathcal{M}=\mathcal{N}=\Bbb H^2$, it is easy to see any analytic function from Poincare disk into  compact subset of Poincare disk induces a  holomorphic map satisfying assumptions of Theorem 1.1. Moreover,the conjugate of any  analytic function
from Poincare disk into  compact subset of Poincare disk also gives an anti-holomorphic map satisfying assumptions of Theorem 1.1.

For orthonormal  frames $(e'_1,e'_2)$ for $\mathcal{M}_1$, and orthonormal frames $(e_1,e_2)$ for $\mathcal{M}_2$, satisfying $e'_2=J_1e'_1$,$e_2=J_2e_1$, letting
$\phi_j=\langle \nabla_{e'_j}f,$ $ e_1\rangle +i\langle \nabla_{e'_j}f, e_2\rangle$, (\ref{1Dfg}) then  reads as
\begin{align*}
\phi_1=i\phi_2,
\end{align*}
while (\ref{2Dfg}) reads as
\begin{align*}
\phi_1=-i\phi_2.
\end{align*}

\begin{Remark}
Using the above two identities, it is easy to check the $i\kappa{\rm Im}(\phi^j\overline{\phi_s})\phi_j$ part of  linearized operator ${\bf H}$ in Lemma \ref{linear} is self-adjoint. Moreover, when the sectional curvature $\kappa$ is non-positive, $\phi_s\to -i\kappa{\rm Im}(\phi^j\overline{\phi_s})\phi_j$ is   a non-negative operator.
\end{Remark}

\begin{Lemma}\label{decay}
Let $Q$ be a holomorphic map or an anti-holomorphic map from $\Bbb H^2$ to $\mathcal{N}$ with compact image.  Fix any orthonormal frame  $\{e^{\infty}_l,Je^{\infty}_{l}\}^n_{l=1}$ for $Q^*T\mathcal{N}$. Then we have
\begin{align}
&|A^{\infty}|\lesssim e^{-r}, |\phi^{\infty}|\lesssim e^{-r}\label{fr1}\\
&|\nabla^j A^{\infty}|\lesssim  e^{-(j+1)r}, \mbox{ }|\nabla^{j} \phi^{\infty}|\lesssim  e^{-(j+1) r}\label{fr2}
\end{align}
where $r$ denotes the radial variable in polar coordinates for $\Bbb H^2$.
\end{Lemma}
\begin{proof}
It is convenient to work with Poincare disk model for $\Bbb H^2:=\{z\in\Bbb C: |z|<1\}$ equipped with metric $4(1-|z|^2)^{-2}(dxdx+dydy)$.  Then there holds $r=\ln(\frac{1+|z|}{1-|z|})$. Now, (\ref{fr1}), (\ref{fr2}) follow by  direct calculations, see \cite{Li1} for instance.
\end{proof}

\begin{Lemma}\label{heat}
Let $Q:\Bbb H^2\to \mathcal{N}$ be a  harmonic map  in Theorem 1.1 or  Theorem 1.2. Let $\delta>0,\sigma\ge 1$.  Assume that $v_0:\Bbb H^2\to \mathcal{N}$ satisfies
\begin{align*}
\|v_0-Q\|_{H^{\sigma+2\delta}} \le \epsilon\ll 1.
\end{align*}
Then $v_0$ evolves to a global solution  $v:\Bbb R^+\times\Bbb H^2\to \mathcal{N}$  of  heat flow, and  $v(s,x)$ converges to $Q(x)$ uniformly on $\Bbb H^2$.
Moreover, for $\gamma\ge 0$ there holds
\begin{align*}
\sup_{s>0}\omega_{\frac{1}{2}(\gamma-\sigma)-\delta}(s)\|\phi_s\|_{H^{\gamma}_x} \lesssim \epsilon.
\end{align*}
\end{Lemma}
\begin{proof}
The proof is based on the extrinsic formulation of heat flows and the caloric gauge. We recommend the reader to read the nice presentation of [Sec. 4, \cite{LLOS2}] where the results for $\sigma=1$ were established.  Small modifications of arguments of \cite{LLOS2} yield the desired result.
\end{proof}

\begin{Lemma}\label{heat2}
Let $v_0:\Bbb H^2\to \mathcal{N}$ be a  map such that
\begin{align*}
\|dv_0\|_{L^2_x\cap L^{\infty}_x}+\|\nabla dv_0\|_{L^{2}_x} \le K.
\end{align*}
Then if $\mathcal{N}$ is negatively curved or $K$ is small,  $v_0$ evolves to a global solution  $v:\Bbb R^+\times\Bbb H^2\to \mathcal{N}$   to heat flow, and  $v(s,x)$
satisfies
\begin{align*}
\|\nabla d v\|_{L^2_x}+\min(s^{\frac{1}{2}},1)\|\nabla^2 d v\|_{L^2_x} +\min(s,1)\|\nabla^2 d v\|_{L^{\infty}_x} +\min(s^{\frac{3}{2}},1)\|\nabla^3 d v\|_{L^{\infty}_x}  &\lesssim_{K} 1  \\
\sup_{s>0}\zeta_{\frac{1}{2}}(s)\|\partial_s v\|_{L^{\infty}_x} +
\sup_{s>0}\zeta_{1}(s)\|\nabla \partial_s v\|_{L^{\infty}_x} +
\sup_{s>0}\zeta_{\frac{3}{2}}(s)\|\nabla^2\partial_s v\|_{ L^{\infty}_x}+\sup_{s>0}\zeta_{2}(s)\|\nabla^3\partial_s v\|_{ L^{\infty}_x} &\lesssim_{K} 1.
\end{align*}
where $\zeta_{\gamma}(s):={\bf 1}_{s\in (0,1]}  s^{\gamma}+ {\bf 1}_{s\ge 1} e^{\rho_0 s}$, $\rho_0>0$.
\end{Lemma}
\begin{proof}
If $K$ is small, the proof follows by Poincare inequality, Harnack inequality  and  Bochner-Weitzenbock formulas. If $K$ is not small, we assume $\mathcal{N}$ is negatively curved. Then the   Bochner-Weitzenbock can be improved, and  the desired results follow by the same arguments as our previous works \cite{Li1,Li3}.
\end{proof}

\section{Appendix C}

In this paper, we prove Lemma 7.1. Let's first recall background materials on harmonic analysis on $\Bbb H^2$.
We refer to Helgason \cite{Hel}   for geometry analysis  and harmonic analysis on hyperbolic spaces $\Bbb H^n$. The notations adopted here exactly coincide with  Koornwinder \cite{Ko}.  $\Bbb H^2$ can be realized as the symmetric  space $G/K$, where $G = SO_{0}(1,2)$,
and $K = SO(2)$. In geodesic polar coordinates on $H^2$, the  Laplace-Beltrami operator is given by
$$
\Delta:=\partial^2_r+\coth r\partial_r+\sinh^{-2} r \partial^2_{\theta}.$$
The spherical functions $\psi_{\lambda}$ on $\Bbb H^2$  are normalized radial eigenfunctions of $\Delta$:
$$
\Delta\psi_{\lambda}=-(\lambda^2+\frac{1}{4})\psi_{\lambda}, \mbox{  }\psi_{\lambda}(0)=1.$$ And they can be represented as
\begin{align}
\psi_{\lambda}&=\int_{K}e^{-(\frac{1}{2}+i\lambda)\mathrm{H} (a_{-r}k)} dk\label{P81}\\
&=const. \int^{r}_{-r} (\cosh r-\cosh s)^{-\frac{1}{2}}e^{-i\lambda s} ds.\nonumber
\end{align}
We also have the Harish-Chandra expansion of $\psi_{\lambda}$, i.e.
\begin{align*}
\psi_{\lambda}&=\mathbf{c}(\lambda)\Phi_{\lambda}(r)+\mathbf{c}(-\lambda)\Phi_{-\lambda}(r),\mbox{ }\mbox{ } \mbox{ } r>0, \lambda\in \Bbb C\setminus\Bbb Z,
\end{align*}
where $\mathbf{c}(\lambda)$ denotes the Harish-Chandra c-function given  by
\begin{align}\label{j3}
\mathbf{c}(\lambda)=const.\frac{\Gamma(i \lambda)}{\Gamma(i\lambda+\frac{1}{2})},
\end{align}
and $\Phi_{\lambda}(r)$ is
\begin{align}\label{j2}
\Phi_{\lambda}(r)=const.(\sinh r)^{-\frac{1}{2}}e^{i\lambda r}\sum^{\infty}_{j=0}\Gamma_{j}(\lambda)e^{-2j r},
\end{align}
with $\{\Gamma_{j}\}$ defined by the recurrence formula
\begin{align*}
\Gamma_{0}(\lambda)=1, \mbox{  } \Gamma_{j}(\lambda)=-\frac{1}{4k(k-i\lambda)}\sum^{j-1}_{l=0}(k-l)\Gamma_{l}(\lambda).
\end{align*}
For $\{\Gamma_{j}\}$,
\cite{APV} proved that there exists some $\nu>0$
\begin{align}\label{j1}
|\partial^{l}_{\lambda}\Gamma_{j} (\lambda)|\le C_{l}k^{\nu}(1+|\lambda|)^{-l-1},\mbox{ } \lambda\in \Bbb R.
\end{align}

For reader's convenience, we recall the following lemma of \cite{APV} whose proof is based on the Kunze-Stein phenomenon.
\begin{Lemma}[Lemma 5.1,\cite{APV}] \label{hexie}
For any radial function $h$ on $\Bbb H^2$, any $2\le s,k<\infty$ and $g\in L^{k'}(\Bbb H^2)$, there holds
\begin{align*}
\| {g * h} \|_{{L^s}} \lesssim \left\| g \right\|_{L^{k'}}\left\{ \int_0^\infty (\psi _0{(r)})^{S}|h(r)|^P  \sinh rdr \right\}^{1/P},
\end{align*}
where $S= \frac{2\min \{s,k\} }{{s + k}}$, $P= \frac{{sk}}{{k + s}}$, and $\psi _0$ is the spherical function defined above.
\end{Lemma}

Let $\chi_0(\lambda)$ be a smooth function defined on $\Bbb R$ such that $\chi_{0}(\lambda)$ vanishes for $|\lambda|\ge 4$, and $\chi_{0}(\lambda)$ equals to 1 for $|\lambda|\le 2$. Let $\chi_{\infty}(\lambda)=1-\chi_0(\lambda).$  Define the kernel of  low frequency and high frequency Schr\"odingner propagators as follows:
\begin{align}
w^{low,\sigma}_t(r)&=\int^{\infty}_{0}\chi_0(\lambda)(\lambda^2+\frac{1}{4})^{\sigma}|c(\lambda)|^{-2}\psi_{\lambda}(r)e^{it(\lambda^2+\frac{1}{4})}d\lambda\label{P83}\\
w^{high,\sigma}_t(r)&=\int^{\infty}_{0}\chi_{\infty}(\lambda)(\lambda^2+\frac{1}{4})^{\sigma}|c(\lambda)|^{-2}\psi_{\lambda}(r)e^{it(\lambda^2+\frac{1}{4})}d\lambda.\label{P82}
\end{align}
where the integral (\ref{P82}) is understood in the sense of oscillatory integrals.

\begin{Lemma}\label{5657}
Let $\sigma\in\Bbb C$.
If $r\in (0,\frac{1}{2}|t|)$, then for any $|t|\ge 1$, $N\ge 1$
\begin{align*}
|w^{\infty}_t(r)|\lesssim_{N} |t|^{-N}\psi_{0}(r).
\end{align*}
If $r\in  [\frac{1}{2}|t|,\infty)$, then for any $|t|\ge 1$,  $N\ge 1$
\begin{align*}
|w^{high,\sigma}_t(r)|\lesssim_{N} |t|^{-N}e^{-\frac{1}{2}r}(1+r)^{N}.
\end{align*}
\end{Lemma}
\begin{proof}
{\bf Case 1.  $r\in (0,\frac{1}{2}|t|)$, $|t|\ge 1$.}
By (\ref{P81}), $w^{\infty}_{t}(r)$ defined by  (\ref{P82}) now reads as
\begin{align}\label{Fccng}
w^{high,\sigma}_t(r)&=\int_{K}e^{-\frac{1}{2}\mathrm{H}(a_{-r}k)}dk\int^{\infty}_0\chi_{\infty}(\lambda)b(\lambda)e^{it(\lambda^2+\frac{1}{4})-i{\mathrm H}(a_{-r}k)\lambda}d\lambda.
\end{align}
where $b(\lambda)$ is
\begin{align*}
b(\lambda)=\chi_{\infty}(\lambda)(\lambda^2+\frac{1}{4})^{\sigma}|\mathbf{c}(\lambda)|^{-2}.
\end{align*}
If  $r\in (0,\frac{1}{2}|t|)$, then one has
\begin{align*}
|\partial_{\lambda}[ t(\lambda^2+\frac{1}{4})- {\mathrm H}(a_{-r}k)\lambda]|\ge 2|t|-| {\mathrm H}(a_{-r}k)|\ge |t|.
\end{align*}
Then by stationary phrase method, we get for any $N\ge  1$, there exists $C_{N}>0$ such that
\begin{align*}
|\int^{\infty}_0 b(\lambda)e^{it(\lambda^2+\frac{1}{4})-i{\mathrm H}(a_{-r}k)\lambda}d\lambda|\le C_{N}|t|^{-N}.
\end{align*}
Thus (\ref{Fccng}) gives
\begin{align*}
|w^{high,\sigma}_t(r)|\le C_{N} |t|^{-N} \int_{K}e^{-\frac{1}{2}\mathrm{H}(a_{-r}k)}dk\lesssim C_{N} |t|^{-N}\psi_0(r).
\end{align*}

{\bf Case 2.  $r\in [\frac{1}{2}|t|,\infty)$, $|t|\ge 1$.}
In this case, we directly apply  (\ref{P82}). By the fact
$$|\partial_{\lambda}[t(\lambda^2+\frac{1}{4})]\ge  |\lambda t|\ge |t|$$
and integration by parts for $L_{\sigma}:= [\max(0,{\rm Re}(\sigma))]+4$ times, we obtain
\begin{align*}
|w^{high,\sigma}_t(r)|&\lesssim t^{-L_{\sigma}}\int^{\infty}_{0} \sum_{0\le j\le L_{\sigma}}{\lambda^{-L_{\sigma}-j}}
\left|\partial^{L_{\sigma}-j}_{\lambda}[\chi_{\infty}(\lambda)(\lambda^2+\frac{1}{4})^{\frac{1}{4}}|\mathbf{c}(\lambda)|^{-2}\psi_{\lambda}(r)]\right| d\lambda.
\end{align*}
Hence, applying the point-wise estimates in (\ref{j1}),  (\ref{j2}),   (\ref{j3}) yields
\begin{align*}
|w^{high,\sigma}_t(r)|&\lesssim  {t^{-L_{\sigma}}}   e^{-\frac{1}{2}r}(1+r)^{L_{\sigma}}.
\end{align*}
The same arguments show that applying  integration by parts for $N\ge L_{\sigma}$ times gives
\begin{align*}
|w^{\infty}_t(r)| \le C_{N} {|t|^{-N}} e^{-\frac{1}{2}r}(1+r)^{N}.
\end{align*}
\end{proof}

\begin{Lemma}\label{5656}
Let $\sigma\in \Bbb C$.
If $r\in (0,\frac{1}{2}|t|)$, then for any $|t|\ge 1$,
\begin{align}\label{GG}
|w^{low,\sigma}_t(r)|\lesssim  |t|^{-\frac{1}{2}}\psi_{0}(r).
\end{align}
If $r\in  [\frac{1}{2}|t|,\infty)$, then for any $t\ge 1$,
\begin{align*}
|w^{low,\sigma}_t(r)|\lesssim  |t|^{-\frac{1}{2}} e^{-\frac{1}{2}r}(1+r)^{2}.
\end{align*}
\end{Lemma}
\begin{proof}
{\bf Case 1.  $r\in (0,\frac{1}{2}|t|)$, $|t|\ge 1$.}
By (\ref{P81}), $w^{0}_{t}(r)$ defined by  (\ref{P82}) now reads as
\begin{align}\label{Fccng2}
w^{low,\sigma}_t(r)&=\int_{K}e^{-\frac{1}{2}\mathrm{H}(a_{-r}k)}dk\int^{\infty}_0\chi_{0}(\lambda)\tilde{b}(\lambda)e^{it(\lambda^2+\frac{1}{4})-i{\mathrm H}(a_{-r}k)\lambda}d\lambda.
\end{align}
where $\tilde{b}(\lambda)$ is
\begin{align*}
\tilde{b}(\lambda)=\chi_{0}(\lambda)(\lambda^2+\frac{1}{4})^{\sigma}|\mathbf{c}(\lambda)|^{-2}
\end{align*}
Since one has
\begin{align*}
|\partial^2_{\lambda}[ t(\lambda^2+\frac{1}{4})- {\mathrm H}(a_{-r}k)\lambda]|\ge 2|t|,
\end{align*}
by stationary phrase method, for  $r\in (0,\frac{1}{2}|t|)$  we get
\begin{align*}
|\int^{\infty}_0 \tilde{b}(\lambda)e^{it(\lambda^2+\frac{1}{4})-i{\mathrm H}(a_{-r}k)\lambda}d\lambda|\lesssim t^{-\frac{1}{2}}.
\end{align*}
Thus (\ref{Fccng2}) shows
\begin{align*}
|w^{low,\sigma}_t(r)|\lesssim t^{-\frac{1}{2}} \int_{K}e^{-\frac{1}{2}\mathrm{H}(a_{-r}k)}dk\lesssim t^{-\frac{1}{2}}\psi_0(r).
\end{align*}

{\bf Case 2.  $r\in [\frac{1}{2}|t|,\infty)$, $|t|\ge 1$.}
In this case, we  apply  (\ref{P82}) to get
\begin{align}\label{Fccng3}
w^{low,\sigma}_t(r)&=\int^{\infty}_0\chi_{0}(\lambda)(\lambda^2+\frac{1}{4})^{\sigma}e^{it(\lambda^2+\frac{1}{4})}(I^0_++I^0_{-})d\lambda.
\end{align}
where
\begin{align*}
I^0_{\pm}&:=c(\lambda) e^{\pm i\lambda r}(\sinh r)^{-\frac{1}{2}}(\sum^{\infty}_{j=0}\Gamma_{j}(\pm \lambda)e^{-kr}).
\end{align*}
Using the point-wise estimates in (\ref{j1}),  (\ref{j2}),   (\ref{j3}) and the fact
$$|\partial^2_{\lambda}[t(\lambda^2+\frac{1}{4})]\ge  2| t|$$
we obtain by stationary phrase method that for $r$ in Case 2, the RHS of (\ref{Fccng3}) is bounded  as
\begin{align*}
|w^{low,\sigma}_t(r)|\lesssim t^{-\frac{1}{2}}e^{-\frac{1}{2}r}(1+r)^2.
\end{align*}
\end{proof}

The bound (\ref{GG}) needs to be refined.
\begin{Lemma}\label{566}
Let $\sigma\in \Bbb R$, $\gamma\in (1,\frac{3}{2})$, $t\ge 1$. For $p,q\in (2,\infty)$,   there holds
\begin{align*}
\|w^{low,\sigma}_t(r)*f\|_{L^{p}_x}\lesssim  |t|^{-\gamma} \|f\|_{L^{q'}_x}.
\end{align*}
\end{Lemma}
\begin{proof}
{\bf Step 1.} For $|t|\ge 1$, $p,q\in (2,\infty]$, [Theorem 3.4, \cite{AP}] implies
\begin{align}\label{M11}
\|s_t(r)*f\|_{L^p_x}\lesssim  |t|^{-\frac{3}{2}}\|f\|_{L^{q'}_x}
\end{align}
where $s_t(r)$ equals to $w^{high,0}+w^{low,0}$ in our notations.
And one infer from Lemma \ref{hexie} and Lemma \ref{5657} that for any $\sigma_1\in \Bbb C$,  $p,q\in (2,\infty)$, $|t|\ge 1$, $N\ge 1$,
\begin{align}\label{M10}
\|w^{high,\sigma_1}_t(r)*f\|_{L^{p}_x}\lesssim_{N} |t|^{-N} \|f\|_{L^{q'}_x}.
\end{align}
Therefore, (\ref{M11}) and (\ref{M10}) give
\begin{align}\label{Hjj2}
\|w^{low,0}_t(r)*f\|_{L^{p}_x}\lesssim |t|^{-\frac{3}{2}} \|f\|_{L^{q'}_x},
\end{align}
for  $p,q\in (2,\infty)$, $|t|\ge 1$.
Meanwhile,  Lemma \ref{5656} gives for any $\sigma'\in \Bbb C$,  $p,q\in (2,\infty)$, $|t|\ge 1$,
\begin{align}\label{Hjj1}
\|w^{low,\sigma'}_t(r)*f\|_{L^{p}_x}\lesssim  |t|^{-\frac{1}{2}} \|f\|_{L^{q'}_x}.
\end{align}

Given $\sigma\in \Bbb R$, take $\sigma'$ to be
$$\sigma'=\frac{2\sigma}{3-2\gamma}
$$
Then by Sobolev interpolation,  we obtain from  (\ref{Hjj2}), (\ref{Hjj1}) that
\begin{align*}
\|w^{low,\sigma}_t(r)*f\|_{L^{p}_x}\lesssim_{N} \|w^{low,0}_t(r)*f\|^{\gamma-\frac{1}{2}}_{L^{p}_x} \|w^{low,\sigma'}_t(r)*f\|^{\frac{3}{2}-\gamma}_{L^p_x}\lesssim  |t|^{-\gamma} \|f\|_{L^{q'}_x}.
\end{align*}

\end{proof}

Now, Lemma 7.1 follows by Lemma \ref{566}, (\ref{M10}), Strichartz estimates of \cite{AP}, Kato smoothing estimates of Kaizuka \cite{K} and Keel-Tao's bilinear arguments.
{\bf Proof of Lemma 7.1:}
\begin{proof}
{\bf Step 1. Non-endpoint Results.}
First, we have a non-endpoint result, i.e.,
\begin{align}\label{ppp1}
\|  (-\Delta)^{\frac{1}{4}}u\|_{L_t^pL_x^q}\le \|  e^{\alpha' r} \|_{L_t^2L_x^2},
\end{align}
where $(p,q)$ is an admissible pair and $p>2$. The proof of (\ref{ppp1}) follows from the Christ-Kiselev lemma, Strichartz estimates  and the dual form Kato's smoothing effects for $e^{it\Delta}$.\\
{\bf Step 2. Bilinear Argument for Endpoint.}
Along with our previous work \cite{Li2}, it suffices to prove the following claim \\
{\bf Claim B.}
There exists a constant $\beta(q)>0$ such that for all $k,j\in \Bbb Z$
\begin{align}
& \int_{\Bbb R}  \int_{\Bbb H^2} \int_{t - {2^j} \le s \le t -  2^{j- 1}} G (t) e^{i(t - s){\Delta}}(-\Delta )^{-\frac{1}{4}}F (s)dsdxdt \nonumber \\
&\lesssim 2^{-\beta|j|}\| e ^{ \alpha' r }F  \|_{L_t^2L_x^2}\| G \|_{L_t^2L_x^{q'}}.\label{jj1}
\end{align}
if $F(s)$ and $G(t)$ are supported on a time interval of size $2^j$ on $\{(t,s):{t - {2^j} \le s \le t - {2^{j- 1}}}\}$.

{\bf Step 2.1. Sum of Negative $j$.}
 For $j\le0$, $q\in (2,\infty)$, choose $m\in (2,\infty)$ to be slightly larger than 2, then  H\"older and Strichartz estimates  give for $\frac{1}{m}\ge \frac{1}{2}-\frac{1}{q}$ (then $(m,q)$ is an admissible pair), the LHS of (\ref{jj1}) is bounded by
\begin{align*}
&\| {\int_{t - {2^j} \le s \le t - {2^{j - 1}}} {{e^{i(t - s)\Delta}}(-\Delta)^{ \frac{1}{4}}F(s)} ds}  \|_{L_t^{m}L_x^q} \left\| {G(t)} \right\|_{L_t^{m'}L_x^{q'}}
 \nonumber\\ &\lesssim { \|e^{\alpha' r} F(t)  \|_{L_t^2L_x^2}}{ \| {G(t)}  |_{L_t^{m'}L_x^{q'}}}
  \lesssim { \| e^{\alpha' r}F(t) \|_{L_t^2L_x^2}}{ \| {G(t)}  \|_{L_t^2L_x^{q'}}}{2^{j(\frac{1}{m'} - \frac{1}{2})}}.
\end{align*}
where we used the time support of $G(t)$ is of size $2^j$ in the last line.  Therefore, (\ref{jj1}) is done when  $j\le 0 $.

{\bf Step 2.2.  High Frequency and Low Frequency for Positive  $j$.} Decompose the LHS of  (\ref{jj1}) into
\begin{align}
& \int_{\Bbb R}  \int_{\Bbb H^2} \int_{t - {2^j} \le s \le t -  2^{j- 1}  } G (t)   w^{high,\frac{1}{4}}_{t-s}*{F(s)} dsdxdt  \label{jj22}\\
&+ \int_{\Bbb R}  \int_{\Bbb H^2} \int_{t - {2^j} \le s \le t -  2^{j- 1}  } G (t)   w^{low,\frac{1}{4}}_{t-s}*{F(s)}  dsdxdt.   \label{jj3}
\end{align}
Then H\"older inequalities give for any $q\in(2,\infty)$,  the (\ref{jj22}) term is bounded by
\begin{align*}
 &\| \int_{t - {2^j} \le s \le t - {2^{j - 1}}}   w^{high,\frac{1}{4}}_{t-s}*{F(s)} ds \|_{{L^{\infty}_tL^{q}_{x}}}\|G\|_{L^1_tL^{q'}_x}
\lesssim_{N} 2^{-Nj}  \|  F  \|_{L_t^1L^{m'}_x}  \| {G(s)}  \|_{L^1_tL^{q'}_x}\\
&\lesssim_{N} 2^{-Nj}2^{j}  \|  F  \|_{L^2_tL^{m'}_x}  \| {G(s)} \|_{L^2_tL^{q'}_x}\lesssim_{N} 2^{-Nj}2^{j}  \|  e^{\alpha' r}F  \|_{L^2_tL^{2}_x}  \| {G(s)} \|_{L^2_tL^{q'}_x}
\end{align*}
 where we chose some $m>2$  such that $e^{-\alpha' r}\in L^{\frac{2m}{m-2}}$, and  applied  (\ref{M10}). This is admissible for (\ref{jj1}) by choosing large$N$.
Similarly,  by H\"older inequalities and Lemma \ref{566} we also see  for any $q\in(2,\infty)$,  the (\ref{jj3}) term is dominated by
\begin{align*}
 &\| \int_{t - {2^j} \le s \le t - {2^{j - 1}}}   w^{low,\frac{1}{4}}_{t-s}*{F(s)} ds \|_{{L^{\infty}_tL^{q}_x}}\|G\|_{L^1_tL^{q'}_x}
\lesssim  2^{-\gamma j}  \|  F  \|_{L_t^1L^{m'}_x}  \| {G(s)}  \|_{L^1_tL^{q'}_x}\\
&\lesssim  2^{-\gamma j}2^{j}  \|  F  \|_{L^2_tL^{m'}_x}  \| {G(s)} \|_{L^2_tL^{q'}_x}\lesssim  2^{-\gamma j}2^{j}  \|  e^{\alpha' r}F  \|_{L^2_tL^{2}_x}  \| {G(s)} \|_{L^2_tL^{q'}_x}
\end{align*}
 where  we again chose some $m>2$  such that $e^{-\alpha' r}\in L^{\frac{2m}{m-2}}$. This is admissible for (\ref{jj1}) by choosing $\gamma>1$.
\end{proof}

\end{document}